\documentclass[a4,12pt]{amsart}
\usepackage{amssymb,amstext,amsmath,amscd,amsthm,amsfonts,enumerate,latexsym}
\usepackage{color}
\usepackage[dvipdfmx]{graphicx}

\usepackage[dvipdfmx]{hyperref}

\usepackage[all]{xy}

\theoremstyle{plain}
\newtheorem{thm}{Theorem}[section]

\newtheorem*{thm*}{Theorem}
\newtheorem*{cor*}{Corollary}

\newtheorem{prop}[thm]{Proposition}

\newtheorem{lem}[thm]{Lemma}
\newtheorem{cor}[thm]{Corollary}

\newtheorem*{claim*}{Claim}

\theoremstyle{definition}
\newtheorem{defn}[thm]{Definition}

\newtheorem{ex}[thm]{Example}

\newtheorem{setting}[thm]{Setting}
\newtheorem{notation}[thm]{Notation}

\theoremstyle{remark}
\newtheorem{rem}[thm]{Remark}

\numberwithin{equation}{thm}

\newtheorem*{ac}{Acknowledgments}

\def\Min{\operatorname{Min}}

\def\Im{\operatorname{Im}}

\def\Ker{\operatorname{Ker}}

\def\Hom{\operatorname{Hom}}

\def\Max{\operatorname{Max}}
\def\Assh{\operatorname{Assh}}
\def\End{\mathrm{End}}

\def\m{\mathfrak m}
\def\n{\mathfrak n}

\def\p{\mathfrak p}

\newcommand{\Aut}{\mathrm{Aut}}

\newcommand{\rme}{\mathrm{e}}

\newcommand{\rmH}{\mathrm{H}}

\newcommand{\rmK}{\mathrm{K}}

\newcommand{\rmQ}{\mathrm{Q}}

\newcommand{\calF}{\mathcal{F}}

\newcommand{\calS}{\mathcal{S}}

\newcommand{\calY}{\mathcal{Y}}

\newcommand{\fka}{\mathfrak{a}}

\newcommand{\fkm}{\mathfrak{m}}
\newcommand{\fkn}{\mathfrak{n}}

\newcommand{\fkp}{\mathfrak{p}}
\newcommand{\fkq}{\mathfrak{q}}

\newcommand{\mapright}[1]{%
\smash{\mathop{%
\hbox to 1cm{\rightarrowfill}}\limits^{#1}}}

\newcommand{\mapleft}[1]{%
\smash{\mathop{%
\hbox to 1cm{\leftarrowfill}}\limits_{#1}}}

\def\depth{\operatorname{depth}}

\def\ch{\operatorname{ch}}

\def\GL{\operatorname{GL}}
\def\Ass{\operatorname{Ass}}

\def\height{\mathrm{ht}}
\def\Spec{\operatorname{Spec}}

\def\red{\operatorname{red}}

\title[On the ubiquity of Arf rings]{On the ubiquity of Arf rings}

\author[E. Celikbas]{Ela Celikbas}
\address{School of Mathematical and Data Sciences, West Virginia University,
Morgantown, WV 26506 U.S.A}
\email{ela.celikbas@math.wvu.edu}

\author[O. Celikbas]{Olgur Celikbas}
\address{School of Mathematical and Data Sciences, West Virginia University,
Morgantown, WV 26506 U.S.A}
\email{olgur.celikbas@math.wvu.edu}

\author[C. Ciuperc\u{a}]{C\u{a}t\u{a}lin Ciuperc\u{a}}
\address{Department of Mathematics 2750\\ North Dakota State University\\PO BOX 6050\\ Fargo, ND 58108-6050\\ U.S.A}
\email{catalin.ciuperca@ndsu.edu}

\author[Naoki Endo]{Naoki Endo}
\address{School of Political Science and Economics, Meiji University, 1-9-1 Eifuku, Suginami-ku, Tokyo 168-8555, Japan}
\email{endo@meiji.ac.jp}
\urladdr{https://www.isc.meiji.ac.jp/~endo/}

\author[S. Goto]{Shiro Goto}
\address{Department of Mathematics, School of Science and Technology, Meiji University, 1-1-1 Higashi-mita, Tama-ku, Kawasaki 214-8571, Japan}
\email{shirogoto@gmail.com}

\author[R. Isobe]{Ryotaro Isobe}
\address{General Education Division, Oshima College, Oshima-gun, Japan}
\email{r.isobe.math@gmail.com}

\author[N. Matsuoka]{Naoyuki Matsuoka}
\address{Department of Mathematics, School of Science and Technology, Meiji University, 1-1-1 Higashi-mita, Tama-ku, Kawasaki 214-8571, Japan}
\email{naomatsu@meiji.ac.jp}


\thanks{E. Celikbas was partially supported by West Virginia University Mathematics Excellence in Research Fund. This project started when E. and O. Celikbas visited Meiji University in July 2019. They are grateful for the kind hospitality of the Department of Mathematics, Meiji University. Endo was partially supported by JSPS Grant-in-Aid for Scientific Research (C) 23K03058 and Young Scientists 20K14299 and JSPS Overseas Research Fellowships. Goto was partially supported by JSPS Grant-in-Aid for Scientific Research (C) 21K03211. Isobe was partially supported by JSPS Grant-in-Aid for Young Scientists 21K13767 and Grant-in-Aid for JSPS Fellows 20J10517. Matsuoka was partially supported by JSPS Grant-in-Aid for Scientific Research (C) 18K03227.}

\begin{document}

\maketitle

\setlength{\baselineskip} {15.2pt}

\begin{abstract}
We introduce and develop the theory of weakly Arf rings, generalizing the Arf rings originally defined by J. Lipman in 1971.
We provide characterizations of the weakly Arf rings and study their relation with the Arf rings and the strictly closed rings.
Furthermore, we give various examples of weakly Arf rings that come from idealizations, fiber products,
determinantal rings, and invariant subrings.
\end{abstract}

{\footnotesize \tableofcontents}


\section{Introduction}\label{intro}

The {\it Arf rings}, which we focus on in the present paper, trace back to the classification of certain singular points of plane curves by C. Arf in 1949 (see, e.g., \cite{Duval}, \cite{Sertoz1}, and \cite{Sertoz2}).
In 1971, J. Lipman introduced the Arf rings by extracting the essence of the rings considered by C. Arf.  Using this notion, he proved that a one-dimensional complete Noetherian local domain $A$ with an algebraically closed residue field of characteristic zero has minimal multiplicity (i.e., the embedding dimension of $A$ equals the multiplicity of $A$), provided that $A$ is {\it saturated} in the sense of  O. Zariski \cite{Zariski}. The proof of this result depends on the fact that such a saturated ring is an Arf ring. 
The notion of Arf ring is defined for Noetherian semi-local rings $A$ satisfying that every localization $A_M$ at a maximal ideal $M$ is Cohen-Macaulay and of dimension one. 
In \cite[Theorem 2.2]{L}, it is proved that a ring $A$ is Arf if and only if every integrally closed open ideal is stable; equivalently, all {\it the local rings infinitely near to $A$}, i.e., the localizations of its blow-ups, have minimal multiplicity.
As J. Lipman mentioned in \cite{L}, the defining conditions of an Arf ring are technical, but they are convenient to work with, and easy to state.

\begin{defn}[\cite{Arf, L}]\label{1.1}
Let $A$ be a Noetherian semi-local ring such that $A_M$ is a one-dimensional Cohen-Macaulay local ring for every $M \in \Max A$. Then $A$ is called {\it an Arf ring} if the following conditions hold:
\begin{enumerate}[$(1)$]
\item Every integrally closed ideal $I$ in $A$ that contains a non-zerodivisor  has {\it a principal reduction}, i.e., $I^{n+1}=aI^n$ for some $n \ge 0$ and  $a\in I$. 
\item If $x, y, z \in A$ such that $x$ is a non-zerodivisor on $A$ and $y/x, z/x \in \rmQ(A)$ are integral over $A$, then $yz/x \in A$, where $ \rmQ(A)$  denotes the total ring of fractions of $A$.
\end{enumerate}
\end{defn}

A typical example of an Arf ring is a ring with multiplicity at most two; see \cite[Example, page 664]{L}, \cite[Proposition 2.8]{CCGT}.  As the Arf property is preserved by the standard procedures in ring theory, such as completion, localization, and faithfully flat extensions, examples of Arf rings are abundant; see \cite[Corollaries 2.5 and 2.7]{L}. 
Assume for simplicity that $A$ is a one-dimensional complete Noetherian local domain with algebraically closed field. 
As noted by J. Lipman in \cite{L}, among all the Arf rings between $A$ and its integral closure $\overline{A}$, there is a smallest one, called {\it the Arf closure}. Since the Arf closure maintains the multiplicity of $A$ and commutes with quadratic transforms, i.e., blow-ups of the maximal ideal, the multiplicity sequence of $A$ along the unique maximal ideal of $\overline{A}$ is the same as that of the Arf closure. This fact leads to characterizations of Arf rings by means of the  semigroup of values, which gives rise to the notion of Arf semigroups. The value semigroup of the Arf closure appears as {\it an Arf semigroup}, i.e., a numerical semigroup $H$ satisfying the following condition:
\begin{center}
If $x, y, z \in H$ such that $x \le y$ and $x\le z$, then $y+z-x \in H$.
\end{center}

The Arf property for numerical semigroups and algorithms to compute the Arf  closure of various rings, such as the coordinate rings of curves, were studied in \cite{AS, BF, RGGB}.
Furthermore, V. Barucci and R. Fr\"oberg \cite{BF} explored the question of when the Arf semigroups are almost symmetric, and gave a characterization of them. Recently, E. Celikbas, O. Celikbas, S. Goto, and N. Taniguchi (the fourth author of this paper) generalized their results; see \cite[Theorem 1.1, Corollary 1.2]{CCGT}.

The theory of Arf rings is, so to speak, a prototype of the theory of one-dimensional Cohen-Macaulay rings, and it has evolved into mostly the theory of numerical semigroups, i.e., the theory of singularities of plane curves,  and the study of Arf closures. Having said that, the theory of Arf rings might also develop in ring theory and one still needs to pursue the work of C. Arf and J. Lipman in the literature.



The main purpose of this paper is to inspect the theory of Arf rings in detail, and to attempt to motivate the interested readers to study in this direction. More precisely, we are concerned with the question of what happens if we remove the condition $(1)$ and assume only the condition $(2)$ in Definition \ref{1.1}. We call such a ring {\it a weakly Arf ring}.

\begin{defn}
A commutative ring $A$ is said to be {\it weakly Arf}, provided that $yz/x \in A$ whenever $x,y,z\in A$ such that $x \in A$ is a non-zerodivisor on $A$ and $y/x, z/x \in \rmQ(A)$ are integral over $A$. 
\end{defn}

The difference between Arf and weakly Arf rings is determined by the existence of principal reductions of integrally closed ideals. Thus, when $(A, \m)$ is a one-dimensional Cohen-Macaulay local ring, these notions are equivalent if the residue class field $A/\m$ of $A$ is infinite, or if $A$ is analytically irreducible, i.e., the $\m$-adic completion $\widehat{A}$ of $A$ is an integral domain. In particular, these notions coincide for numerical semigroup rings. 
However, Arf and weakly Arf rings are not equivalent in general as weakly Arf rings are defined in a more general situation. Besides, as J. Lipman pointed out, there is an example of a one-dimensional Cohen-Macaulay local ring that is weakly Arf, but not Arf; see \cite[page 661]{L}. In the present paper, we delve into this example by giving two different proofs (see Example \ref{9.6} and Corollary \ref{10.5}), and construct other examples to study the theory of Arf and weakly Arf rings. 

The investigation of weakly Arf rings opens the door leading to the frontier of this new topic and gives us new knowledge of Arf rings. For example, in Sections \ref{sec5} and \ref{sec6}, we obtain Arf rings arising from idealizations and fiber products. Besides, we provide an affirmative answer for the conjecture of O. Zariski, which states the Arf closure coincides with the strict closure; see Section \ref{sec7}.
We also observe that Arf rings appear in determinantal rings and invariant subrings, see Sections  \ref{sec11} and \ref{sec12}. In addition, Theorem \ref{ch} reveals that the Arf property depends on the characteristic of the ring.

\medskip

We now state our results more precisely, explaining how this paper is organized. 
In Section \ref{sec2}, we begin with the basic properties of the algebra $A^I$ of $A$ at an ideal $I$, which will be needed throughout this paper. In contrast with Arf rings, we give a characterization of the weakly Arf rings in terms of the stability of the integral closure of a principal ideal. Section \ref{sec3} provides further properties of the weakly Arf rings, including the fact that the weakly Arf property is inherited under cyclic purity maps. In particular, we give a partial answer to the question of when the rings of invariants are weakly Arf; see Corollary \ref{3.2}.

In Section \ref{sec4}, we explore the strict closedness of rings in connection with Arf and weakly Arf rings. Theorem \ref{Zariski-Lipman} shows that, if $A$ is a Noetherian ring which satisfies Serre's $(S_2)$ condition, then $A$ is strictly closed in its integral closure $\overline{A}$ if and only if $A$ is weakly Arf and $A_P$ is an Arf ring for every $P \in\Spec A$ with $\height_AP = 1$. Hence, if we additionally assume that $\height_AM \ge 2$ for every $M \in \Max A$, the strict closedness coincides with the weakly Arf property; see Corollary \ref{4.4}. 
As an application of Theorem \ref{Zariski-Lipman}, we prove the following.

\begin{thm}[Corollary \ref{4.7}]
Let $A$ be a Noetherian integral domain. Suppose that $A$ satisfies Serre's $(S_2)$ condition and $A$ contains an infinite field. Then $A$ is  weakly Arf if and only if so is the polynomial ring $A[X_1, X_2, \ldots, X_n]$ for every $n \ge 1$. 
\end{thm}

In Sections \ref{sec5} and \ref{sec6}, we study the question of when  idealizations and fiber products are weakly Arf and Arf rings, which allows us to enrich the theory and produce concrete examples of such rings. The main results of these sections are the following.

\begin{thm}[Theorem \ref{5.2}]
Let $R$ be a Noetherian ring, and let $M$ be a finitely generated torsion-free $R$-module. Then the idealization $A=R\ltimes M$ of $M$ over $R$ is a weakly Arf ring if and only if $R$ is a weakly Arf ring and $M$ is an $\overline{R}$-module.  
\end{thm}

\begin{thm}[Theorem \ref{6.2}]
Let $(R, \m), (S, \n)$ be Noetherian local rings with a common residue class field $k=R/\m = S/\n$. Suppose that $\depth R>0$ and $\depth S > 0$. Then the fiber product $A = R\times_k S$ of $R$ and $S$ over $k$ is a weakly Arf ring if and only if so are the components $R$ and $S$.
\end{thm}

In Section \ref{sec7}, we give a modification of Arf closures, called {\it weakly Arf closures}, to investigate weakly Arf rings in accordance with our situation. We prove in Theorem \ref{4.3} that the Arf property implies the strict closedness of the ring without assuming that the ring contains a field. This generalizes \cite[Theorem 4.6]{L}, and hence the strict closure coincides with the Arf closure. This gives a complete answer for the conjecture posed by O. Zariski; see \cite[page 651]{L}. 

\begin{thm}[Corollary \ref{C3}]
Let $A$ be a Noetherian semi-local ring such that $A_M$ is a one-dimensional Cohen-Macaulay local ring for every $M \in \Max A$. Suppose that $\overline{A}$ is a finitely generated $A$-module. Then the Arf closure of $A$ coincides with the strict closure of $A$ in $\overline{A}$.
\end{thm}

As an application of weakly Arf closures, in Section \ref{sec8}, we are concerned with core subalgebras of the polynomial ring $k[t]$ over a field $k$. The class of core subalgebras includes the semigroup rings $k[H]$ of a numerical semigroup $H$. 

Section \ref{sec9} is devoted to study the weakly Arf property for the algebra $A^I$ of an ideal $I$. For a Noetherian ring $A$, $\Max \Lambda(A)$ denotes the set of all the maximal elements in the set of all proper ideals that are the integral closure of a principal ideal generated by a non-zerodivisor on $A$ with respect to inclusion. By computing the set $\Max\Lambda(A)$, we give an example of a weakly Arf ring which is not Arf; see Example \ref{9.6}. 
Notice that, if $A$ is not integrally closed, there exists $M \in \Max \Lambda(A)$ such that $\mu_A(M) \ge 2$; see Proposition \ref{9.7}. 
Here, we denote by $\mu_A(N)$ the number of elements in a minimal system of generators of an $A$-module $N$.   
By using this phenomenon, we define $A_1$ to be $\overline{A}$, if $A$ is integrally closed. Otherwise, if $A \ne \overline{A}$, we define $A_1 = A^M = \bigcup_{n \ge 0}\left[M^n:M^n\right]$, where $M \in \Max \Lambda(A)$ such that $\mu_A(M) \ge 2$. Set $A_0 = A$, and for each $n \ge 1$, define recursively $A_n = \left(A_{n-1}\right)_1$. Hence, we have a chain of rings
$$
A = A_0 \subseteq A_1 \subseteq \cdots \subseteq A_n \subseteq \cdots \subseteq \overline{A}
$$
by the algebras $A^M$, where $M \in \Max \Lambda(A)$; see Definition \ref{9.9}. The main result of Section \ref{sec9} is the following, which extends the condition ${\rm (i)} \Leftrightarrow {\rm (iii)}$ as in \cite[Theorem 2.2]{L}. 

\begin{thm}[Theorem \ref{9.11}]
Let $A$ be a Noetherian ring, and consider the following conditions: 
\begin{enumerate}[$(1)$]
\item $A$ is a weakly Arf ring.
\item For every $M \in \Max \Lambda(A)$, $M:M$ is a weakly Arf ring and $M^2 = aM$ for some $a \in M$. 
\item  For every chain $A = A_0 \subseteq A_1 \subseteq \cdots \subseteq A_n \subseteq \cdots \subseteq \overline{A}$ obtained from Definition \ref{9.9}, $A_n$ is a weakly Arf ring for every $n \ge 0$.
\item For every chain $A = A_0 \subseteq A_1 \subseteq \cdots \subseteq A_n \subseteq \cdots \subseteq \overline{A}$ obtained from Definition \ref{9.9}, and for every $n \ge 0$ and $N \in \Max \Lambda (A_n)$, $N^2 = bN$ for some $b \in N$. 
\end{enumerate}
Then the implications $(1) \Leftrightarrow (2) \Leftrightarrow (3) \Rightarrow (4)$ hold. If $\dim A=1$, or $A_{\p}$ is quasi-unmixed for every $\p \in \Spec A$, the implication $(4) \Rightarrow (1)$ holds.  
\end{thm}

In Section \ref{sec10}, we study Example \ref{9.6} in detail, determine all the integrally closed ideals that contain a non-zerodivisor, and give our second proof for this example. Furthermore, we provide other examples of weakly Arf rings that are not Arf. 
Inspired by Corollary \ref{3.2}, we note that some of the invariant subrings could be weakly Arf. In Section \ref{sec11},  we explore the weakly Arf property for the rings of invariant acting on subgroups of $\GL_2(k)$, where $k$ is a prime field of characteristic 2. Invariant subrings may appear as determinantal rings. Therefore, in view of Section \ref{sec11}, it seems worth studying when determinantal rings are weakly Arf; we do this in Section \ref{sec12}. Theorem \ref{ch} guarantees that the Arf property depends on the characteristic of the ring. 

In Section \ref{sec13}, we investigate the strict closedness and the Arf property of rings in connection with the direct summands, and prove the following.

\begin{thm}[Corollary \ref{Final}]
Let $R$ be a Noetherian semi-local ring such that $R_M$ is a one-dimensional Cohen-Macaulay local ring for every $M \in \Max R$. Suppose that $R$ is an Arf ring. Then, for every finite subgroup $G$ of $\Aut R$ such that the order of $G$ is invertible in $R$, $R^G$ is an Arf ring. 
\end{thm}

Throughout this paper, unless otherwise specified, let $A$ be an arbitrary commutative ring, let $W(A)$ be the set of non-zerodivisors on $A$, and let $\calF_A$ be the set of {\it open} ideals in $A$, i.e., the ideals of $A$ that contain a non-zerodivisor on $A$. For an ideal $I$ in $A$, we denote by $\overline{I}$ the integral closure of $I$ in $A$. We set $\Lambda(A) = \{\overline{(x)} \mid x \in W(A)\}$. For $A$-submodules $X$ and $Y$ of the total ring of fractions $\rmQ(A)$, we set $X:Y = \{a \in \rmQ(A) \mid a Y \subseteq X\}$.

For an $A$-module $M$, $\ell_A(M)$ denotes the length of $M$ and $\mu_A(M)$ stands for the number of elements of a minimal system of generators for $M$. When $A$ is a Noetherian local ring with maximal ideal $\m$, for each $\m$-primary ideal $I$ in $A$, let $\rme^0_{I}(A)$ denote the multiplicity of $A$ with respect to $I$. We also denote $\rme(A) = \rme^0_{\m}(A)$.


\section{Basic properties of weakly Arf rings}\label{sec2}

In this section we give the definitions and some basic properties of the Arf and weakly Arf rings. 
For each $I \in \calF_A$, there is a filtration of endomorphism algebras as follows:
$$
A \subseteq I:I \subseteq I^2: I^2 \subseteq \cdots \subseteq I^n:I^n \subseteq \cdots \subseteq\rmQ(A).
$$
Set
$
A^I = \bigcup_{n\geq 0}\left[I^n: I^n\right].
$
The ring $A^I$, an intermediate ring between $A$ and $\rmQ(A)$, coincides with the blow-up of $A$ at $I$ when $A$ is Noetherian and of dimension one.  For each $n > 0$, note that $I^n \in \calF_A$ and $A^I = A^{I^n}$. When $A$ is Noetherian and $I$ contains a principal reduction, $A^I$ is a module-finite extension over $A$, and $A \subseteq A^I \subseteq \overline{A}$  
where $\overline{A}$ denotes the integral closure of $A$ in $\rmQ(A)$. 
In particular, for a Noetherian semi-local ring $A$ such that $A_M$ is a one-dimensional Cohen-Macaulay local ring for every $M \in \Max A$, all the blow-ups $A^I$ of $A$ at $I \in \calF_A$ are finitely generated $A$-modules and $A \subseteq A^I \subseteq \overline{A}$, because there exists an integer $n >0$ such that $I^n$ contains a reduction; see \cite[Proof of Proposition 1.1]{L}.
Notice that, if $a \in I$ is a reduction of $I$, i.e., $I^{r+1} = a I^r$ for some $r\ge0$, then one has:
$$
A^I = A\left[\frac{I}{a}\right] = \frac{I^r}{a^r}, \text{ where } \frac{I}{a} = \left\{\frac{x}{a} ~\middle|~ x \in I\right\} \subseteq \rmQ(A).
$$
Moreover, if $a \in I$ is a reduction of $I$, then $A^I = I^n:I^n$ for every $n \ge r$. Hence, the reduction number $\red_Q(I)$ of $I$ with respect to $Q=(a)$, the minimum integer $r \ge 0$ satisfying the equality $I^{r+1} = QI^r$, is equal to the minimum integer $n \ge 0$ such that $A^I = I^n : I^n$. Therefore, $\red_Q(I)$ does not depend on the choice of a principal reduction $Q=(a)$ of $I$. 

An ideal $I \in \calF_A$ is called {\it stable} in $A$, if $A^I=I:I$, or, equivalently, $IA^I = I$. Moreover, an ideal $I \in \calF_A$ is stable if and only if $I^2 = aI$ for some $a \in I$. See \cite[Lemma 1.8]{L} for the proof of this fact.  

Before stating the definition of the weakly Arf rings, we recall the notion of Arf rings.

\begin{defn}[\cite{Arf, L}] \label{2.2}
Let $A$ be a Noetherian semi-local ring such that $A_M$ is a one-dimensional Cohen-Macaulay local ring for every $M \in \Max A$. Then $A$ is said to be {\it an Arf ring}, if the following conditions are satisfied:
\begin{enumerate}[$(1)$]
\item If $I \in \calF_A$ is an integrally closed ideal in $A$, then $I$ has a principal reduction.
\item If $x,y,z \in A$ such that $x \in W(A)$ and $y/x, z/x \in \overline{A}$, then $yz/x \in A$.
\end{enumerate}
\end{defn}

A Noetherian local ring $A$ with multiplicity at most two is Arf; see \cite[Example, page 664]{L}, \cite[Proposition 2.8]{CCGT}. 
Numerical semigroup rings provide numerous examples of Arf rings. Indeed, let $k[[t]]$ be the formal power series ring over a field $k$. For an integer $n \ge 2$, $A=k[[t^n, t^{n+1}, \ldots, t^{2n-1}]]$ is an Arf ring; see \cite[Example 4.7]{CCGT}. 
Idealizations and fiber products also produce Arf rings; see Section \ref{sec6}. For example, let $ A=k[[x,y,z]]/(x^3-yz, y^2-zx, z^2-x^2y)$, where $k[[x, y, z]]$ denotes the formal power series ring over a field $k$. Both the idealization $A\ltimes \mathfrak m \cong k[[x,y,z, u, v, w]]/I$ and the fiber product $A\times_k A\cong k[[x,y,z, u, v, w]]/J$ are Arf rings by Theorems~\ref{12.4} and \ref{6.10}, since $A$ is Arf and $\mathfrak m \overline{A}=\mathfrak m$, where $\mathfrak m$ is the maximal ideal of $A$ and 
\begin{eqnarray*}
I&=&(u,v,w)^2+(yu-xv, zu-yv, yv-xw)+(x^3-yz, y^2-zx, z^2-yx^2) \\ 
&&+ \ (x^2u-zv, zv-yw, 
x^2v-zw)\\ 
J&=&(x^3-yz, y^2-zx, z^2-yx^2)+(u^3-vw, v^2-wu, w^2-u^2v) \\ 
&&+  \ (x, y, z)\cdot (u, v, w).
\end{eqnarray*} 
Arf rings show up in many other places such as invariant subrings and determinantal rings as well; see Sections \ref{sec11} and \ref{sec12}. Furthermore, let us here mention that the Arf property depends on the characteristic of the rings; see Theorem \ref{ch}.

In what follows we present a weaker version of this concept in commutative rings that are not necessarily  Noetherian semi-local. In the rest of this section, we present several basic properties and characterizations of this new class of rings.

\begin{defn}\label{2.3}
Let $A$ be an arbitrary commutative ring. 
We say that $A$ is {\it a weakly Arf ring} if condition $(2)$ in Definition \ref{2.2} holds, i.e., if $x, y, z \in A$ such that $x \in W(A)$ and $y/x, z/x \in \overline{A}$, then $yz/x \in A$.
\end{defn}

Every Arf ring is weakly Arf. The difference between Arf and weakly Arf rings is only the existence of reductions of integrally closed ideals. Thus, for a one-dimensional Cohen-Macaulay local ring $A$ with maximal ideal $\m$, these notions are equivalent if the residue class field $A/\m$ of $A$ is infinite, or if $A$ is analytically irreducible, i.e., the $\m$-adic completion $\widehat{A}$ of $A$ is an integral domain. However, we will show in Sections 9 and 10 that weakly Arf rings are not necessarily Arf even though the local rings are Cohen-Macaulay and of dimension one; see Example \ref{9.6} and Remark \ref{10.7}. 

Every integrally closed ring is a weakly Arf ring; hence so is every ring whose total ring of fractions coincides with the ring itself (e.g., a Noetherian local ring $A$ with $\depth A=0$). 
Next we give a list of some of the examples of weakly Arf rings that follow from the results proved later; see Examples \ref{2.8b}, \ref{9.6}, \ref{9.15}, Corollaries \ref{3.2}, \ref{4.7}, \ref{5.8}, and Theorems \ref{6.10}, \ref{15.1}.

\begin{ex}
Let $k$ be a field. 
\begin{enumerate}[$(1)$]
\item Let $k[s, t]$ be the polynomial ring over $k$. Then, for each $\ell \ge 1$, $A = k[s, t^2, t^{2\ell +1}]$ is a weakly Arf ring.
\item Let $S = k[t]$ be the polynomial ring over $k$. Then, for each $\ell \ge 2$, $A = k[t^{\ell} + t^{\ell + 1}] + t^{\ell+2}S$ is a weakly Arf ring. 
\item Let $A$ be an integral domain and let $G$ be a finite subgroup of $\Aut A$. Suppose that the order of $G$ is invertible in $A$. If $A$ is a weakly Arf ring, then so is the invariant subring $A^G$. 
\item Let $A$ be a Noetherian integral domain with Serre's $(S_2)$ condition,  containing an infinite field. Then $A$ is a weakly Arf ring if and only if so is the polynomial ring $A[X_1, X_2, \ldots, X_n]$ for every $n \ge 1$. 
\item Let $S = k[t]$ be the polynomial ring over $k$ and set $R = k[t^2, t^{2\ell + 1}]~(\ell \ge 1)$. Then the idealization $A = R \ltimes S^{\oplus n}$ is a weakly Arf ring for every $n \ge 0$. 
\item Let $(R, \m)$ be a Noetherian local ring. Then $R$ is a weakly Arf ring if and only if so is the fiber product $A = R \times_{R/\m}R$. 
\item Let $k[X, Y, Z]$ be the polynomial ring over $k$. Then $A =k[X, Y, Z]/{\rm I}_2(
\begin{smallmatrix}
X & Y^2 & Z \\
0 & Z & Y
\end{smallmatrix}
)$
 is a weakly Arf ring, where  ${\rm I}_2(
M)$ denotes the ideal of $k[X, Y, Z]$ generated by $2\times 2$-minors of a matrix $M$. 
\item Let $k[[X, Y]]$ be the formal power series ring over $k$ and set $A = k[[X, Y]]/(XY(X+Y))$. If $k$ is a prime field of characteristic $2$, then $A$ is a weakly Arf ring, but not an Arf ring.
\end{enumerate}
\end{ex}

We now explore basic properties of weakly Arf rings. We begin with the following, which corresponds to the equivalence ${\rm (i)} \Leftrightarrow {\rm (ii)}$ of \cite[Theorem 2.2]{L} for Arf rings.

\begin{thm}\label{2.4}
Consider the following conditions:
\begin{enumerate}[$(1)$]
\item $A$ is a weakly Arf ring.
\item For every $I \in \Lambda(A)$, there exists $a \in I$ such that $I^2 = aI$.
\end{enumerate}
Then the implication $(1) \Rightarrow (2)$ holds and the converse holds if $A$ is Noetherian. 
\end{thm}

\begin{proof}
$(1) \Rightarrow (2)$ Let $I \in \Lambda(A)$. We write $I = \overline{(a)}$ for some $a \in W(A)$. For each $y, z \in I$,  we have $y, z \in a \overline{A}$, because $\overline{(a)} = a\overline{A} \cap A$. This shows $y/a, z/a \in \overline{A}$, and we get $yz \in a A$ as $A$ is weakly Arf. Write $yz = a w$ for some $w \in A$. Then, since 
$$
\frac{w}{a} = \frac{y}{a}\cdot \frac{z}{a} \in \overline{A},
$$
we have $w \in a \overline{A} \cap A = \overline{(a)} = I$. Therefore, $I^2 = a I$, as desired.

$(2) \Rightarrow (1)$ Suppose that $A$ is Noetherian. Let $x, y, z \in A$ such that $x \in W(A)$ and $y/x, z/x \in \overline{A}$. Set $I = \overline{(x)}$. Then, by our hypothesis,  we can choose $a \in I$ such that $I^2 = aI$, that is, $A^I = I:I$. Remember that the reduction number does not depend on the choice of principal reductions of $I$. Thus, since $A$ is Noetherian, the ideal $(x)$ is a reduction of $I$; hence $I^2 =xI$. As $y/x, z/x \in \overline{A}$, we have $y, z \in x \overline{A} \cap A = \overline{(x)} = I$. Therefore, $yz \in I^2 = xI \subseteq xA$, which completes the proof.
\end{proof}

We  now consider the question of how the weakly Arf property is inherited under localizations. For a commutative ring $A$, we set 
$$
X_1(A) = \{P \in \Spec A \mid \depth A_P \le 1\}.
$$

\begin{lem}\label{2.6}
Let $A$ be a Noetherian ring. Then one has the equality
$$
\bigcap_{P \in X_1(A)} A_P = A \quad \text{in} \ \ \rmQ(A).
$$
\end{lem}

\begin{proof}
In the ring $\rmQ(A)$, we consider $A$ as a subring of $A_P$ for every $P \in X_1(A)$; hence $A \subseteq \bigcap_{P \in X_1(A)} A_P$. Conversely, let $x \in \bigcap_{P \in X_1(A)} A_P$ and assume that $x \not\in A$. We write $x = b/a$, where $a, b \in A$ such that $a \in W(A)$. Then, since $x \not\in A$, we see that  $(a):_A b \subsetneq A$. Let $P \in \Ass_A(A/[(a):_A b])$. We then have $P \in \Ass_A(A/(a))$, because the homothety map $\widehat{b} : A/[(a):_A b] \to A/(a)$ is injective. Hence $\depth_{A_P} (A_P/aA_P) = 0$, so that $P \in X_1(A)$. Notice that  $\depth A_P = 1$, because $a \in A$ is a non-zerodivisor on $A$. In particular, $x \in A_P$. Choose $c \in A$ and $s \in A\setminus P$ such that 
$$
x = \frac{b}{a}= \frac{c}{s} \quad \text{in} \ \ A_P.
$$
Thus, there exists $u \in A\setminus P$ such that $u(bs - ac) = 0$, namely, $us \in (a):_A b \subseteq P$. Consequently, because $u \notin P$, we conclude that $s \in P$, which gives a contradiction. Hence  $\bigcap_{P \in X_1(A)} A_P = A$, as claimed.
\end{proof}

Recall that, for a Noetherian ring $A$ and $n \in \Bbb Z$, we say that {\it $A$ satisfies Serre's $(S_n)$ condition}, if $\depth A_P \ge \min \{n, \dim A_P\}$ for all $P \in \Spec A$. 
We then have the following. 

\begin{thm}\label{2.7}
Let $A$ be a Noetherian ring. Consider the following conditions:
\begin{enumerate}[$(1)$]
\item $A$ is a weakly Arf ring.
\item $S^{-1}A$ is a weakly Arf ring for every multiplicatively closed subset $S$ of $A$.
\item $A_P$ is a weakly Arf ring for every $P \in X_1(A)$.
\end{enumerate}
Then the implications $(2) \Rightarrow (3) \Rightarrow (1)$ hold. If $A$ satisfies $(S_1)$, then the implication $(1) \Rightarrow (2)$ holds. 
\end{thm}

\begin{proof}
$(2) \Rightarrow (3)$ This is obvious.

$(3) \Rightarrow (1)$ Let $x, y, z \in A$ such that $x \in W(A)$ and $y/x, z/x \in \overline{A}$. We will show that $yz \in xA$. Indeed, for each $P \in X_1(A)$, we have $yz \in xA_P$, since $A_P$ is weakly Arf and $y/x, z/x \in (\overline{A})_P = \overline{A_P}$. Therefore, by Lemma \ref{2.6}, $$yz \in \bigcap_{P \in X_1(A)}x\cdot A_P = x \cdot\bigcap_{P \in X_1(A)}A_P =xA.$$ Hence $A$ is a weakly Arf ring. 

$(1) \Rightarrow (2)$ The proof is based on \cite[Lemma 4.10, part (ii), page 679]{L}. We assume that $A$ satisfies $(S_1)$. Let $x,y,z \in S^{-1}A$ with $x \in W(S^{-1}A)$ such that $y/x, z/x \in \overline{S^{-1}A}=S^{-1} \overline{A}$. Without loss of generality, we may assume that 
$$
x=\frac{a}{1},  \ \  y=\frac{b}{1}, \  \ z=\frac{c}{1}
$$ 
where $a, b, c \in A$ and $a \in W(A)$. In fact, since $x$ is regular in $S^{-1}A$, we may choose $a$ to be a regular  element of $A$. To see this, let $K$ be the kernel of the map $A\to S^{-1}A$. Let $P_1,\ldots, P_n$ be all the associated prime ideals of $A$; by our assumptions, they are all minimal prime ideals of $A$. Note that  $(a) + K \not \subseteq P_i$ for every $1 \le i \le n$. It then follows that  there exists $k \in K$ such that $a+k$ does not belong to any of the ideals $P_i$, i.e., $a+k $ is a regular element of $A$. Since  $x=a/1=(a+k)/1$ in $S^{-1}A$, we may replace $a$ by $a+k$ and assume that $a \in W(A)$. Now, because 
$$
\frac{y}{x} = \frac{b}{a},  \ \ \frac{z}{x} = \frac{c}{a} \ \in S^{-1}\overline{A}, 
$$
there exist $s, t \in S$ such that $sy/x=sb/a \in \overline{A}$ and $tz/x=tc/a \in \overline{A}$. Since $A$ is a weakly Arf ring and $a \in W(A)$, it follows that ${(st)bc} \in aA$, and therefore 
$$
\frac{st}{1}\cdot yz = \frac{st}{1}\cdot \frac{bc}{1} \in \frac{a}{1} {\cdot}(S^{-1}A) = x(S^{-1}A)
$$
which yields $yz \in x(S^{-1}A)$. Hence $S^{-1}A$ is a weakly Arf ring.
\end{proof}

Next we give an example that illustrates Theorem \ref{2.7}. Recall, for a Noetherian local ring $A$, we denote by $\rme(A)$ the multiplicity of $A$. 

\begin{ex}\label{2.8b}
Let $B = k[s,t]$ be the polynomial ring over a field $k$, and let $A = k[s, t^2, t^{2 \ell + 1}]~(\ell \ge1)$. Then $A$ is a two-dimensional Cohen-Macaulay ring  and $A$ is a weakly Arf ring.
\end{ex}


\begin{proof}
Since $B$ is a module-finite extension over the ring $A$, we have $\dim A = 2$. In addition, $A$ is a Cohen-Macaulay ring, since $A \cong k[X, Y, Z]/(X^{2\ell + 1} - Y^2)$. To show that $A$ is weakly Arf, by Theorem \ref{2.7}, it is enough to show that $A_P$ is weakly Arf for every $P \in X_1(A)$. In fact, for each $P \in X_1(A)$, the definition ensures that $\depth A_P \le 1$. We may assume that $\depth A_P = 1$, as the local ring $A_P$ with $\depth A_P=0$ is always weakly Arf.  Note that $A_P$ is a one-dimensional Cohen-Macaulay local ring with $\rme(A_P) \le 2$ and $B_P = \overline{A_P}$ is a DVR. In particular, $A_P$ is an Arf ring by \cite[Proposition 2.8]{CCGT}. Hence, $A_P$ is a weakly Arf ring, as well.
\end{proof}



The following corollary is a direct consequence of Theorem \ref{2.7}; see also \cite[Proposition 2.8]{CCGT}.

\begin{cor}\label{2.9}
Let $A$ be a Noetherian ring. Suppose that $A$ satisfies $(S_2)$ and $\rme(A_P) \le 2$ for every $P \in \Spec A$ with $\height_A P \le 1$. Then $A$ is a weakly Arf ring.
\end{cor}

\begin{proof}
For every $P \in X_1(A)$, $A_P$ is Cohen-Macaulay, as $A$ satisfies $(S_2)$. We may assume $\dim A_P = 1$. The assumption shows that $A_P$ is an Arf ring by \cite[Proposition 2.8]{CCGT}. Therefore, $A$ is a weakly Arf ring.
\end{proof}

Closing this section, let us discuss the weakly Arf property in terms of the algebra $A^{I}$ of the ideal $I \in \Lambda(A)$. To do this, we need some auxiliaries. We will revisit the algebras $A^I$ in Section \ref{sec9}.

\begin{lem}\label{2.8a}
Suppose that the ideal $(a)$ is integrally closed for every $a \in W(A)$. Then $A$ is integrally closed. 
\end{lem}

\begin{proof}
Let $x \in \overline{A}$ and write $x = b/a$ with $ b \in A$ and $a \in W(A)$. Since $b \in a \overline{A} \cap A = \overline{(a)} = (a)$, we have $x \in A$. Hence $A$ is integrally closed. 
\end{proof}

\begin{prop}\label{2.9}
Let $A$ be a Noetherian ring and assume that $A$ is a weakly Arf ring. For each $a \in W(A)$, we set $I = \overline{(a)}$ and $B = A^I$. Then the following assertions hold.
\begin{enumerate}[$(1)$]
\item $B=A^I$ is Noetherian and weakly Arf. 
\item $B=A$ if and only if $I = (a)$. 
\end{enumerate}
\end{prop}

\begin{proof}
First notice that $B$ is Noetherian since $B$ is a module-finite extension over $A$. 
As $I \in \Lambda(A)$, we have $I^2 = a I$ by Theorem \ref{2.4}. Hence $B = A^I = I:I$. Moreover, $B = A\left[\frac{I}{a}\right] = \frac{I}{a}$ in $\rmQ(A)$. Thus we get $B = A$ if and only if $I = (a)$ which proves the assertion $(2)$. Although the assertion $(1)$ immediately follows from Proposition \ref{9.1}, we now give a brief proof of this statement here. Let $\alpha, \beta, \gamma \in B$ such that $\alpha \in W(B)$ and $\beta/\alpha, \gamma/\alpha \in \overline{B}$. Since $B = \frac{I}{a}$, we can write 
$$
\alpha = \frac{x}{a}, \ \ \beta = \frac{y}{a}, \ \  \gamma = \frac{z}{a}
$$ 
together with $x, y, z \in I$. Note that $x \in W(A)$. Since $\beta/\alpha = y/x, \gamma/\alpha=z/x$ in $\overline{B} = \overline{A}$ and $A$ is weakly Arf, we get $yz \in xA$. We set $J = \overline{(x)}$. Then $y, z \in  x\overline{A} \cap A = J$ and $J^2 = xJ$, because $J \in \Lambda(A)$ and $\red_{(x)}(J)$ does not depend on the choice of reductions $(x)$ of $J$. Hence there exists $i \in J$ such that $yz = x i$. Therefore
$$
\frac{\beta\gamma}{\alpha} = \frac{yz}{ax} = \frac{i}{a} \in \frac{J}{a} \subseteq \frac{\overline{I}}{a} = \frac{I}{a} = B
$$
which shows $B$ is a weakly Arf ring.  
\end{proof}

We are now ready to prove the last result of this section. For an $A$-module $M$, we denote by $\ell_A(M)$ the length of $M$ as an $A$-module. 

\begin{thm}\label{2.10}
Let $A$ be a Noetherian ring. Suppose that $\overline{A}$ is a finitely generated $A$-module and $0 < \ell_A(\overline{A}/A) < \infty$.  If $A$ is a weakly Arf ring, then
there exist subrings $A_0, A_1, \ldots, A_{\ell}$ of $\rmQ(A)$ satisfying the  conditions
\begin{enumerate}[$(1)$]
\item $A=A_0 \subsetneq A_1 \subsetneq \cdots \subsetneq A_{\ell} = \overline{A}$,
\item $A_i = {(A_{i-1})}^{I_{i-1}}$, and $I_{i-1} = \overline{a_{i-1}A_{i-1}}$ for some $a_{i-1} \in W(A_{i-1})$.
\end{enumerate}
\end{thm}

\begin{proof}
Let $n = \ell_A(\overline{A}/A) > 0$. Since $A$ is not integrally closed, by Lemma \ref{2.8a} we can choose $a \in W(A)$ such that the ideal $(a)$ is not integrally closed. We set $I=\overline{(a)}$. Proposition \ref{2.9} guarantees that $A^I = I:I$ is Noetherian, weakly Arf, and $A \subsetneq A^I \subseteq \overline{A}$. Notice that $\overline{A_1} = \overline{A}$ is a module-finite extension over $A_1$ and $\ell_{A_1}(\overline{A_1}/A_1) \le \ell_A(\overline{A}/A_1) < n$. Thus, by repeating this procedure, we get the required subrings $A_0, A_1, \ldots, A_{\ell}$ of $\rmQ(A)$. 
\end{proof}


\section{Further properties of weakly Arf rings}\label{sec3}

The purpose of this section is to develop the theory of weakly Arf rings. 
Let us begin with the following, which plays an important role in our discussion.

\begin{prop}\label{3.1}
Let $\varphi : A \to B$ be a homomorphism of rings. Suppose that $aB \cap A = (a)$ and $\varphi(a) \in W(B)$ for every $a \in W(A)$. If $B$ is a weakly Arf ring, then so is $A$.  
\end{prop}

\begin{proof}
Let $x, y, z \in A$ such that $x \in W(A)$ and $y/x, z/x \in \overline{A}$. The map $\varphi : A \to B$ naturally extends to a map $\varphi : \rmQ(A) \to \rmQ(B)$ of the total rings of fractions, because $\varphi(a) \in W(B)$ for every $a \in W(A)$. Since $B$ is weakly Arf, we get $\varphi(x) \in W(B)$, and $\varphi(y/x), \varphi(z/x) \in \overline{B}$. Thus we conclude that $\varphi(yz) = \varphi(y) \varphi(z) \in \varphi(x) B = xB$. 
This implies $yz \in xB \cap A = (x)$. Therefore $A$ is a weakly Arf ring.
\end{proof}

We summarize some consequences of Proposition \ref{3.1}.

\begin{cor}\label{3.2} The following assertions hold.
\begin{enumerate}[$(1)$]
\item Let $A$ be an integral domain and let $R \subseteq A$ be a subring of $A$ such that $R$ is a direct summand of $A$ as an $R$-module. If $A$ is a weakly Arf ring, then so is $R$. 
\item Let $A$ be an integral domain and let $G$ be a finite subgroup of $\operatorname{Aut}A$. Suppose that the order of $G$ is invertible in $A$. If $A$ is a weakly Arf ring, then so is $R=A^G$.
\item Let $\varphi: R \to A$ be a faithfully flat homomorphism of rings. If $A$ is a weakly Arf ring, then so is $R$. 
\item Let $n >0$ and let $A=R[X_1, X_2, \ldots, X_n]$ be the polynomial ring over a ring $R$. If $A$ is a weakly Arf ring, then so is $R$. 
\item Let $M$ be a torsion-free module over a commutative ring $R$. If the idealization $A = R \ltimes M$ is a weakly Arf ring, then so is $R$. 
\end{enumerate}
\end{cor}

We now examine the weakly Arf property in terms of the completion.

\begin{prop}\label{3.3}
Let $(A,\m)$ be a one-dimensional Noetherian local ring. Then the following conditions are equivalent.
\begin{enumerate}[$(1)$]
\item $A$ is a weakly Arf ring.
\item The $\m$-adic completion $\widehat{A}$ of $A$ is a weakly Arf ring.
\end{enumerate} 
\end{prop}

\begin{proof}
$(2) \Rightarrow (1)$ This follows from Corollary \ref{3.2} (3). 

$(1) \Rightarrow (2)$ Let $J \in \Lambda(\widehat{A})$ and write $J = \overline{(\alpha)}$, where $\alpha \in W(\widehat{A})$. We may assume that $\alpha$ is not a unit in $\widehat{A}$. Hence, $J$ is an $\widehat{\m}$-primary ideal of $\widehat{A}$. We are now going to show that $J$ is stable, that is, $J^2 = \alpha J$. To do this, let $I = J \cap A$. Then $I$ is an $\m$-primary ideal of $A$ and $J = I\widehat{A}$. Moreover, since
$$
I \widehat{A} \subseteq \overline{I} \widehat{A} \subseteq \overline{I \widehat{A}} = \overline{J} = J = I \widehat{A},
$$
we obtain $I \widehat{A} = \overline{I} \widehat{A}$. Thus $I$ is integrally closed in $A$, because $\widehat{A}$ is faithfully flat over $A$. Set $Q=\alpha \widehat{A}$ and $\fkq = Q \cap A$. Then $\fkq \widehat{A} = Q$, so that $\mu_A(\fkq)=1$, namely, $\fkq = aA$ for some $a \in A$. Hence $Q= \alpha \widehat{A} = \fkq \widehat{A} = a \widehat{A}$. 
In this case, one can show $I = \overline{aA}$. Indeed, let us consider the equalities
$$
I \widehat{A} = J = \overline{(\alpha)} = \overline{\alpha \widehat{A}} = \overline{a \widehat{A}} = \overline{a A}\cdot\widehat{A}
$$
where the last equality follows from Remark \ref{3.4} below. Again, since $\widehat{A}$ is faithfully flat over $A$, we get
$I = \overline{aA}$.  Remember that $\alpha$ forms a non-zerodivisor on $\widehat{A}$. It follows that $a$ is a non-zerodivisor on $A$. Thus, $I \in \Lambda(A)$ and hence $I^2 = x I$ for some $x \in I$ (use Theorem \ref{2.4}). In particular, $I^2 = a I$. Consequently, the equalities 
$$
J^2 = (I\widehat{A})^2 = a (I \widehat{A}) = \alpha (I \widehat{A}) = \alpha J
$$
yield that $\widehat{A}$ is a weakly Arf ring. This completes the proof.
\end{proof}

In the proof of Proposition \ref{3.3}, we use the following fact. 

\begin{rem}\label{3.4}
Let $(A, \m)$ be a Noetherian local ring, let $\fka$ be an $\m$-primary ideal of $A$. Then $\overline{\fka}\widehat{A} = \overline{\fka \widehat{A}}$.
\end{rem}

\begin{proof}
Let $J = \overline{\fka \widehat{A}}$. Then $J$ is an $\widehat{\m}$-primary ideal of $\widehat{A}$ and $\overline{\fka}\widehat{A} \subseteq J$. Since $\fka \widehat{A}$ is a reduction of $J$, we can choose an integer $n \ge 0$ satisfying $J^{n+1} = (\fka \widehat{A})J^n$. Let $I = J \cap A$. Then $I$ is an $\fkm$-primary ideal of $A$ and $J = I \widehat{A}$. The equality $I^{n+1} \widehat{A} = (\fka I^n)\widehat{A}$ implies $I^{n+1} = \fka I^n$, that is, $\fka$ is a reduction of $I$. Hence $I \subseteq \overline{\fka}$, so that we have $J = I \widehat{A} \subseteq \overline{\fka}\widehat{A}$, as desired.
\end{proof}

By Corollary \ref{3.2} (3), the implication $(2) \Rightarrow (1)$ in Proposition \ref{3.3} holds if $A$ is a Noetherian local ring of arbitrary dimension. However, the implication $(1) \Rightarrow (2)$ does not hold in the case $\dim R \ge 2$.

\begin{rem}
Let ${\Bbb C}[[t, s]]$ be the formal power series ring over the field ${\Bbb C}$ and set $R={\Bbb C}[[t^4, t^5, t^6, s]]$. Then, since $R$ is an isolated complete intersection singularity with $\dim R=2$, by \cite{PS} we can choose a UFD $A$ such that $R \cong \widehat{A}$. Note that $A$ is a weakly Arf ring, because it is a normal domain. If $\widehat{A}$ is weakly Arf, then the faithfully flat homomorphism $S={\Bbb C}[[t^4, t^5, t^6]] \to R \cong \widehat{A}$ guarantees that $S$ is weakly Arf. Hence $S$ is an Arf ring. This is impossible, because $S$ does not have minimal multiplicity. 
\end{rem}

The following example shows that, for a weakly Arf ring $A$, the ring $A/(x)$ is not necessarily weakly Arf for a general non-zerodivisor $x$ on $A$, and vice versa. 


\begin{ex}\label{3.7}
Let $U = k[X, Y, Z]$ be the polynomial ring over a field $k$. Then the following hold.
\begin{enumerate}[$(1)$]
\item $U/(X^3-Z^2)$ is a weakly Arf ring.
\item $U/(X^3-Z^2, Y^2-XZ)$ is not a weakly Arf ring.
\end{enumerate}
\end{ex}

\begin{proof}
Let $t$ denote an indeterminate over $k$. Note that we have $U/(X^3 - Z^2) \cong k[t^2, t^3, Y]$. By Example \ref{2.8b}, $U/(X^3-Z^2)$ is a weakly Arf ring. However, $U/(X^3-Z^2, Y^2-XZ) \cong k[t^4, t^5, t^6]$ is not weakly Arf, because $(t^4)$ is a reduction of $\m$, but $\m^2 \ne t^4 \m$, where $\m = (t^4, t^5, t^6)$.
\end{proof}

In addition, let $k[[t]]$ be the formal power series ring over a field $k$ and consider the ring $A=k[[t^4, t^5, t^6]]$. Then $A$ is not weakly Arf since $A$ is not an Arf ring and every open ideal of $A$ has a principal reduction. On the other hand, as $\depth A/(x) =0$, $A/(x)$ is a weakly Arf ring for every $x \in W(A)$.

\medskip

The next proposition plays a crucial role in our arguments; for example, in the proofs of Example \ref{9.6} and Corollary \ref{10.6} in Sections \ref{sec9} and \ref{sec10}.

\begin{prop}\label{3.8}
Let $\{A_i\}_{i \in \Lambda}~(\Lambda \ne \emptyset)$ be a family of commutative rings and let $A = \prod_{i \in \Lambda}A_i$. Then $A$ is a weakly Arf ring if and only if so is $A_i$ for every $i \in \Lambda$. 
\end{prop}

\begin{proof}
Notice that, for each $a = (a_i)_{i \in \Lambda} \in A$, we have $a \in W(A)$ if and only if $a_i \in W(A_i)$ for every $i \in \Lambda$. Then  $\rmQ(A) = \prod_{i \in \Lambda}\rmQ({A_i})$ and $\overline{A} \subseteq \prod_{i \in \Lambda}\overline{A_i}$.
For each $i \in \Lambda$ and $x \in \rmQ(A_i)$, we define $\widetilde{x} \in \rmQ(A)$ as follows: 
$$
\left(\widetilde{x}\right)_j=
\begin{cases}
\ x & (j=i)\\
\ 1 & (j\ne i)
\end{cases}
\quad \text{in} \ \ \rmQ(A_j) \ \ \text{for each}\ \ j \in \Lambda.
$$
Hence, if $x \in \overline{A_i}$, then $\widetilde{x} \in \overline{A}$, because $\widetilde{x} = \overline{x} + \alpha$, where 
\begin{center}
$\left(\overline{x}\right)_j=
\begin{cases}
\ x & (j=i)\\
\ 0 & (j\ne i)
\end{cases}$ \quad
and \quad
$\alpha_j=
\begin{cases}
\ 0 & (j=i)\\
\ 1 & (j\ne i)
\end{cases}$
\ \ for each $j \in \Lambda$.
\end{center}

We now assume that $A$ is a weakly Arf ring. Fix $i \in \Lambda$. Let $x, y, z \in A_i$ such that $x \in W(A_i)$ and $y/x, z/x \in \overline{A_i}$. One can show that  $\widetilde{x} \in W(A)$ and $\widetilde{y}/\widetilde{x}, \widetilde{z}/\widetilde{x} \in \overline{A}$. Hence, $\widetilde{y} \cdot \widetilde{z} \in \widetilde{x} A$, which yields $yz \in x A_i$. Therefore $A_i$ is a weakly Arf ring. Conversely, suppose that $A_i$ is weakly Arf for every $i \in \Lambda$. Let $a, b, c \in A$ such that $a \in W(A)$ and $b/a, c/a \in \overline{A}$. Write $a = (a_i)_{i \in \Lambda}$, $b = (b_i)_{i \in \Lambda}$, and $c = (c_i)_{i \in \Lambda}$, where $a_i, b_i, c_i \in A_i$. For each $i \in \Lambda$, passing to the $i$-th projection $p: \rmQ(R) \to \rmQ(R_i),\  x \mapsto x_i$, we see that
$$
\frac{b_i}{a_i}, \  \frac{c_i}{a_i}  \in \overline{A_i}.
$$ 
This gives $b_i c_i \in a_i A_i$. Therefore, $bc \in aA$, and hence $A$ is a weakly Arf ring.
\end{proof}

If $\Lambda$ is a finite set, then $\overline{A} = \prod_{i \in \Lambda}\overline{A_i}$. However, the next example shows that the equality is false in general.

\begin{ex}
Let $B = k[t]$ be the polynomial ring over a field $k$. For each $n>0$, we set  $A_n = k[t^n, t^{n+1}]$. Then $\overline{A} \subsetneq \prod_{n>0}\overline{A_n}$, where $A = \prod_{n>0} A_n$.
\end{ex}

\begin{proof}
For each $n >0$, we set $x_n = t$ and $x = (x_n) \in \rmQ(A)$. Then, since $\overline{A_n} = B$, we have $x \in \prod_{n>0}\overline{A_n}$. Suppose that $x \in \overline{A}$. Then there exists an integral equation
$$
x^{\ell} + c_1x^{\ell -1} + \cdots + c_{\ell} = 0
$$
where $\ell >0$ and $c_i \in A$ for all $1 \le i \le \ell$. For each $n >0$, in $A_n$ we have 
$$
t^{\ell} + [c_1]_nt^{\ell -1} + \cdots + [c_{\ell}]_n = \left[x^{\ell} + c_1x^{\ell -1} + \cdots + c_{\ell}\right]_n =0,
$$
which implies $[c_i]_n \ne  0$ for some $1 \le i \le \ell$. As $[c_i]_n \in A_n$, we see that $i \in H_n$, where $H_n = \left<n, n+1\right>$ denotes the numerical semigroup generated by $n$ and $n+1$. Therefore, $n \le i \le \ell$ which gives a contradiction. Hence, $x \notin \overline{R}$, as desired.
\end{proof}

In this paper we will frequently refer to the examples arising from the idealizations and the amalgamated duplications. Let $A$ be a commutative ring and let $I$ be an arbitrary ideal of $A$. For an element $\alpha \in R$, we set $R(\alpha)=A\oplus I$ as an additive group and define the multiplication on $R(\alpha)$ by
$$
(a, x) \cdot (b, y) = (ab, ay + bx + \alpha (xy))
$$
for $(a, x), (b, y) \in R(\alpha)$. If $\alpha = 0$, then $R(0) = A \ltimes I$ is called {\it the idealization $I$ over $A$}, which is introduced by M. Nagata (\cite[Page 2]{N}). If $\alpha = 1$, then $R(1) = A \bowtie I$ is called {\it the amalgamated duplication of $A$ along $I$} in \cite{marco}. Notice that the amalgamated duplication $A \bowtie I$ behaves very much the same way as the idealization $A \ltimes I$ (see e.g., \cite{marco2, marco, EGI, Shapiro}). Moreover, the amalgamated duplication contains the fiber product of the two copies of the natural homomorphism $A \to A/I$ via the identification:
$A \bowtie I \cong A \times_{A/I} A$, where $(a,i) \mapsto (a, a+i)$. 
Later we will deal with the idealizations and the fiber products; see Sections 5, 6, and 13. 

Furthermore, we have an isomorphism $A \bowtie A \cong A \times A$ of rings. 
As a special case of the amalgamated duplications, we immediately get the following. 


\begin{cor}
$A \bowtie A$ is a weakly Arf ring if and only if so is $A$.
\end{cor}



\section{Strict closures and polynomial extensions}\label{sec4}

We start this section by recalling the definition of the strict closure of a ring introduced by J. Lipman \cite{L}. 
For a commutative ring $A$, we define
$$
A^{*}=\left\{ x \in \overline{A} \mid x \otimes 1 = 1 \otimes x \text{ in } \overline{A} \otimes_A \overline{A} \right\} \ \subseteq \ \overline{A}
$$
which forms a subring of $\overline{A}$, containing $A$. The ring $A^{*}$ is  called {\it the strict closure of $A$ in $\overline{A}$}, and we say that $A$ is {\it strictly closed in $\overline{A}$}, if the equality $A=A^{*}$ holds.

We begin with the following.

\begin{lem}\label{4.1}
Let $A$ be a Noetherian ring. If $A$ satisfies $(S_1)$, then 
$(S^{-1}A)^* = S^{-1} A^*$ in $\rmQ(S^{-1}A)$ for every multiplicatively closed subset $S$ of $A$. 
\end{lem}

\begin{proof}
Since $A$ satisfies $(S_1)$, every associated prime ideal of $A$ is minimal prime, so $S^{-1}\rmQ(A) = \rmQ(S^{-1}A)$. Remember that $S^{-1}\overline{A} = \overline{S^{-1}A}$. Let $\xi \in S^{-1}\overline{A}$ and write $\xi = x/s$ where $x \in \overline{A}, s \in S$. We identify $\xi = x/s$ with $(1/s) \otimes x$ via the natural isomorphism $S^{-1}\overline{A} \cong S^{-1} A \otimes_A \overline{A}$. We also have the canonical isomorphisms
$$
S^{-1}\overline{A} \otimes_{S^{-1}A} S^{-1}\overline{A} \cong (S^{-1} A \otimes_A \overline{A}) \otimes_{S^{-1}A} (S^{-1} A \otimes_A \overline{A}) \cong S^{-1}A \otimes_A (\overline{A} \otimes_A \overline{A}),
$$
whence the elements $\xi \otimes 1$ and $1 \otimes \xi$ in $S^{-1}\overline{A} \otimes_{S^{-1}A} S^{-1}\overline{A}$ correspond to the elements 
$\left(1/s\right)\left[1 \otimes (x \otimes 1)\right]$ and $\left(1/s\right)\left[1 \otimes (1 \otimes x)\right]$ in $S^{-1}A \otimes_A (\overline{A} \otimes_A \overline{A})$, respectively. Therefore
\begin{center}
$\xi \in (S^{-1}A)^*$ if and only if $1 \otimes (x \otimes 1) = 1 \otimes (1 \otimes x)$ in $S^{-1}A \otimes_A (\overline{A} \otimes_A \overline{A})$. 
\end{center}
The latter condition is equivalent to saying that there exists $t \in S$ such that $t\cdot(x\otimes 1) = t\cdot (1 \otimes x)$ in $\overline{A} \otimes_A \overline{A}$, i.e.,  $tx \in A^*$. It then follows that $\xi \in (S^{-1}A)^*$ if and only if $\xi = x/s \in S^{-1}A^*$, as desired. 
\end{proof}

Recall that, for a commutative ring $A$, we set $X_1(A) =\{P \in \Spec A \mid \depth A_P \le 1\}$; see Section 2. The following ensures that the strict closedness is a local condition. 

\begin{prop}\label{4.2}
Let $A$ be a Noetherian ring. Suppose that $A$ satisfies $(S_1)$. Then 
the following conditions are equivalent.
\begin{enumerate}[$(1)$]
\item $A$ is strictly closed in $\overline{A}$.  
\item $A_P$ is strictly closed in $\overline{A_P}$ for every $P \in X_1(A)$.
\end{enumerate}
\end{prop}

\begin{proof}
By Lemma \ref{4.1}, we only need to prove the implication $(2) \Rightarrow (1)$. Suppose the contrary, i.e., we assume that $A \subsetneq A^*$. Choose $P \in \Ass_A(A^*/A)$ and write $P = A:_A x$ for some $x \in A^*$. By setting $x = b/a$ with $a \in W(A)$ and $b \in A$, we then have $P = (a):_A b$ and the homothety map $\widehat{b} : A/P \to A/(a)$ is injective. Since $P \in \Ass_A(A/(a))$, we have $\depth_{A_P}(A_P/aA_P) = 0$. This implies $\depth A_P = 1$. In particular, $P \in X_1 (A)$, and hence $A_P$ is strictly closed. 
Therefore, because $x/1 \in (A_P)^* = A_P$, we can write $x/1 = y/s$ with $y \in A, s \in A\setminus P$, so there exists $t \in A\setminus P$ such that $t(sx) = t y \in A$. This implies $ts \in P$, which is a contradiction. Hence $A$ is strictly closed in $\overline{A}$. 
\end{proof}

The relation between Arf rings and the rings that are strictly closed in their integral closure was explored in \cite[Section 4]{L}. More precisely, for a Noetherian semi-local ring $A$ such that $A_M$ is a one-dimensional Cohen-Macaulay local ring for every $M \in \Max A$, J. Lipman proved that,  if $A$ is strictly closed in $\overline{A}$, then $A$ is an Arf ring, and the converse holds if $A$ contains a field; see \cite[Proposition 4.5, Theorem 4.6]{L}. 

Let us first emphasize that the Arf property implies the strict closedness of the ring without assuming that the ring contains a field. Indeed, to show  \cite[Theorem 4.6]{L}, J. Lipman only used the assumption that the ring contains a field in the proof of \cite[Lemma 4.7]{L}. We prove that this lemma also holds in the case where the ring does not contain a field. 

\begin{lem}
Let $(A, \m)$ be a one-dimensional Cohen-Macaulay local ring. Suppose that $\m^2 = z\m$ for some $z\in \m$. Set $A_1 = A^{\m} = \m z^{-1}$.  Let $A_1 \subseteq C \subseteq \overline{A}$ be an intermediate ring such that $C$ is a finitely generated $A$-module and let $\alpha : C\otimes_A C \to C \otimes_{A_1} C$ be an $A$-algebra map such that $\alpha(x \otimes y) = x \otimes y$ for every $x, y \in C$. Then $\Ker \alpha = (0):_{C \otimes_AC} z$. 
\end{lem}

\begin{proof}
Notice that $\Ker \alpha = (x \otimes 1 - 1 \otimes x \mid x \in A_1)$. This implies $\Ker \alpha \subseteq (0):_{C \otimes_AC} z$, as $x \in A_1$ and $A_1 = \m z^{-1}$. Let us make sure of the opposite inclusion. Let $n = \mu_R(C)$ and choose a minimal free presentation
$$
A^{\oplus \ell} \overset{\varphi}{\to} A^{\oplus n} \overset{\varepsilon}{\to} C \to 0
$$
of $C$ as an $A$-module, where $\ell \ge 0$. This induces an exact sequence
$$
C \otimes_AA^{\oplus \ell} \overset{C \otimes \varphi}{\longrightarrow} C \otimes_AA^{\oplus n} \overset{C \otimes \varepsilon}{\longrightarrow} C \otimes_AC \longrightarrow 0
$$
of $A$-modules. Note that if $C$ is a free $A$-module, then $n=1$ and $C = A = A_1$. Hence, we may assume that $C$ is not $A$-free, i.e., $\ell >0$. 
Let $\{f_j\}_{1 \le j \le \ell}$ and $\{e_i\}_{1 \le i \le n}$ be the standard bases of $A^{\oplus \ell}$ and $A^{\oplus n}$, respectively. Let us write $\varphi = \left[a_{ij}\right]_{1 \le i \le n, 1 \le j \le \ell}$, where $a_{ij} \in \m$. Thus, for each $1 \le j \le \ell$, we have 
$$
\varphi(f_j) = \sum_{i=1}^{n}a_{ij} e_i. 
$$
Let $\xi \in C \otimes_A C$ such that $\xi z =0$. We write $\xi = \sum_{i=1}^nc_i (1 \otimes b_i)$, where $c_i \in C$ and $b_i = \varepsilon(e_i)$. Since $\xi z = \sum_{i=1}^n (zc_i)(1 \otimes b_i) =0$, we can choose $\eta \in C \otimes_AA^{\oplus \ell}$ such that  
$$
(C \otimes \varphi)(\eta) = \sum_{i=1}^n (zc_i)(1 \otimes e_i).
$$
Let us write $\eta = \sum_{j=1}^{\ell}d_j(1 \otimes f_j)$ with $d_j \in C$. As $C \otimes_AA^{\oplus \ell} \cong C^{\oplus \ell}$, the equalities
$$
\sum_{i=1}^n (zc_i)(1 \otimes e_i) = \sum_{j=1}^{\ell}d_j(1 \otimes \varphi(f_j)) = \sum_{j=1}^{\ell}d_j\left[1 \otimes \left(\sum_{i=1}^{n}a_{ij} e_i \right)\right] = \sum_{i=1}^n \left(\sum_{j=1}^{\ell} a_{ij}d_j\right) (1 \otimes e_i)
$$
yield that 
$$
zc_i = \sum_{j=1}^{\ell} a_{ij}d_j
$$
for every $1 \le i \le n$. Hence, in $C \otimes_A C$, we have 
$$
\xi = \sum_{i=1}^n c_i (1 \otimes b_i) = \sum_{i=1}^n\left(\sum_{j=1}^{\ell} a_{ij}d_j\right)(1 \otimes b_i) =  \sum_{j=1}^{\ell}\left(\sum_{i=1}^{n} b_{ij}d_j\right)(1 \otimes b_i)  
$$
where $b_{ij} = a_{ij} z^{-1} \in \m z^{-1}=A_1$. Therefore, in $C \otimes_{A_1} C$, we obtain
$$
\alpha(\xi) = \sum_{j=1}^{\ell}\left(\sum_{i=1}^{n} (d_j \otimes b_{ij} b_i)\right) = \sum_{j=1}^{\ell}\left[ d_j \otimes \left(\sum_{i=1}^n b_{ij}b_i \right) \right] = 0
$$
where the last equality comes from the fact that
$$
\sum_{i=1}^n b_{ij}b_i = \frac{1}{z} \left(\sum_{i=1}^n a_{ij}b_i\right) =  \frac{1}{z} \varepsilon(\varphi (f_j)) =0.
$$
This completes the proof.
\end{proof}

Consequently we have the following. 
 
\begin{thm}[{cf. \cite[Proposition 4.5, Theorem 4.6]{L}}]\label{4.3}
Let $A$ be a Noetherian semi-local ring such that $A_M$ is a one-dimensional Cohen-Macaulay local ring for every $M \in \Max A$. Then $A$ is strictly closed in $\overline{A}$ if and only if $A$ is an Arf ring. 
\end{thm}

In our more general context, the following still holds.

\begin{thm}\label{Zariski-Lipman}
Let $A$ be a Noetherian ring. Suppose that $A$ satisfies $(S_2)$. Then the following conditions are equivalent.
\begin{enumerate}[$(1)$]
\item $A$ is strictly closed in $\overline{A}$.
\item $A$ is a weakly Arf ring and $A_P$ is an Arf ring for every $P \in \Spec A$ with $\height_AP=1$.
\end{enumerate}
\end{thm}

\begin{proof}
$(1) \Rightarrow (2)$ Let $x,y,z \in A$ with $x \in W(A)$ such that $y/x, z/x \in \overline{A}$. In the ring $\overline{A}\otimes_A \overline{A}$, we have 
$$
\frac{yz}{x} \otimes 1=\frac{y}{x} \otimes \left(x \cdot \frac{z}{x}\right)=\left(\frac{y}{x}\cdot x\right) \otimes  \frac{z}{x}=1 \otimes \frac{yz}{x}.
$$
If $A$ is strictly closed, then $yz/x \in A$. Hence $A$ is a weakly Arf ring. Let us make sure of the Arf property for the local ring $A_P$. For each $P \in \Spec A$ with $\height_AP = 1$, $A_P$ is strictly closed; see Proposition \ref{4.2}. Hence, by \cite[Proposition 4.5]{L}, we conclude that $A_P$ is an Arf ring.

$(2) \Rightarrow (1)$ By Proposition \ref{4.2}, it is enough to show that $A_P$ is strictly closed for every $P \in X_1(A)$. As $P \in X_1(A)$, we have $\depth A_P \le 1$. If $\depth A_P = 0$, then it coincides with its total ring of fractions. This yields that $A_P$ is strictly closed. We assume  that $\depth A_P=1$. Then, because $A$ satisfies $(S_2)$, we see that $A_P$ is a Cohen-Macaulay local ring with $\dim A_P=1$. Hence, by Theorem \ref{4.3}, $A_P$ is strictly closed. This completes the proof.
\end{proof}

Recall that, for a given one-dimensional Cohen-Macaulay local ring $(A, \m)$, if $A$ is an Arf ring, then $A$ is weakly Arf. The converse holds if every integrally closed $\m$-primary ideal $I$ in $A$ contains a principal reduction, i.e., there exists $a \in I$ such that $(a)$ is a reduction of $I$ (see Definitions \ref{2.2} and \ref{2.3}). In particular, the notions of Arf ring and  weakly Arf ring coincide, if the ring possesses an infinite residue class field.
Hence we get the following. 

\begin{cor}\label{4.4}
Let $A$ be a Noetherian ring. Suppose that $A$ satisfies $(S_2)$ and one of the following conditions:
\begin{enumerate}[$(1)$]
\item $A$ contains an infinite field.
\item $\height_AM \ge 2$ for every $M \in \Max A$.
\end{enumerate}
Then $A$ is strictly closed in $\overline{A}$ if and only if $A$ is a weakly Arf ring, or, equivalently $A_P$ is an Arf ring for every $P \in \Spec A$ with $\height_AP=1$. 
\end{cor}


\begin{proof}
By Theorem \ref{Zariski-Lipman}, it is enough to show that $A$ is weakly Arf if and only if $A_P$ is an Arf ring for every $P \in \Spec A$ with $\height_AP=1$. Recall that the weakly Arf property of $A$ and that of $A_P$ for every $P \in X_1(A)$ coincide; see Theorem \ref{2.7}. Hence 
the assertion follows, if $A$ contains an infinite field.
Suppose now that $\height_AM \ge 2$ for every $M \in \Max A$ and that $A$ is a weakly Arf ring. Let $P \in \Spec A$ with $\height_AP=1$. Choose a maximal ideal $M$ in $A$ such that $P \subsetneq M$. Then $\dim A/P \ge 1$ and hence the cardinality of $A/P$ is infinite. Thus $A_P$ is a one-dimensional Cohen-Macaulay local ring with infinite residue class field and that is weakly Arf. Hence $A_P$ is an Arf ring. 
\end{proof}

In Corollary \ref{4.4}, we cannot remove the assumption that $\height_AM \ge 2$ for every $M \in \Max A$ unless $A$ contains an infinite field as we show in the following.

\begin{rem}
For $i=1, 2$, let $A_i$ be a Cohen-Macaulay local ring with $\dim A_i = i$. 
Then $A = A_1 \times A_2$ is a Cohen-Macaulay ring with $\dim A=2$. If $A_1$ and $A_2$ are weakly Arf rings, then so is $A$; see Proposition \ref{3.8}. Moreover, the strict closedness of $A$ leads us to obtain that both $A_1$ and $A_2$ are strictly closed in $\overline{A_1}$ and $\overline{A_1}$, respectively. Therefore, if $A_1$ is not strictly closed, then $A$ cannot be strictly closed in $\overline{A}$ even if $A_2$ is strictly closed. 
\end{rem}


\begin{cor}
Let $(A, \m)$ be a Noetherian local ring with $\dim A \ge 2$. Suppose that $A$ satisfies $(S_2)$. Then $A$ is strictly closed in $\overline{A}$ if and only if $A$ is a weakly Arf ring. 
\end{cor}

In the rest of this section, we explore the weakly Arf property for  polynomial extensions.

\begin{lem}\label{4.5}
Let $\rmQ(A)[X]$ be the polynomial ring over $\rmQ(A)$ and set $B = A[X]$. Suppose $\overline{B}  \subseteq \rmQ(A)[X]$ $($e.g., $A$ is an integral domain$)$. Then $B^* = (A^*)[X]$. Therefore, $B$ is strictly closed in $\overline{B}$ if and only if $A$ is strictly closed in $\overline{A}$. 
\end{lem}

\begin{proof}
Since $W(A) \subseteq W(B)$, we may consider $\rmQ(A) \subseteq \rmQ(B)$. As $B \subseteq \rmQ(A)[X] \subseteq \rmQ(B)$, $\rmQ(B)$ coincides with the total ring of fractions of $\rmQ(A)[X]$. By \cite[(5.12) Proposition]{Bass}, we see that $\overline{B} = \overline{A}[X]$, so we then have the canonical isomorphisms
\begin{eqnarray*}
\overline{B} \otimes_B \overline{B} &=&  \overline{A}[X] \otimes_{A[X]} \overline{A}[X] \cong (A[X] \otimes_A \overline{A})\otimes_{A[X]} (A[X] \otimes_A \overline{A}) \\
&\cong& A[X] \otimes_A(\overline{A}\otimes_A\overline{A}) \cong (\overline{A}\otimes_A\overline{A})[X].
\end{eqnarray*}
Then the element $\alpha X^n \otimes \beta X^{\ell}$ in $\overline{B} \otimes_B \overline{B}$ corresponds to $X^{n+\ell}\otimes \left(\alpha \otimes \beta\right)$ in $(\overline{A}\otimes_A\overline{A})[X]$ via the above isomorphisms, where $\alpha, \beta \in \overline{A}$ and $n, \ell \ge 0$. Let $\xi \in \overline{B}$. We  write $\xi = \sum_{i=1}^{\ell}\alpha_iX^i$ with $\alpha_i \in \overline{A}$. 
Then
\begin{center}
$\xi \in B^*$ if and only if $\sum_{i=1}^{\ell} X^i \otimes (\alpha_i \otimes 1) = \sum_{i=1}^{\ell} X^i \otimes (1 \otimes \alpha_i)$,
\end{center}
which is equivalent to saying that $\alpha_i \otimes 1 = 1 \otimes \alpha_i$ for all $1 \le i \le \ell$, i.e., $\alpha_i \in A^*$. Hence $B^* = (A^*)[X]$, as desired.
\end{proof}

\begin{cor}\label{polynomial}
Let $A$ be an integral domain. If $A$ is strictly closed in $\overline{A}$, then the polynomial ring $A[X_1, X_2, \ldots, X_n, \ldots]$ is a weakly Arf ring. 
\end{cor}

\begin{proof}
Let $T = A[X_1, X_2, \ldots, X_n, \ldots]$ be the polynomial ring over $A$. 
Let $x, y, z \in T$ such that $x \ne 0$ and $y/x, z/x \in \overline{T}$. Choose an integer $n \gg0$ such that $x, y, z \in T_n$ and $y/x, z/x \in \overline{T_n}$, where $T_n = A[X_1, X_2, \ldots, X_n]$. By Lemma \ref{4.5}, $T_n$ is strictly closed, and hence it is weakly Arf. Hence $yz \in xT_n \subseteq xT$, so that $T$ is a weakly Arf ring. 
\end{proof}

By Corollary \ref{3.2} (4), if the polynomial ring $A[X]$ over a commutative ring $A$ is weakly Arf, then so is $A$. For the converse, we have the following.


\begin{thm}\label{4.6}
Let $A$ be a Noetherian ring. Suppose that $A$ satisfies $(S_2)$, the polynomial ring $\rmQ(A)[X]$ is integrally closed $($e.g., $A$ is an integral domain$)$, and one of the following conditions:
\begin{enumerate}[$(1)$]
\item $A$ contains an infinite field.
\item $\height_AM \ge 2$ for every $M \in \Max A$.
\end{enumerate}
Then $A$ is a weakly Arf ring if and only if so is $A[X]$. 
\end{thm}

\begin{proof}
By Corollary \ref{4.4}, $A$ is weakly Arf if and only if it is strictly closed in $\overline{A}$. We set $B = A[X]$. Then $\overline{B} \subseteq \rmQ(A)[X]$. Hence the assertion follows from Lemma \ref{4.5}.
\end{proof}

To sum up this kind of arguments, we finally get the following.

\begin{cor}\label{4.7}
Let $A$ be a Noetherian integral domain. Suppose that $A$ satisfies $(S_2)$ and one of the following conditions:
\begin{enumerate}[$(1)$]
\item $A$ contains an infinite field.
\item $\height_AM \ge 2$ for every $M \in \Max A$.
\end{enumerate}
Then $A$ is a weakly Arf ring if and only if so is the polynomial ring $A[X_1, X_2, \ldots, X_n]$ for every $n \ge 1$. 
\end{cor}

As we show next, the `only if' part of Theorem \ref{4.6} is false in general.

\begin{prop}\label{4.8}
Let $(A, \m)$ be an Artinian local ring and let $B = A[X]$ be the polynomial ring over $A$. Then $B$ is a weakly Arf ring if and only if $\m^2 = (0)$. Hence, if $\m^2 \ne (0)$, then $B$ is not a weakly Arf ring, but $A$ is weakly Arf. 
\end{prop}

\begin{proof}
Let $f \in W(B)$. Set $I = \overline{fB}$. Then one has $\m B \subseteq I$, as $\m$ is nilpotent. We set $C = B/\fkm B \cong (A/\fkm)[X]$. Because $fB$ is a reduction of $I$, we then have 
$$
IC \subseteq \overline{fC} = fC
$$
whence $IC = fC$. Therefore, $I = fB + \m B$. Since $I^2 = fI + \m^2 B$, we get $I^2 = fI$ provided $\m^2 = (0)$. Hence $B$ is a weakly Arf ring. Conversely, suppose that $B$ is weakly Arf. Then, by taking $f=X$ and applying the above argument, we get $I^2 = XI + \m^2 B$. Since $I \in \Lambda(B)$, Theorem \ref{2.4} shows that $I^2 = XI$. Hence 
 $\m^2 \subseteq XI$, i.e., $\m^2 =(0)$.
\end{proof}

As an application of Theorem \ref{4.6}, we obtain the following results related to the core subalgebras in the polynomial ring. We will use Proposition \ref{4.9} in Section \ref{sec8}.

\begin{prop}\label{4.9}
Let $S=k[t]$ be the polynomial ring over a field $k$ and let $k \subseteq R \subseteq S$ be an intermediate ring such that $t^nS \subseteq R$ for some $n >0$. Then $R$ is a weakly Arf ring if and only if $R_{\fkm}$ is an Arf ring, where $\fkm = tS \cap R \in \Max R$.
\end{prop}

\begin{proof}
We denote by $N = tS$ the maximal ideal of $S$. Since $R_{\fkm} \subseteq S_{\fkm} = S_{N}$ and every ideal of $R_{\m}$ contains a principal reduction, $R_\fkm$ is an Arf ring if and only if $R_\fkm$ is weakly Arf. Hence, if $R$ is a weakly Arf ring, then the local ring $R_{\m}$ is Arf; see Theorem \ref{2.7}.  
Conversely, if $R_{\m}$ is an Arf ring, then it is strictly closed by \cite[Theorem 4.6]{L}. Note that, for each $P \in \Spec R$ such that $P \ne \fkm$, we have that  $R_P$ is regular. Indeed, if $P \ne \fkm$, then $R:S \not\subseteq P$, because $R:S$ contains $t^n S$. Hence $R_P = S_P$, so that it is regular with $\dim R_P \le 1$. Therefore, if $R_{\m}$ is an Arf ring, then by Proposition \ref{4.2} we conclude that $R$ is strictly closed; hence $R$ is weakly Arf. This completes the proof.
\end{proof}

We will occasionally refer to  the examples arising from numerical semigroup rings.  
Let $0 < a_1, a_2, \ldots, a_\ell \in \Bbb Z~(\ell >0)$ be positive integers such that $\mathrm{GCD}(a_1, a_2, \ldots, a_\ell)=1$. We set 
$$
H = \left<a_1, a_2, \ldots, a_\ell\right>=\left\{\sum_{i=1}^\ell c_ia_i \mid 0 \le c_i \in \Bbb Z~\text{for~all}~1 \le i \le \ell \right\}
$$
and call it {\it the numerical semigroup generated by $\{a_i\}_{1 \le i \le \ell}$}. Let $S = k[t]$ be the polynomial ring over a field $k$. We set $R = k[H] = k[t^{a_1}, t^{a_2}, \ldots, t^{a_\ell}]$ in $S$ and call it {\it the semigroup ring of $H$ over $k$}. The ring  $R$ is a one-dimensional Cohen-Macaulay domain with $\overline{R} = S$ and $\m = (t^{a_1},t^{a_2}, \ldots, t^{a_\ell} ) = tS \cap R \in \Max R$.

With this notation, we end this section with the following. 

\begin{cor}\label{4.10}
Let $S = k[t]$ be the polynomial ring over an infinite field $k$. Let $R=k[H]$ be the semigroup ring of a numerical semigroup $H$. If $R_{\fkm}$ is an Arf ring, then the polynomial ring $R[X_1, X_2, \ldots, X_n]$ is a weakly Arf ring for every $n >0$, where $\fkm = tS \cap R$.
\end{cor}


\section{Arf and weakly Arf rings arising from idealizations}\label{sec5}

In this section, we study the problem of when the idealizations $A = R \ltimes M$ of torsion-free modules $M$ over Noetherian rings $R$ are Arf and weakly Arf. The main goal is to enrich the theory by providing 
concrete examples of weakly Arf rings. 
Throughout this section, let $R$ be a Noetherian ring, and let  $M$ be a finitely generated torsion-free $R$-module. Let $A=R \ltimes M$ be the idealization of $M$ over $R$. Thus $A = R \oplus M$ is considered as an $R$-module and the multiplication in $A$ is given  by
$$
(a,x)\cdot(b,y) = (ab, bx + ay)
$$
for each $(a, x), (b, y) \in A$; see also Section 3. Note that it follows that  $A$ is Noetherian and  $\fka = (0) \times M$ is an ideal of $A$ with $\fka^2=(0)$. Hence, when $(R, \m)$ is a local ring, then so is $A = R \ltimes M$ with the maximal ideal $\fkm \times M$ since  $R \cong (R \ltimes M)/\fka$. 

In this section, we consider $M$ as an $R$-submodule of $L = \rmQ(R) \otimes_R M$, since $M$ is torsion-free. Notice that $\rmQ(A) = \rmQ(R) \ltimes  L$ and $\overline{A} = \overline{R} \ltimes L$.

\begin{lem}\label{5.1}
$M$ is an $\overline{R}$-module if and only if $\overline{(a)}\cdot M = aM$ for every $a \in W(R)$.
\end{lem}

\begin{proof}
Suppose that $M$ is an $\overline{R}$-module. For each $a \in W(R)$, we have $\overline{(a)} = a \overline{R} \cap R \subseteq a \overline{R}$. Hence $\overline{(a)}\cdot M \subseteq a \overline{R} \cdot M = aM$. In particular $\overline{(a)} \cdot M = aM$. Conversely, let $x \in \overline{R}$. We write $x = b/a$ with $a \in W(R)$ and $b \in R$. Then $b \in a \overline{R} \cap R =\overline{(a)}$. Hence $bM\subseteq \overline{(a)}M = aM$, which implies $xM \subseteq M$. Therefore, $M$ is an $\overline{R}$-module.
\end{proof}

We are now ready to prove the following characterization of the weakly Arf property of the idealization $A = R \ltimes M$. 

\begin{thm}\label{5.2}
The following conditions are equivalent.
\begin{enumerate}[$(1)$]
\item $A = R \ltimes M$ is a weakly Arf ring.
\item $R$ is a weakly Arf ring and $\overline{(a)}\cdot M = aM$ for every $a \in W(R)$.
\item $R$ is a weakly Arf ring and $M$ is an $\overline{R}$-module.
\end{enumerate}
\end{thm}

\begin{proof}
$(2) \Leftrightarrow (3)$ This follows from Lemma \ref{5.1}.

$(1) \Rightarrow (2)$ Let $a \in W(R)$. Set $I = \overline{(a)}$. Then, since $\alpha =(a, 0)$ is a non-zerodivisor on $A$, we get $J = \overline{\alpha A} \in \Lambda (A)$. Hence, by Theorem \ref{2.4}, $J^2 = \alpha J = (a, 0)J = aJ$, so that $J = \overline{aA}$. Therefore we have the equalities
$$
J = \overline{aA} = a \overline{A} \cap A = (a \overline{R} \times a L) \cap A = (a \overline{R} \cap R) \times M = I \times M
$$
whence $J^2 = I^2 \times IM$ and $aJ = a I \times aM$. Consequently, $I^2 = aI$ and $IM = aM$. In particular, $R$ is a weakly Arf ring.

$(2) \Rightarrow (1)$ Let $\alpha \in W(A)$. We write $\alpha = (a, b)$, where  $a \in R$ and $b \in M$. Then, because $\alpha \in A$ is a unit in $\rmQ(A)$, so is $a \in R$ in $\rmQ(R)$, i.e., $a \in W(R)$. Let $J = \overline{\alpha A}$. 
As the element $(0, b)$ is nilpotent, we have $J = \overline{(a, 0)} = \overline{aA}$.  Hence, by setting $I = \overline{(a)}$, we obtain
$$
J = \overline{aA} = a \overline{A} \cap A = (a \overline{R} \times a L) \cap A = (a \overline{R} \cap R) \times M = I \times M.
$$
Thus, $J^2 = I^2 \times IM = aI \times aM = aJ$, where the second equality comes from the weakly Arf property for $R$ and our hypothesis. Hence, $J^2 =  a J = \alpha J$, because $\red_{(\alpha)}(J) = \red_{(a)}(J) = 1$. This implies $A$ is a weakly Arf ring. 
\end{proof}

As a consequence of Theorem~\ref{5.2}, we have the following corollaries. 

\begin{cor}\label{5.3}
$R \ltimes R$ is a weakly Arf ring if and only if $R$ is integrally closed. 
\end{cor}

Notice that, if $M$ is a torsion-free $\overline{R}$-module, then it is torsion-free as an $R$-module. 


\begin{cor}\label{5.4}
Let $N$ be a finitely generated torsion-free $\overline{R}$-module. Then the following conditions are equivalent.
\begin{enumerate}[$(1)$]
\item $A = R \ltimes N$ is a weakly Arf ring.
\item $R$ is a weakly Arf ring.
\end{enumerate}
\end{cor}

Hence, if $L$ is a fractional ideal of $\overline{R}$, then $A = R \ltimes L$ is a weakly Arf ring if and only if so is $R$. In particular, we have the following.  

\begin{cor}\label{5.5}
$A = R \ltimes (R:\overline{R})$ is a weakly Arf ring if and only if so is $R$.
\end{cor}

For a given numerical semigroup $H$, the integral closure of the semigroup ring $k[H]$ over a field $k$ coincides with the polynomial ring with one variable; see the discussion before Corollary \ref{4.10}. Hence, by Corollaries \ref{5.4} and \ref{5.5}, one can construct numerous examples of weakly Arf rings obtained from $k[H]$.





\begin{cor}\label{5.8}
Let $(R, \m)$ be a Noetherian local domain and let $M$ be a finitely generated torsion-free $R$-module. Suppose that $R$ is a weakly Arf ring and $\overline{R}$ is a DVR.  
Then the following conditions are equivalent.
\begin{enumerate}[$(1)$]
\item $A = R \ltimes M$ is a weakly Arf ring.
\item $M \cong \overline{R}^{\oplus n}$ as an $R$-module for some $n \ge 0$. 
\end{enumerate}
\end{cor}

\begin{proof}
Notice that $\depth R >0$ and $\depth_RM>0$. Hence the condition $(1)$ is equivalent to the condition that $M$ is an $\overline{R}$-module, i.e., $M$ is $\overline{R}$-free.  
\end{proof}

\begin{ex}
Let $S=k[[t]]$ be the formal power series ring over a field $k$. We set $R= k[[t^2, t^3]]$. Then $A = R \ltimes S^{\oplus n}$ is a weakly Arf ring for every $n \ge 0$. In addition, $R \ltimes (t^2, t^3)$ is a weakly Arf ring, because $(t^2, t^3) = t^2 S \cong S$ as an $R$-module.
\end{ex}

An almost Gorenstein ring is one of the generalizations of a Gorenstein ring, defined by a certain embedding of the rings into their canonical modules. The theory was initially established by V. Barucci and R. Fr{\"o}berg \cite{BF} in the case where the local rings are analytically unramified of dimension one. However, since the notion given by \cite{BF} was not flexible for the analysis of analytically ramified case, in 2013,  S. Goto, N. Matsuoka and T. T. Phuong \cite{GMP} extended the notion over one-dimensional Cohen-Macaulay local rings. Two years later, in 2015, S. Goto, R. Takahashi, and N. Taniguchi \cite{GTT} gave the definition for the almost Gorenstein rings of arbitrary dimension, by using the notion of Ulrich modules. The reader may consult \cite{BF, GMP, GTT} for basic properties of such rings.

Recently, E. Celikbas, O. Celikbas, S. Goto, and N. Taniguchi \cite{CCGT} explored the question of when the Arf rings are almost Gorenstein. In particular, a criterion for the idealization $R\ltimes \m$ to be an almost Gorenstein Arf ring is given in \cite[Corollary 4.11]{CCGT}.  

\medskip

Let us now investigate what happens if the idealization $R\ltimes \m$ is an Arf ring. 


\begin{cor}\label{AGARF}
Let $(R, \m)$ be a one-dimensional Cohen-Macaulay local ring, possessing an infinite residue class field. If $R \ltimes \m$ is an Arf ring, then $R$ is an almost Gorenstein Arf ring, and $\m:\m$ is a Gorenstein ring.
\end{cor}

\begin{proof}
Suppose that $R \ltimes \fkm$ is an Arf ring. Then $R \ltimes \m$ is weakly Arf; so we obtain the condition that $R$ is weakly Arf and $\m \overline{R} = \m$. Then $\overline{R}$ is a module-finite extension over $R$, because $\overline{R} = \m:\m$ is the endomorphism algebra of $\m$. Hence, the $\m$-adic completion $\widehat{R}$ of $R$ is reduced. By \cite[Proposition 2.7, Corollary 2.8]{GMP}, we choose an $R$-submodule $K$ of $\overline{R}$ such that $R\subseteq K \subseteq \overline{R}$ and $K \cong \rmK_R$ as an $R$-module, where $\rmK_R$ denotes the canonical module of $R$. As $\m K \subseteq \m$, we conclude that $R$ is an almost Gorenstein ring (\cite[Theorem 3.11]{GMP}). Since $R$ has an infinite residue class field, it is Arf; hence $R$ has minimal multiplicity. Thus $\m^2 = a\m$ for some $a \in \m$. The Gorenstein property for the algebra $\m:\m$ follows from \cite[Theorem 5.1]{GMP}. This completes the proof.
\end{proof}


Inspired by Corollary \ref{AGARF}, in what follows, we study the Arf rings obtained by idealization more generally. 
We denote by $A = R\ltimes M$ the idealization of $M$ over $R$.  Consider the canonical ring homomorphisms 
$
\varphi : R \to A, \ a \mapsto (a, 0), \ p: A \to R, \ (a, b) \mapsto a.
$

\begin{lem}\label{12.1}
The following assertions hold true. 
\begin{enumerate}[$(1)$]
\item Let $I$ be an ideal of $R$ and set $L = I \times M$. Then $L$ is an ideal of $A$, and $L$ is integrally closed in $A$ if and only if $I$ is integrally closed in $R$.
\item Conversely, let $L$ be an ideal of $A$. Then $L$ is integrally closed in $A$ if and only if there exists an integrally closed ideal $I$ in $R$ such that $L=I\times M$. When this is the case, one has $I = \varphi^{-1}(L)$.
\end{enumerate}
\end{lem}

\begin{proof}
$(1)$ Notice that $L$ is an $R$-submodule of $A$. For each $\alpha \in A$, $\ell \in L$, we write $\alpha = (a, x)$ and $\ell = (i, y)$ with $a \in R, i \in I, x, y \in M$. Then $$
\alpha \ell = (a, x)(i, y) = (ai, ay + ix) \in L
$$
which yields that $L$ is an ideal of $A$. Suppose that $L$ is integrally closed in $A$. Then, for every $b \in \overline{I}$, we get $(b, 0) \in \overline{IA} \subseteq \overline{L} = L$. Hence, $b \in I$, i.e., $I$ is an integrally closed ideal in $R$. On the other hand, we assume that $I$ is integrally closed. For every $\alpha=(a, x) \in \overline{L}$, via the projection $p:A \to R$, we obtain $a \in \overline{LR} = \overline{I} = I$. Therefore, $\alpha \in I \times M = L$, whence $L$ is integrally closed. 

$(2)$ Let us make sure of the `only if' part. Since $L$ is integrally closed, it contains $(0) \times M$. Let $I = \varphi^{-1}(L)$. Then, it is straightforward to check that $L = I \times M$. Hence, by $(1)$, we conclude that $I$ is integrally closed. 
\end{proof}

Let us note the following. 

\begin{lem}\label{12.2}
For each $f = (a, x) \in A$, the following assertions hold true. 
\begin{enumerate}[$(1)$]
\item $f \in W(A)$ if and only if $a \in W(R)$.
\item $\overline{(f)} = \overline{(a,0)A}$ in $A$.
\end{enumerate}
\end{lem}

\begin{proof}
$(1)$ Suppose that $f \in W(A)$ and $a \not\in W(R)$.  Choose $\fkp \in \Ass R$ such that $a \in \fkp$. Let $P = \fkp \times M$. We then have an isomorphism
$$
A_P \cong R_{\fkp} \ltimes M_{\fkp}
$$ 
of $R_{\fkp}$-algebras, which implies $P \in \Ass A$. This is impossible, because $f \in P$. Hence $a \in W(R)$. Conversely, we assume that $f \notin W(A)$ and $a \in W(R)$. Take $P \in \Ass A$ such that $f \in P$. Let us write $P = \fkp \times M$ with $\fkp \in \Spec R$. Then $a \in \fkp$. Since $a \in W(R)$ and $M$ is torsion-free, we see that $a$ is a non-zerodivisor on $M$. Therefore we get $\depth_{R_\fkp} M_{\fkp}>0$, whence $\depth A_P>0$. This makes a contradiction. 

$(2)$ This follows from the fact that $(0, x)$ is nilpotent. 
\end{proof}

We then have the following. 

\begin{prop}\label{12.3}
$\Lambda(A) = \{I \times M \mid I \in \Lambda (R)\}$.
\end{prop}

\begin{proof}
Let $I \in \Lambda(R)$ and write $I = \overline{(a)}$ with $a \in W(R)$. We consider $L = I\times M$. By Lemma \ref{12.1} (1), $L$ is an integrally closed ideal in $A$. Let $f = (a, 0) \in L$. Then $f$ is a non-zerodivisor on $A$ and $(0)\times M \subseteq \overline{fA}\subseteq \overline{L} = L$.  Moreover, we have that $I \times (0) \subseteq IA \subseteq \overline{aA}  = \overline{fA}$. Hence, $L = I \times M \subseteq \overline{fA} \subseteq L$, i.e., $L = \overline{fA} \in \Lambda (A)$. 

Conversely, let $L \in \Lambda(A)$. Since $L$ is integrally closed, we can choose an integrally closed ideal $I$ in $R$ such that $L = I \times M$ and $I = \varphi^{-1}(L)$. 
Let $f \in W(A)$ such that $L = \overline{fA}$. Write $f = (a, x)$ where $a \in I, x \in M$. Then, by Lemma \ref{12.2}, we get $L = \overline{(a, 0)A}$ and $a \in W(R)$. Hence $I = \overline{(a)} \in \Lambda(R)$, which gives a desired equality. 
\end{proof}

Similarly for the weakly Arf property, we have the following. 

\begin{thm}\label{12.4}
The following conditions are equivalent.
\begin{enumerate}[$(1)$]
\item $A = R \ltimes M$ is an Arf ring.
\item $R$ is an Arf ring and $\overline{(a)}\cdot M = aM$ for every $a \in W(R)$.
\item $R$ is an Arf ring and $M$ is an $\overline{R}$-module.
\end{enumerate}
\end{thm}

\begin{proof}
Since $R$ is Noetherian and $M$ is finitely generated as an $R$-module, $A=R \ltimes M$ is a Noetherian ring with $\dim A = \dim R$. Notice that $A$ is semi-local if and only if so is $R$. Moreover, we have $\Max A = \{\m \times M \mid \m \in \Max R\}$, and for each $\m \in \Max R$, the equality $\dim A_{\n} = \dim R_{\m}$ holds, where $\n = \m \times M$. Hence, $A_\n$ is a one-dimensional Cohen-Macaulay local ring if and only if so is $R_\m$.

$(2) \Leftrightarrow (3)$ This follows from Lemma \ref{5.1}.

$(1) \Rightarrow (3)$ Suppose that $A$ is an Arf ring. For each integrally closed ideal $I \in \calF_R$ in $R$, we set $L = I \times M$. Then $L \in \calF_A$ is an integrally closed ideal in $A$, so that $L^2 = f L$ for some $f \in L$. By setting $f =(a, x) \in W(A)$ with $a \in I, x \in M$, we have $a \in W(R)$ and $L = \overline{fA} = \overline{(a, 0)}A$.  Hence $(a, 0)$ is a reduction of $L$, so that $L^2 = aL$. Therefore $I^2 = a I$ and $IM = aM$. In particular, $R$ is an Arf ring. 
We are now going to prove that $M$ is an $\overline{R}$-module.
Let $\varphi \in \overline{R}$ and write $\varphi = x/a$ where $x \in R$ and $a \in W(R)$. Then $x = a\varphi \in a\overline{R} \cap R = \overline{(a)}$. We set $I = \overline{(a)}$ and $L = I \times M$. Then $L$ is the integral closure of $(a, 0)A$. Applying the above argument, we get $IM = aM$, whence $a \varphi M \subseteq a M$. Thus, $\varphi M \subseteq M$, i.e., $\overline{R}\cdot M = M$.

$(3) \Rightarrow (1)$ We assume that $R$ is an Arf ring and $\overline{R}\cdot M = M$. Let $L \in \calF_A$ be an integrally closed ideal of $A$. Choose an integrally closed ideal $I$ in $R$ such that $L = I \times M$. We see that  $I$ contains a non-zerodivisor on $R$. The Arf property of $R$ implies that  $I^2 = a I$ for some $a \in I$. Hence $a \in W(R)$ and $I = \overline{(a)}$. By setting $f= (a, 0) \in L$, we obtain 
\begin{center}
$L^2 = I^2 \times IM$ and $fL = a I \times aM$. 
\end{center}
Since $IM \subseteq (a \overline{R})M = aM$, we have $IM = aM$. Hence $L^2 = fL$, i.e., $A$ is an Arf ring. 
\end{proof}

\begin{cor}
Let $N$ be a finitely generated torsion-free $\overline{R}$-module. Then the following conditions are equivalent.
\begin{enumerate}[$(1)$]
\item $A = R \ltimes N$ is an Arf ring.
\item $R$ is an Arf ring.
\end{enumerate}
\end{cor}



\section{Arf and weakly Arf rings arising from fiber products}\label{sec6}

In this section we investigate the Arf and weakly Arf properties of local rings obtained from fiber products. First, we recall the definition and basic properties of fiber products. Let $R$, $S$, and $T$ be arbitrary commutative rings, and let $f:R \to T$ and $g: S \to T$ be homomorphism of rings. {\it The fiber product $R\times_TS $ of $R$ and $S$ over $T$ with respect to $f$ and $g$} is the set 
$$
A:=R\times_TS = \{(a, b) \in R \times S \mid f(a) = g(b)\}
$$
which forms a subring of $B=R \times S$. We then have a commutative diagram
\[
\xymatrix{
          &
R \ar[rd]^{f}  &  \\
A \ar[rd]_{p_2}\ar[ur]^{p_1} &  & T\\
& S \ar[ur]_{g}
}
\]
of rings, where $p_1 : A \to R, (x, y)\mapsto x$ and $p_2: A \to S, (x, y)\mapsto y$ stand for the projections. Hence we get an exact sequence
$$
0 \longrightarrow A \overset{i}{\longrightarrow} B \overset{\varphi}{\longrightarrow} T
$$
of $A$-modules, where $\varphi = \left[\begin{smallmatrix}
f \\
-g
\end{smallmatrix}\right]$. The map $\varphi$ is surjective if either $f$ or $g$ is surjective. Moreover, if both $f$ and $g$ are surjective, then $T$ is cyclic as an $A$-module, so that $B$ is a module-finite extension over $A$. 
Therefore we have the following (see e.g., \cite{AAM}).

\begin{lem}\label{6.1}
Suppose that $f:R \to T$ and $g: S \to T$ are surjective. Then the following assertions hold true. 
\begin{enumerate}[$(1)$]
\item $A$ is a Noetherian ring if and only if $R$ and $S$ are Noetherian rings.
\item $(A, J)$ is a local ring if and only if  $(R, \m)$ and $(S, \n)$ are local rings. When this is the case, $J = (\m \times \n) \cap A$.
\item If $(R, \m)$ and $(S, \n)$ are Cohen-Macaulay local rings with $\dim R=\dim S = d>0$ and $\depth T\ge d-1$, then $(A, J)$ is a Cohen-Macaulay local ring and $\dim A = d$.
\end{enumerate}
\end{lem}

With this notation the first main result in this section is stated as follows.

\begin{thm}\label{6.2}
Let $(R, \m)$ and $(S, \n)$ be Noetherian local rings with a common residue class field $k=R/\m = S/\n$. Suppose that $\depth R>0$ and $\depth S > 0$. Then the following conditions are equivalent. 
\begin{enumerate}[$(1)$]
\item $A = R\times_k S$ is a weakly Arf ring.
\item $R$ and $S$ are weakly Arf rings. 
\end{enumerate}
\end{thm}

\begin{proof}
Notice that $\depth A >0$. We denote by $J = \m \times \n$ the maximal ideal of $A$. Since $B=R\times S$ is a module-finite birational extension over $A$, we get $\rmQ(A) = \rmQ(R) \times \rmQ(S)$ and $\overline{A} = \overline{R} \times \overline{S}$ as well. 

$(2) \Rightarrow (1)$ Let $K \in \Lambda(A)$. We choose $\alpha \in W(A)$ such that $K = \overline{(\alpha)}$.  May assume $\alpha \in J = \m \times \n$. Let us write $\alpha = (x, y)$ with $x \in \fkm$ and $y \in \fkn$.  Since $\alpha \in W(A)$, we have $x \in W(R)$ and $y \in W(S)$. Set $I_1 = \overline{(x)}$ in $R$ and $I_2 = \overline{(y)}$ in $S$. Then 
$$
K = \overline{(\alpha)} = \alpha \overline{A} \cap A = \left[ (x, y)(\overline{R} \times \overline{S}) \cap B\right] \cap A = \left[(x \overline{R} \cap R) \times (y \overline{S} \cap S)\right] \cap A = I_1 \times I_2
$$
where the last equality follows from the fact that $I_1 \times I_2 \subseteq \m \times \n = J \subseteq A$. Since $R$ and $S$ are weakly Arf, by Theorem \ref{2.4}, we get $I_1^2 = xI_1$ and $I_2^2 = yI_2$. This yields
$$
K^2 = I_1^2 \times I_2^2  = (xI_1) \times (yI_2) = (x, y) (I_1 \times I_2) = \alpha K
$$
so that $A$ is a weakly Arf ring.

$(1) \Rightarrow (2)$ Let $I_1 \in \Lambda (R)$. Set $I_1 = \overline{(x)}$ with $x \in W(R)$. We may assume $x \in \m$. Hence $I_1 \subseteq \m$. As $\depth S>0$, we can choose $y \in \n$ such that $y \in W(S)$. We put $\alpha = (x, y)$. Then $\alpha \in W(A)$. Let $K = \overline{(\alpha)}$ be the integral closure of the ideal $(\alpha)$ in $A$. The weakly Arf property for $A$ leads us to obtain $K^2 = \alpha K$. As in the proof of $(2) \Rightarrow (1)$, we have $K = \overline{\alpha A} =  I_1 \times \overline{(y)}$. Hence the equalities 
$K^2 = I_1^2 \times (\overline{(y)})^2$ and $\alpha K = (x I_1) \times  (y\overline{(y)})$ induce $I_1^2 = xI_1$. Therefore $R$ is weakly Arf. Similarly,  $S$ is a weakly Arf ring. 
\end{proof}

When $R = S$, we can remove the condition $\depth R > 0$. 
For an integer $i$, an ideal $I$ and a module $M$ over a Noetherian ring $R$, we denote by 
${\rmH}^i_I(M)$ the $i$-th local cohomology module of $M$ with respect to $I$.

\begin{cor}\label{6.3}
Let $(R, \m)$ be a Noetherian local ring. Then 
$A = R\times_{R/\m} R$ is weakly Arf if and only if so is $R$. 
\end{cor}

\begin{proof}
By Theorem \ref{6.2}, we may assume $\depth R = 0$, i.e, $R$ is a weakly Arf ring. To prove $\depth A=0$, we assume $\depth A > 0$. Then $\rmH^0_J (A) = (0)$, where $J = \m \times \m$.  The exact sequence 
$0 \to A \to B \to R/\fkm \to 0$ of $A$-modules gives rise to a sequence
$$
0 \to \rmH_J^0(A) \to \rmH_\m^0(R) \times \rmH_\m^0(R) \to R/\fkm.
$$
This makes a contradiction, because $\rmH_\m^0(R) \ne (0)$. Hence $\depth A=0$. Therefore $A$ is a weakly Arf ring. 
\end{proof}

Hence we have the following.

\begin{cor}
Let $(R, \m)$ be a Noetherian local ring. Then $R \times_{R/\mathfrak{m}} R \times_{R/\mathfrak{m}} \cdots \times_{R/\mathfrak{m}} R$ is weakly Arf provided that $R$ is a weakly Arf ring.
\end{cor}

Let us now explore a gluing which is a special kind of fiber product. 

\begin{thm} Let $S$ be a Noetherian semi-local ring. Suppose that $S$ is integrally closed and it contains a field $k$. Let $J$ be an ideal of $S$ and assume the following conditions:
\begin{enumerate}[$(1)$]
\item $J \subseteq \operatorname{J}(S)$,
\item $J \cap \operatorname{W}(S) \ne \emptyset$, and
\item $\operatorname{dim}_k S/J < \infty$,
\end{enumerate}
where $\operatorname{J}(S)$ denotes the Jacobson radical of $S$. 
Then $R= k + J$ is a weakly Arf ring.
\end{thm}

\begin{proof}
Notice that $S$ is a module-finite extension over $R$. Indeed, because of the condition $(3)$, we see that $\ell_R(S/J)<\infty$. It then follows that $S/J$ is a finitely generated $R$-module. The natural surjection $S/J \to S/R$ guarantees that $S/R$ is a finitely generated $R$-module, whence so is $S$. Therefore, $R$ is a Noetherian ring with $\dim R= \dim S$. In addition, we claim that $R$ is a local ring with maximal ideal $J$. In fact, note that $J \in \Max R$. The condition $(1)$ ensures  $J \subseteq M$ for every $M \in \Max S$, so that $J \subseteq M \cap R$. Hence $J=M \cap R$, as claimed. 
By the condition $(2)$, we obtain $\rmQ(R) = \rmQ(S)$ and hence $S = \overline{R}$. Let $I \in \Lambda(R)$. Set $I = \overline{(a)}$ for some $a \in W(R)$. We may assume that $a \in R$ is not a unit in $R$, i.e., $a \in J$. Then $\overline{(a)} = a\overline{R} \cap R = a S$, because $a S\subseteq J \subseteq R$. Therefore
$$
I^2 =(aS)^2 = a(aS) =a I.
$$
Hence $R$ is a weakly Arf ring.
\end{proof}

Let us note some examples. 

\begin{ex}
Let $k[[X,Y]]$ be the formal power series ring over a field $k$. We consider the subring $S = k[[X^4,X^3Y,X^2Y^2, XY^3, Y^4]]$ and ideal $J = (X^4, X^3Y, XY^3, Y^4)$. Then $R= k + J$ is a weakly Arf ring.
\end{ex}

\begin{ex}\label{ex 4.3}
Let $S= k[[X_1, X_2, \ldots, X_d, Y_1, Y_2, \ldots, Y_d]]$~$(d \ge 1)$ be the formal power series ring over a field $k$ and $J = (X_1, X_2, \ldots, X_d) \cap (Y_1, Y_2, \ldots, Y_d)$. Then $R = S/J$ is a weakly Arf ring.
\end{ex}


Next we explore the Arf property for fiber products. 
In what follows, let $(R, \m)$, $(S, \n)$ be Noetherian local rings with a common residue class field $k = R/\m = S/\n$. As in Theorem \ref{6.2}, we assume that $\depth R>0$ and $\depth S>0$. We set $A = R\times_k S \subseteq B = R \times S$. 

To state the second main result, let us begin with the following. 

\begin{lem}\label{6.8}
Let $I$ be a proper ideal of $R$ and let $J$ be a proper ideal of $S$. We set $K = I \times J$ which forms an ideal of $A$. Then $K$ is integrally closed in $A$ if and only if $I$ and $J$ are integrally closed in $R$ and $S$, respectively. 
\end{lem}

\begin{proof}
Note that $K \subseteq \m \times \n \subsetneq A$ and $K$ is an ideal of $B$, so $K$ is an ideal of $A$. If $I$ and $J$ are integrally closed, then so is $K$ in $B$. Hence, $K$ is integrally closed in $A$. Conversely, we assume $K = \overline{K}$ in $A$. Let us choose $x \in \overline{I}$. Then $x \in \m$, so that $(x, 0) \in \m \times \n \subseteq A$. Hence $(x, 0) \in \overline{I \times J} = \overline{K} = K$, which implies $x \in I$. Therefore, $I$ is integrally closed in $R$. Similarly, we have $J = \overline{J}$ in $S$, as desired. 
\end{proof}

We apply Lemma \ref{6.8} to get the following.

\begin{cor}\label{6.9}
Let $K$ be a proper ideal of $A$. Then $K$ is integrally closed in $A$ if and only if there exist a proper integrally closed ideal $I$ of $R$ and a proper integrally closed ideal $J$ of $S$ such that $K = I \times J$. 
\end{cor}

\begin{proof}
We only prove the `only if' part. Suppose that $K$ is integrally closed in $A$. Then $K = \overline{K} = K \overline{A} \cap A$. Since $\overline{A} = \overline{R} \times \overline{S}$, we choose ideals $I_1$ and $J_1$ in $\overline{R}$ and $\overline{S}$ satisfying $K\overline{A} = I_1 \times J_1$. Therefore
\begin{eqnarray*}
K &=& (I_1 \times J_1) \cap A = [(I_1 \times J_1) \cap B] \cap A \\
   &=& [(I_1 \cap R) \times (J_1 \cap S)] \cap A \\
   &=& (I_1 \cap R) \times (J_1 \cap S)
\end{eqnarray*}
where the last equality follows from the fact that $I_1$ and $J_1$ are proper ideals in $\overline{R}$ and $\overline{S}$, respectively. By setting $I = I_1 \cap R$ and $J = J_1 \cap S$, we have $K = I \times J$, and the assertion follows from Lemma \ref{6.8}.
\end{proof}

We are now ready to prove the second main result in this section. 

\begin{thm}\label{6.10}
Let $(R, \m)$, $(S, \n)$ be Noetherian local rings with a common residue class field $k = R/\m = S/\n$. Suppose that $\depth R>0$, $\depth S>0$, and  $\dim R = \dim S = 1$.  Then the following conditions are equivalent.
\begin{enumerate}[$(1)$]
\item $A = R \times_k S$ is an Arf ring.
\item $R$ and $S$ are Arf rings. 
\end{enumerate}
\end{thm}

\begin{proof}
Suppose that $A$ is Arf. Let $I$ be an integrally closed proper ideal in $R$ such that $I \in \calF_R$, i.e., it contains a non-zerodivisor on $R$. We set $K = I \times \n$. Then $K$ is integrally closed in $A$, so that it is stable, i.e., $K^2 = \alpha K$ for some $\alpha \in K$. Write $\alpha = (x, y)$ with $x \in I$ and $y \in \n$. We then have $I^2 = xI$. Hence $R$ is an Arf ring. Similarly, $S$ is Arf. 
On the other hand, suppose that both $R$ and $S$ are Arf rings. Let $K \in \calF_A$ be an integrally closed proper ideal in $A$. By Corollary \ref{6.9}, we may choose integrally closed ideals $I$ in $R$ and $J$ in $S$ such that $K = I \times J$. Since $I$ and $J$ are stable, $K$ is also stable. Hence $A$ is an Arf ring. 
\end{proof}

Closing this section, let us note an example of a weakly Arf ring obtained from the amalgamated duplication. Let $R$ be a commutative ring and $I \in \calF_R$. We set $A = R \bowtie I$ the amalgamated duplication of $R$ with respect to $I$. Notice that $\rmQ(A) = \rmQ(R) \bowtie \rmQ(R)$ and $\overline{A} = \overline{R}  \bowtie  \overline{R}$.

With this notation we have the following. 

\begin{prop}
Suppose that $I$ is an ideal of $\overline{R}$. Then $A= R \bowtie I$ is  weakly Arf if and only if so is $R$.
\end{prop}

\begin{proof}
Since $R$ is a direct summand of $A$ as an $R$-module, by Corollary \ref{3.2} (1), if $A$ is weakly Arf, then so is $R$. Conversely, assume that $R$ is a weakly Arf ring. Let $\alpha, \beta, \gamma \in A$ such that $\alpha \in W(A)$ and $\beta/\alpha, \gamma/\alpha \in \overline{A}$.  Write $\alpha = (a, x)$, $\beta/\alpha = (b, y)$, and $\gamma/\alpha = (c, z)$, where $a \in R$, $x \in I$, and $b, c, y, z \in \overline{R}$. Then
$$
\beta = (a, x)(b, y) = (ab, ay + xb + xy), \quad \gamma = (a, x)(c, z) = (ac, az + xc + xz)
$$
so that $ay, az \in I$, because $I$ is an ideal of $\overline{R}$. As $ab, ac \in a\overline{R}$, the weakly Arf property of $R$ guarantees that $abc = (ab\cdot ac)/a \in R$. Hence
\begin{eqnarray*}
\frac{\beta\gamma}{\alpha} &=& \alpha (b, y)(c,z) = (a, x) (bc, bz + yc + yz) \\
& = & (abc, abz + ayc + ayz + xbc + xbz + xyc + xyz) \in  A
\end{eqnarray*}
which shows that $A$ is a weakly Arf ring. 
\end{proof}

\begin{cor}
Let $R$ be a Noetherian ring. Suppose that $\overline{R}$ is a module-finite extension over $R$. For each $f \in W(R) \cap (R:\overline{R})$, we set $I = f\overline{R}$. Then $A = R \bowtie I$ is weakly Arf if and only if so is $R$. 
\end{cor}


\section{Weakly Arf closures}\label{sec7}

As proved by J. Lipman in \cite{L}, among all the Arf rings between a ring and its integral closure, there is a smallest one,  called {\it the Arf closure}. The reader may refer to \cite{L} for its basic properties. 
In this section we slightly modify this closure, and use it in our investigation of the weakly Arf rings. 


For a Noetherian ring $A$, we define $\calY_A$ to be the set of all weakly Arf rings $B$ such that $B$ is an intermediate ring between $A$ and $\overline{A}$ and is a module-finite extension over the ring $A$. Suppose that $\calY_A \ne \emptyset$. Notice that this assumption is automatically satisfied if $\overline{A}$ is a module-finite extension over $A$. However, the set $\calY_A$ could be non-empty, even though $\overline{A}$ is not finitely generated, as we show next.


\begin{rem}\label{7.1}
Let $S=k[[t]]$ be the formal power series ring over a field $k$, $R_0 = k[[t^4, t^5, t^6]]$ and $R_1 = k[[t^4, t^5, t^6, t^7]]$ be semigroup rings. We consider the idealizations $A = R_0\ltimes S$ and $B=R_1 \ltimes S$. Then $\overline{A} = S \ltimes \rmQ(S)$ is not finitely generated as an $A$-module. However, $B$ is a module-finite extension over $A$ and is a weakly Arf ring; see Corollary \ref{5.8}. 
\end{rem}

Inspired by Lipman's construction of the Arf closures, we set
$$
A_1 = A\left[IA^I \mid I \in \Lambda(A)\right] \subseteq \rmQ(A)
$$
which forms a subring of $\rmQ(A)$, containing $A$. We then have $A_1 \subseteq B$ for every $B \in \calY_A$. Indeed, for each $I \in \Lambda(A)$, we choose $a \in W(A)$ such that $I = \overline{(a)}$. Let $B \in \calY_A$. Set $J = \overline{aB}$. By Theorem \ref{2.4}, we have $J^2 = aJ$, because $J \in \Lambda(B)$. Hence, $B^J = B\left[\frac{J}{a}\right] = \frac{J}{a}$ and $B^J \supseteq A\left[\frac{I}{a}\right] = A^I$, so that $I A^I \subseteq I B^J \subseteq JB^J = J \subseteq B$. Thus $A_1 \subseteq B$, as claimed. 
In particular, $A_1$ is Noetherian.

\begin{defn}\label{7.2}
For each $n \ge 0$, we define recursively
$$
A_n=
\begin{cases}
\ A & (n=0),\\
\ \left[A_{n-1}\right]_1 & (n >0).
\end{cases}
$$
Then, for each $B \in \calY_A$, we have a chain 
$$
A=A_0 \subseteq A_1 \subseteq \cdots \subseteq A_n \subseteq \cdots \subseteq B
$$
of rings. Set ${\rm Arf}(A) = \bigcup_{n \ge 0} A_n$. Notice that ${\rm Arf}(A)$ coincides with the Arf closure when $A$ is a Noetherian semi-local ring such that every localization of $A$ at maximal ideal $M$ is a one-dimensional Cohen-Macaulay local ring and $\overline{A}$ is a finitely generated $A$-module. 
\end{defn}

We then have the following (see \cite[Proposition-Definition 3.1]{L}).

\begin{prop}\label{7.3} 
Let $A$ be a Noetherian ring. Suppose that $\calY_A \ne \emptyset$ $($e.g., $\overline{A}$ is a finitely generated $A$-module$)$. Then ${\rm Arf}(A)$ is a weakly Arf ring, and ${\rm Arf}(A) \subseteq B$  for every $B \in \calY_A$.
\end{prop}

\begin{proof} 
Let $B \in \calY_A$. As $A \subseteq {\rm Arf}(A) \subseteq B$, the ring ${\rm Arf}(A)$ is a finitely generated $A$-module and ${\rm Arf}(A) = A_n$ for sufficiently large $n$. Let $I \in \Lambda(A_n)$. We write $I = \overline{xA_n}$ for some $x \in W(A_n)$. Then $I \left(A_n\right)^I \subseteq A_{n+1} = A_n$. For each $y \in \left(A_n\right)^I$, we have $xy \in x \overline{A_n} \cap A_n = \overline{xA_n} = I$, because $\overline{A_n} = \overline{\left(A_n\right)^I}$. Hence $x \left(A_n\right)^I \subseteq I$. On the other hand, because $xA_n$ is a reduction of $I$, we obtain 
$$
\left(A_n\right)^I = \left(A_n\right)\left[\frac{I}{x}\right] \supseteq \frac{I}{x}
$$
which yields $I \subseteq x \left(A_n\right)^I$. It follows that  $I \left(A_n\right)^I = x \left(A_n\right)^I $. Therefore, $I \left(A_n\right)^I  \subseteq I$, i.e., $I \left(A_n\right)^I =I$. Thus $I$ is stable, and hence ${\rm Arf}(A) = A_n$ is a weakly Arf ring. 
\end{proof}


Although the Arf closures behave well and play one of the central roles in Lipman's paper, the notion is defined in a bit restricted situation. In the following, we define the closure in a more relaxed situation. 

\medskip

In what follows, let $A$ denote an arbitrary commutative ring. We consider 
$$
A_1 = A\left[\frac{yz}{x} ~\middle|~ x \in W(A), y, z \in A \ \text{such that} \ \frac{y}{x}, \frac{z}{x} \in \overline{A} \ \right] \ \ \text{in} \ \ \rmQ(A).
$$
Then $A \subseteq A_1 \subseteq \overline{A}$ is an intermediate ring between $A$ and $\overline{A}$.

\begin{defn}
For each $n \ge 0$, we define recursively
$$
A_n=
\begin{cases}
\ A & (n=0),\\
\ \left[A_{n-1}\right]_1 & (n >0).
\end{cases}
$$
Notice that, for each $n \ge 0$, $\rmQ(A_n) = \rmQ(A)$, $\overline{A_n} = \overline{A}$, and we have a chain 
$$
A=A_0 \subseteq A_1 \subseteq \cdots \subseteq A_n \subseteq \cdots \subseteq \overline{A}
$$
of rings. Set $A^a = \bigcup_{n \ge 0} A_n$ and call it {\it the weakly Arf closure} of $A$.
\end{defn}

\begin{prop}\label{7.5}
The following assertions hold true. 
\begin{enumerate}[$(1)$]
\item $A^a$ is a weakly Arf ring. 
\item For every intermediate ring $A \subseteq B \subseteq \overline{A}$ such that $B$ is weakly Arf, one has $A^a \subseteq B$. In particular, $A^a \subseteq A^*$.
\end{enumerate}
\end{prop}

\begin{proof}
$(1)$ Let $x, y, z \in A^a$ such that $x \in W(A^a)$ and $y/x, z/x \in \overline{A^a}$. For large $n \gg 0$, we have $x, y, z \in A_n$, whence $x \in W(A_n)$ and $y/x, z/x \in \overline{A_n}$. Hence $yz/x \in A_{n+1} \subseteq A^a$, so that $A^a$ is a weakly Arf ring. 

$(2)$ Firstly, we check that $A_n \subseteq B$ for every $n \ge 0$. 
Suppose $n > 0$ and the assertion holds for $n-1$. Let $x, y, z \in A_{n-1}$ such that $x \in W(A_{n-1})$ and $y/x, z/x \in \overline{A_{n-1}}$. Since $A_{n-1} \subseteq B$, we have $x, y, z \in B$. Because $\rmQ(B) = \rmQ(A)$, we have $x \in W(B)$. Moreover, $y/x, z/x \in \overline{B} = \overline{A}$, which yields that $yz/x \in B$. Therefore $A_n = (A_{n-1})_1 \subseteq B$. 
\end{proof}

Hence we have the following.

\begin{cor}\label{C1}
Suppose that $A$ is Noetherian and $\calY_A \ne \emptyset$ $($e.g., $\overline{A}$ is a finitely generated $A$-module$)$. Then
${\rm Arf}(A) = A^a$. 
\end{cor}

\begin{cor}\label{C2}
Let $A$ be a Noetherian ring. Suppose that $\overline{A}$ is a finitely generated $A$-module and one of the following conditions hold:
\begin{enumerate}[$(1)$]
\item $A$ contains an infinite field.
\item $\height_AM \ge 2$ for every $M \in \Max A$.
\end{enumerate}
If $A^a$ satisfies $(S_2)$, then  $A^a = A^*$.
\end{cor}

\begin{proof}
Notice that $A^a$ is Noetherian and weakly Arf. If $A$ contains an infinite field, then so does $A^a$. On the other hand, if we assume $\height_AM \ge 2$ for every $M \in \Max A$, then $\height_{A^a}N \ge 2$ for every $N \in \Max A^a$. Indeed, let $N \in \Max A^a$ and set $M = N \cap A$. By \cite[Theorem 4.2]{L}, we then have 
$$
\dim(R^a)_N = \dim R_M
$$
which yields $\height_{A^a}N \ge 2$. Then, by Corollary \ref{4.4}, $A^a$ is strictly closed in $\overline{A^a}$. Hence $A^a = (A^a)^* \supseteq A^*$. Thus $A^a = A^*$, as desired. 
\end{proof}



Finally we reach the goal of this section by giving  an affirmative answer for the conjecture posed by O. Zariski; see \cite[page 651]{L}. 

\begin{cor}\label{C3}
Let $A$ be a Noetherian semi-local ring such that $A_M$ is a one-dimensional Cohen-Macaulay local ring for every $M \in \Max A$. Suppose that $\overline{A}$ is a finitely generated $A$-module. Then the Arf closure of $A$ coincides with the strict closure of $A$ in $\overline{A}$.
\end{cor}

\begin{proof}
Since ${\rm Arf}(A)$ is an Arf ring, it is strictly closed in $\overline{A}$ by Theorem \ref{Zariski-Lipman}. Hence the assertion follows from Proposition \ref{7.5}.
\end{proof}



\section{Core subalgebras of polynomial rings}\label{sec8}

In this section, we study the weakly Arf property and the weakly Arf closure  for certain subalgebras of the polynomial ring with one indeterminate over a field $k$. The class of subalgebras discussed in this section includes the semigroup rings $k[H]$ of numerical semigroups $H$.

Throughout this section, let $S = k[t]$ denote the polynomial ring over a field $k$, and let $R$ be a $k$-subalgebra of $S$. We say that $R$ is {\it a core of $S$}, if $t^nS \subseteq R$ for some integer $n > 0$. If $R$ is a core of $S$, then 
$$
k[t^n, t^{n+1}, \ldots, t^{2n -1}] \subseteq R \subseteq S,
$$
and a given $k$-subalgebra $R$ of $S$ is a core if and only if $R \supseteq k[H]$ for some numerical semigroup $H$. Thus, once $R$ is a core of $S$, $R$ is a finitely generated $k$-algebra of dimension one, and $S$ is a birational module-finite extension of $R$ with $t^nS \subseteq R:S$. Although typical examples of cores are the semigroup rings $k[H]$ of  numerical semigroups $H$, cores of $S$ do not necessarily arise from the semigroup rings $k[H]$ in general. Let us note the simplest example.

\begin{ex}\label{8.1}
Let $R = k[t^2+t^3] + t^4S$. Then $R \ne k[H]$ for any numerical semigroup $H$.
\end{ex}

The goal of this section is to prove the following.

\begin{thm}\label{8.2}
The following assertions hold true. 
\begin{enumerate}[$(1)$]
\item $R^a = R^*$.
\item The following conditions are equivalent. 
\begin{enumerate}[$(a)$]
\item $R$ is a weakly Arf ring.
\item $R$ is strictly closed in $\overline{R}$.
\item $R_{\m}$ is an Arf ring, where $\m = tS \cap R$. 
\end{enumerate}
\item If $R$ is a weakly Arf ring, then the polynomial ring $T= R[X_1, X_2, \dots, X_n]$ is strictly closed for every $n>0$. In particular, $T$ is a weakly Arf ring. 
\end{enumerate}
\end{thm}

\begin{proof}
$(1)$ Since $R \subseteq R^a \subseteq R^* \subseteq \overline{R}$, we obtain $R^a$ is a core of $S$. By Proposition \ref{7.5}, $R^a$ is a weakly Arf ring, and so is the local ring $\left(R^a\right)_M$, where $M = t S \cap R^a \in \Max R^a$. Besides, every ideal in $\left(R^a\right)_M$ has a principal reduction. This implies $\left(R^a\right)_M$ is an Arf ring. Let $P \in \Spec R^a$ such that $P \ne M$ and $\height_{R^a} P = 1$. Then $\left(R^a\right)_P$ is a DVR (see the proof of Proposition \ref{4.9}). By Theorem \ref{Zariski-Lipman}, we conclude that $R^a$ is strictly closed. Hence $R^a = \left(R^a\right)^* \supseteq R^*$, so that $R^a = R^*$.

$(2)$ If $R$ is strictly closed in $\overline{R}$, then $R$ is a weakly Arf ring. The latter condition is equivalent to saying that $R_{\m}$ is Arf; see Proposition \ref{4.9}. If $R$ is a weakly Arf ring, then by Proposition \ref{7.5} (2), we get $R \subseteq R^* = R^a \subseteq R$. Hence $R$ is strictly closed in $\overline{R}$. 

$(3)$ This follows from Lemma \ref{4.5} and the fact that $R$ is an integral domain.
\end{proof}

Let us note one example. 

\begin{ex}
Let $n \ge 3$ and  $R =k[t^n, t^{n+1}, \ldots, t^{2n-1}]$. Then $R$ is a weakly Arf ring, and hence $R =\left(k[t^n, t^{n+1}, \ldots, t^{2n-2}]\right)^*$.
\end{ex}

\begin{proof}
Let $V=k[[t]]$ be the formal power series ring over $k$. The semigroup ring $A = k[[t^n, t^{n+1}, \ldots, t^{2n-1}]]$ is an Arf ring. See \cite[Example 4.7]{CCGT}
for the proof. In particular, $A$ is a weakly Arf ring. Since $A$ is given by the $\m R_{\m}$-adic completion of the local ring $R_{\m}$, by Proposition \ref{3.3} (or Corollary \ref{3.2} (3)), $R_{\m}$ is weakly Arf, where $\m = tS \cap R$.
 Thus, $R$ is also a weakly Arf ring, so that it is strictly closed in $\overline{R}$. 
Consider $T = k[t^n, t^{n+1}, \ldots, t^{2n-2}]$ in $S=k[t]$. Then the local ring $T_{\n}$ is not an Arf ring, where $\n = tS \cap T$. Hence $T$ is not strictly closed in $\overline{T}$. Therefore, by Theorem \ref{8.2}, we conclude that
$T \subsetneq T^* =T^a \subseteq R^* = R$. This shows $R = T^a =T^*$, since $\ell_T(R/T)=1$. This completes the proof.
\end{proof}


\section{Characterization in terms of the algebras $A^I$}\label{sec9}

As J. Lipman proved in \cite{L}, a ring $A$ is Arf if and only if all the local rings infinitely near to $A$, i.e., the localizations of the blow-ups of $A$, have minimal multiplicity. See \cite[Theorem 2.2]{L} for details. In this section, we investigate a characterization of weakly Arf rings in terms of the algebras $A^I$ of $I$ in $\Lambda(A)$. 




\begin{prop}\label{9.1}
Let $I \subseteq A$ be an ideal of $A$. Suppose that $I = \overline{(a)}$ and $I^2 = aI$ for some $a \in I \cap W(A)$. Set $B = I:I ~(= A^I)$. Then the following conditions are equivalent. 
\begin{enumerate}[$(1)$]
\item $B$ is a weakly Arf ring. 
\item If $x \in I \cap W(A)$, $y, z \in A$ such that $y/x, z/x \in \overline{A}$, then $yz/x \in A$. 
\end{enumerate}
\end{prop}

\begin{proof}
Since $I^2 = aI$, we have 
$$
A \subseteq B = I:I =  \frac{I}{a} \subseteq \overline{A}.
$$
Hence $B$ is a birational extension of $A$ and $\overline{B} = \overline{A}$. 
 
$(2) \Rightarrow (1)$ Let $\alpha, \beta, \gamma \in B$ such that $\alpha \in W(B)$ and $\beta/\alpha, \gamma/\alpha \in \overline{B}$. Write $\alpha = x/a, \beta= y/a$, and $\gamma = z/a$, where $x, y, z \in I$. Then $x \in W(A)$ and 
$$
\frac{y}{x} = \frac{\beta}{\alpha}, \ \  \frac{z}{x} = \frac{\gamma}{\alpha} \ \in \  \overline{B} = \overline{A}.
$$
It then follows that $yz \in x A$. Since $yz/x = x\cdot (yz/x^2) \in x \overline{A}\cap A = \overline{(x)} \subseteq \overline{I} = I$, we get 
$$
\frac{\beta \gamma}{\alpha} = \frac{yz}{ax} = \frac{yz/x}{a}\ \in\ \frac{I}{a} = B
$$
and hence $B$ is a weakly Arf ring. 

$(1) \Rightarrow (2)$ Let $x, y, z \in A$ such that $x \in I \cap W(A)$ and $y/x, z/x \in \overline{A}$. Then $y, z \in x \overline{A} \cap A = \overline{(x)} \subseteq \overline{I} = I$. Hence $x/a, y/a, z/a \in B$. Since $x/a \in W(B)$, the weakly Arf property for $B$ implies $yz/ax \in B$, so that $yz \in x I \subseteq x A$. This completes the proof. 
\end{proof}

As a direct consequence, we have the following.

\begin{cor}\label{9.2}
$A$ is a weakly Arf ring if and only if $A^I$ is a weakly Arf ring for every $I \in \Lambda (A)$.  
\end{cor}

\begin{proof}
Suppose that $A$ is a weakly Arf ring. For each $I \in \Lambda(A)$, if we write $I = \overline{(a)}$ for some $a \in W(A)$, then $I^2 = a I$; see Theorem \ref{2.4}. In particular, $I$ is stable, i.e., $A^I = I:I$. Hence, by Proposition \ref{9.1}, $A^I = I:I$ is a weakly Arf ring. The converse holds, because $A \in \Lambda(A)$. 
\end{proof}

\begin{cor}\label{9.3}
For each $a \in W(A)$, we set $I = \overline{(a)}$. Then the following conditions are equivalent.
\begin{enumerate}[$(1)$]
\item $I^2 = aI$ and $B=I:I$ is a weakly Arf ring.
\item For each $b \in I \cap W(A)$, $J^2 = bJ$ and $J:J$ is a weakly Arf ring, where $J = \overline{(b)}$.  
\end{enumerate}
\end{cor}

\begin{proof}
$(1) \Rightarrow (2)$ Let $b \in I$ be a non-zerodivisor on $A$. We set $J = \overline{(b)}$. For each $y, z \in J$, we have $y/b, z/b \in \overline{A}$, because $J= \overline{(b)} \subseteq b \overline{A}$. Hence $yz/b \in A$ by Proposition \ref{9.1}. Thus
$$
\frac{yz}{b} = b\cdot \frac{yz}{b^2} \in b \overline{A} \cap A = J.
$$
This yields $yz \in bJ$. Therefore, $J^2 = bJ$. Let us make sure of the weakly Arf property for $J:J$. Let $x, y, z \in A$ such that $x \in J \cap W(A)$ and $y/x, z/x \in \overline{A}$. Since $x \in I \cap W(A)$ and $B$ is weakly Arf, we obtain $yz/x \in A$. Again, by Proposition \ref{9.1}, we see that $J:J$ is a weakly Arf ring. 

$(2) \Rightarrow (1)$ If we choose $b= a$, then the implication holds. 
\end{proof}

\begin{notation}\label{9.4}
We now define $\Gamma(A)$ to be the set of all the proper ideals $I$ in $\Lambda(A)$. We denote by $\Max \Lambda(A)$ the set of all the maximal elements in $\Gamma(A)$ with respect to inclusion. 
\end{notation}

Hence we have the following. 

\begin{cor}\label{9.5}
Consider the following conditions:
\begin{enumerate}[$(1)$]
\item $A$ is a weakly Arf ring. 
\item If $M \in \Max \Lambda(A)$, then $M:M$ is a weakly Arf ring and $M^2 = aM$, $M= \overline{(a)}$ for some $a \in M \cap W(A)$. 
\end{enumerate}
Then the implication $(1) \Rightarrow (2)$ holds and the converse holds if $A$ is Noetherian. 
\end{cor}

\begin{proof}
By Corollary \ref{9.2}, we only need to prove the implication $(2) \Rightarrow (1)$. Assume that $A$ is Noetherian. Let $I \in \Lambda(A)$. We write $I = \overline{(a)}$ for some $a \in W(A)$. We will show that $I^2 = a I$. Indeed, note first that we may assume that $I$ is a proper ideal of $A$, i.e., $I \in \Gamma(A)$. As $A$ is Noetherian, we may choose an ideal $M \in \Max \Lambda(A)$ such that $I \subseteq M$. Hence, $M:M$ is a weakly Arf ring and $M^2 = bM$, $M= \overline{(b)}$ for some $b \in M \cap W(A)$. As $I \subseteq M$, we have $a \in M \cap W(A)$, whence $I^2 = a I$ by Corollary \ref{9.3}. Therefore, $A$ is a weakly Arf ring, as desired.
\end{proof}

We explore one example.

\begin{ex}\label{9.6}
Let $k[[X, Y]]$ be the formal power series ring over the field $k = \Bbb Z/(2)$. 
We consider $A = k[[X, Y]]/(XY(X+Y))$. Then $\Max \Lambda(A) = \{(x, y^2), (x^2, y), (x+y, xy)\}$, where $x, y$ denote the images of $X, Y$ in $A$, respectively. Hence, $A$ is a weakly Arf ring, but not an Arf ring. 
\end{ex}

\begin{proof}
We denote by $\m=(x, y)$ the maximal ideal of $A$.
Let us consider the ideals $I_1 =(x, y^2)$, $I_2 =(x^2, y)$, and $I_3 =(x+y, xy)$. Then 
\begin{eqnarray*}
I_1 \!\!&=&\!\! (x) + \m^2 = (x)\oplus ((x+y)y), \\
I_2  \!\!&=&\!\! (y) + \m^2 = (x(x+y))\oplus (y), \\
I_3  \!\!&=&\!\! (x+y) + \m^2 = (x+y)\oplus (xy).
\end{eqnarray*}
Hence, we have the isomorphisms 
\begin{eqnarray*}
I_1 : I_1 \!\!&\cong&\!\! \End_A(I_1) \cong A/((x+y)y) \times A/(x) \\
I_2:I_2 \!\!&\cong&\!\! \End_A(I_2) \cong A/(y) \times A/(x(x+y)) \\ 
I_3 : I_3 \!\!&\cong&\!\!  \End_A(I_3) \cong A/(x+y) \times A/(xy)
\end{eqnarray*}
of $A$-algebras, where $\End_A(I)$ denotes the endomorphism algebra of an ideal $I$ in $A$. The rings $A/(x)$, $A/(y)$, and $A/(x+y)$ are DVRs. Moreover, the multiplicities of $A/((x+y)y)$, $A/(x(x+y))$, and $A/(xy)$ are equal to two; hence they are Arf rings by \cite[Proposition 2.8]{CCGT}. Then, by Proposition \ref{3.8}, $I_i:I_i$ is a weakly Arf ring for every $1 \le i \le 3$.

It is straightforward to check that  $I_1^2 = (x+y^2)I_1$, $I_2^2 = (x^2+y)I_2$, and $I_3^2 = (x+y+ xy)I_3$. Notice that $I_i \cap W(A) \ne \emptyset$ for every $1 \le i \le 3$. Since $A$ is reduced, we get $\overline{A} \cong A/(x) \times A/(y) \times A/(x+y)$ which is a principal ideal ring. This implies the equalities $\overline{I_i} = \overline{I_i}\cdot\overline{A} \cap A = I_i \overline{A} \cap A$. It follows that $I_i$ is integrally closed in $A$ for every $1 \le i \le 3$. 
As $k = \Bbb Z/(2)$, the maximal ideal $\m$ does not have a principal reduction (see e.g., \cite[Example 8.3.2]{SH}, \cite[(14.5)]{HIO}). Hence $\m \not\in \Lambda(A)$, so that $I_i \in \Max \Lambda(A)$, because $\ell_A(A/I_i) = 2$. We then have the following. 

\begin{claim*}
For each $I \in \Gamma(A)$, $I \subseteq I_i$ for some $1 \le i \le 3$. 
\end{claim*}

\begin{proof}[Proof of Claim]
Let $f \in W(A)$. We set $I = \overline{(f)}$. Then $f \in \m$, so that $f = x\alpha + y \beta$ for some $\alpha, \beta \in A$. As $\m^2 \subseteq I_i$, we may assume $f \not\in \m^2$. Hence, without loss of generality, we may also assume that $\alpha = 1 + \xi$ with $\xi \in \m$. If $\beta \in \m$, then $f \in I_1$. Suppose $\beta \not\in \m$. Then $\beta \equiv 1$ mod $\m$. By setting $\beta=1 + \eta$ for some $\eta \in \m$, we have 
$$
f = x + y + x\xi + y\eta \in (x+y) + \m^2 = I_3.
$$
Therefore, $I = \overline{(f)}\subseteq \overline{I_i} = I_i$ for some $1 \le i \le 3$. 
\end{proof}

We conclude that $\Max \Lambda(A) = \{I_i \mid 1 \le i \le 3\}$. In particular, by Corollary \ref{9.5}, $A$ is a weakly Arf ring. However, because $\m$ is not stable, the ring $A$ is not Arf.
\end{proof}

In what follows, we assume that $A$ is a Noetherian ring. Notice that $\Max \Lambda (A) = \emptyset$ if and only if $A = \rmQ(A)$. Moreover, we have the following, which plays a key in our arguments. 


\begin{prop}\label{9.7}
The following conditions are equivalent. 
\begin{enumerate}[$(1)$]
\item $A$ is integrally closed.
\item If $M \in \Max \Lambda(A)$, then $\mu_A(M) = 1$. 
\end{enumerate}
\end{prop}

To prove Proposition \ref{9.7}, we need some auxiliaries. 

\begin{lem}\label{9.8}
Let $a, b \in W(A)$ such that $\overline{(a)} \subseteq \overline{(b)} = (b)$. We write $a = bc$ with $c \in W(A)$ and set $I = \overline{(a)}$, $J = \overline{(c)}$. Then the following assertions hold true. 
\begin{enumerate}[$(1)$]
\item $I = bJ \subseteq J$.
\item $I = J$ if and only if $b$ is a unit in $A$.
\item $I^2 = aI$ if and only if $J^2 = cJ$. 
\end{enumerate}
\end{lem}

\begin{proof}
$(1)$ We have $I = \overline{(a)} = \overline{(b)(c)} \supseteq \overline{(b)}\cdot\overline{(c)} = bJ$. For each $x \in I$, let us write $x = b \varphi$ where $\varphi \in A$. Since $I = \overline{(a)} = a \overline{A} \cap A$, we choose $y \in \overline{A}$ such that $x = a y$. Hence 
$$
b\varphi = x = ay = (bc)y = b(cy).
$$
This shows $\varphi = cy \in c \overline{A} \cap A = J$. Therefore $I  = bJ \subseteq J$, as desired.

$(2)$ If $b$ is a unit in $A$, then $J = bJ$. Hence $I = J$. Conversely, suppose that $I = J$. Then $bJ = J$. We can choose $\xi \in (b)$ satisfying $(1 - \xi) J = (0)$. In particular, $(1-\xi) a = 0$. This shows $b$ is a unit in $A$. 

$(3)$ This follows from the fact that $I^2 = b^2 J^2$ and $aI = (b^2c)J = b^2(cJ)$.
\end{proof}


We are now ready to prove Proposition \ref{9.7}

\begin{proof}[Proof of Proposition \ref{9.7}]

$(1) \Rightarrow (2)$ Suppose that $A$ is integrally closed. Let $I \in \Lambda(A)$. Set $I = \overline{(a)}$ for some $a\in W(A)$. Then $I = a\overline{A} \cap A = aA$. In particular, $\mu_A(M) = 1$ for every $M \in \Max \Lambda(A)$.  

$(2) \Rightarrow (1)$ By Lemma \ref{2.8a}, we will show that the ideal $(a)$ is integrally closed in $A$ for every $a \in W(A)$. Indeed, let us consider the set 
$$
\calS = \left\{\overline{(a)} ~\middle|~ a \in W(A), \ (a) \subsetneq \overline{(a)}\right\}
$$
and assume that $\calS \ne \emptyset$. As $A$ is Noetherian, we may choose a maximal element $I$ in $\calS$ with respect to inclusion. Let us write $I = \overline{(a)}$ for some $a \in W(A)$. Then $I$ is a proper ideal in $A$, so there exists $M \in \Max \Lambda(A)$ such that $I \subseteq M$. Let $b\in W(A)$ such that $M = \overline{(b)}$. Since $\mu_A(M)=1$, we obtain $M = (b)$. 
Therefore 
$$
I = \overline{(a)} \subseteq \overline{(b)} = (b).
$$ 
Let us choose $c \in W(A)$ such that $a = bc$. Set $J = \overline{(c)}$. 
If the ideal $(c)$ is integrally closed, i.e., $J = (c)$, then $I = bJ = (bc) = (a)$. This contradicts the assumption that $I \in \calS$. Hence $J \ne (c)$, so $J = \overline{(c)} \in \calS$. The maximality for $I$ ensures that $I = J$. Hence $b$ is a unit in $A$, which makes a contradiction, because $M = (b)$. Consequently, the ideal $(a)$ in $A$ is integrally closed for every $a \in W(A)$, as wanted.
\end{proof}

Hence, by Proposition \ref{9.7}, there exists $M \in \Max \Lambda(A)$ such that $\mu_A(M) \ge 2$, provided $A$ is not integrally closed. By using this phenomenon, we define the following chain of rings between $A$ and $\overline{A}$. 

\begin{defn}\label{9.9}
We define $A_1$ to be $\overline{A}$ if $A$ is integrally closed. Otherwise, if $A \ne \overline{A}$, we define $A_1 = A^M$, where $M \in \Max \Lambda(A)$ such that $\mu_A(M) \ge 2$. Set $A_0 = A$, and for each $n \ge 1$, define recursively $A_n = \left(A_{n-1}\right)_1$.
\end{defn}

\noindent
By definition, we then have a chain of rings
$$
A = A_0 \subseteq A_1 \subseteq \cdots \subseteq A_n \subseteq \cdots \subseteq \overline{A}
$$
by the algebras $A^M$, where $M \in \Max \Lambda(A)$. Let us remark here that $A_1=A^M$ depends on the choice of $M \in \Max \Lambda(A)$, if $A$ is not integrally closed; see Example \ref{9.15}. Hence, the above chain is not uniquely determined in general. However, if the ring $A$ satisfies certain conditions (e.g., $A$ is a numerical semigroup ring), then the chain is uniquely determined, as we show in Corollary \ref{9.14}. 

Notice that, if $A$ is not integrally closed, then $A \ne A_1$. Indeed, choose $M \in \Max\Lambda(A)$ so that $\mu_A(M) \ge 2$ and $A_1 = A^M$. We set $M = \overline{(a)}$ for some $a \in W(A)$. Then, because $(a)$ is a reduction of $M$, we have 
$$
A_1 = A^M = A\left[\frac{M}{a}\right] \supseteq \frac{M}{a} \supseteq A.
$$
If $A = A_1$, then $M = (a)$, which is impossible, because $\mu_A(M) \ge 2$. Hence, $A \ne A_1$ provided $A$ is not integrally closed. 

\begin{rem}\label{9.10}
For a fixed chain 
$A = A_0 \subseteq A_1 \subseteq \cdots \subseteq A_n \subseteq \cdots \subseteq \overline{A}$, 
we see that $\overline{A}$ is finitely generated as an $A$-module if and only if the chain is stable, i.e., $A_n = A_{n+1}$ for some $n \ge 0$, which is equivalent to saying that $A_n = \overline{A}$ for some $n \ge 0$. 
\end{rem}

For a Noetherian local ring $(A, \m)$, we say that $A$ is {\it quasi-unmixed}, if $\Min \widehat{A} = \Assh \widehat{A}$, i.e., $\dim \widehat{A}/Q = \dim A$ for every $Q \in \Min \widehat{A}$, where $\widehat{A}$ denotes the $\m$-adic completion of $A$. The reader may consult \cite{N, Ratliff, Ratliff2, Ratliff3} for basic properties. 

We also recall the notion of degree for modules. Let $A$ be a Noetherian ring with $\dim A >0$, $X$ a finitely generated $A$-module such that $a X = (0)$ for some $a \in W(A)$. We define
$$
\deg_A X = \sum_{\p \in H_1(A)}\ell_{A_{\p}}(X_\p) \in \Bbb Z
$$
and call it {\it the degree of $X$}, where $H_1(A) =\{\p \in \Spec A \mid \dim A_{\p} = 1\}$. 

With this notation, we have the following, which corresponds to ${\rm (i)} \Leftrightarrow {\rm (iii)}$ of \cite[Theorem 2.2]{L} for Arf rings. 

\begin{thm}\label{9.11}
Let $A$ be a Noetherian ring. Consider the following conditions:
\begin{enumerate}[$(1)$]
\item $A$ is a weakly Arf ring.
\item For every $M \in \Max \Lambda(A)$, $M:M$ is a weakly Arf ring and $M^2 = aM$ for some $a \in M$. 
\item  For every chain $A = A_0 \subseteq A_1 \subseteq \cdots \subseteq A_n \subseteq \cdots \subseteq \overline{A}$ obtained from Definition \ref{9.9}, $A_n$ is a weakly Arf ring for every $n \ge 0$.
\item For every chain $A = A_0 \subseteq A_1 \subseteq \cdots \subseteq A_n \subseteq \cdots \subseteq \overline{A}$ obtained from Definition \ref{9.9}, and for every $n \ge 0$ and $N \in \Max \Lambda (A_n)$, $N^2 = bN$ for some $b \in N$. 
\end{enumerate}
Then the implications $(1) \Leftrightarrow (2) \Leftrightarrow (3) \Rightarrow (4)$ hold. If $\dim A=1$, or $A_{\p}$ is quasi-unmixed for every $\p \in \Spec A$, the implication $(4) \Rightarrow (1)$ holds.  
\end{thm}

\begin{proof}
$(1) \Leftrightarrow (2)$, $(1) \Rightarrow (3)$ See Corollary \ref{9.2} and Corollary \ref{9.5}. 

$(3) \Rightarrow (1)$, $(3) \Rightarrow (4)$ Obvious.

$(4) \Rightarrow (1)$ We may assume $\dim A >0$. We will show that, for every $I \in \Lambda(A)$, there exists $a \in I$ such that $I^2 = aI$. Suppose the contrary and take a counterexample $I \in \Lambda(A)$ such that $\ell = \deg_AA/I$ is as small as possible. Let us now replace $I$ with a maximal element $I \in \Lambda (A)$, satisfying $I^2 \ne aI$ for every $a \in I$ and $\ell = \deg_AA/I$.  We then have $\ell \ne 0$ and $I \subsetneq A$. Let $M \in \Max \Lambda(A)$ such that $I \subseteq M$ and write $I = \overline{(a)}$, $M = \overline{(b)}$ with $a, b \in W(A)$. Then we see that $M \ne (b)$. In fact, suppose that $M = (b)$. Then, since $\overline{(a)} \subseteq \overline{(b)} =(b)$, if we write $a = bc$ with $c \in W(A)$, then $I = bJ$, where $J = \overline{(c)}$. As $I^2 \ne a I$, we have $J^2 \ne c J$ by Lemma \ref{9.8}. The canonical surjection $A/I \to A/J$ shows $\ell \ge \deg_AA/J$, and by the minimality of $\ell$, we have $\ell = \deg_AA/J$. Hence, the maximality for $I$ guarantees that $I = J$, so that $b$ is a unit in $A$. This makes a contradiction. Thus $M \ne (b)$, as claimed. 
Therefore, because $\mu_A(M) \ge 2$, we can choose $A_1$ to be the algebra $A^M$. We set $B = A_1 = A^M$ and $L = \frac{I}{b}$ in $\rmQ(A)$. By \cite[Lemma 2.3]{L}, we have $L \in \Lambda(B)$ and $L = \overline{(\frac{a}{b})}$.
Look at the isomorphism $B/L =\frac{M}{b}/\frac{I}{b} \cong M/I$ of $A$-modules. We then have an exact sequence
$$
0 \to B/L \to A/I \to A/M \to 0
$$
of $A$-modules, which yields $\deg_AB/L < \deg_AA/I$, because $M \ne A$.

\begin{claim*}
If $\dim A=1$, or $A_{\p}$ is quasi-unmixed for every $\p \in \Spec A$, then 
$$
\deg_BB/L \le \deg_AB/L.
$$
\end{claim*}


Let us note the following result which is essentially due to L. J. Ratliff, Jr \cite{Ratliff}. 


\begin{lem}[cf. {\cite[THEOREM 3.8]{Ratliff}}]\label{9.12}
Suppose that $A_{\p}$ is quasi-unmixed for every $\p \in \Spec A$. Then, for every intermediate ring $B$ between $A$ and $\overline{A}$ and for every $P \in \Spec B$, we have 
$$
\dim B_P = \dim A_{\p}
$$ 
where $\p = P \cap A$.
\end{lem}

\begin{proof}
Notice that $B_\p$ is integral over $A_\p$ and $\rmQ(A_\p) = \rmQ(B_\p)$. We may assume that $(A, \m)$ is a local ring with $\p = \m$. Hence $P \in \Max B$. Set $M = P$ and $d = \height_BM$. Choose a maximal chain
$$
P_0 \subsetneq P_1 \subsetneq \cdots \subsetneq P_d = M
$$
of prime ideals in $B$. By setting $\p_i = P_i \cap A$, we have a chain
$$
\p_0 \subseteq \p_1 \subseteq \cdots \subseteq \p_d = \m
$$
of prime ideals in $A$. As $P_0 \in \Min B$, we get $P_0\rmQ(B) = \p_0\rmQ(A)$, which implies  $\p_0 \in \Min A$. Note that $B/P_0$ is an integral extension over $A/{\p_0}$ and 
$$
(0) \subsetneq {P_1}/{P_0} \subsetneq {P_2}/{P_1} \subsetneq \cdots \subsetneq {P_d}/{P_0} = M/{P_0}
$$
is a maximal chain in $\Spec B/P_0$. Thus, by \cite[(34.6) Theorem]{N}, we get that
$$
(0) \subseteq {\p_1}/{\p_0} \subseteq {\p_2}/{\p_1} \subseteq \cdots \subseteq {\p_d}/{\p_0} = \m/{\p_0}
$$
is a maximal chain in $\Spec A/\p_0$, whence $d = \dim A/\p_0$. Moreover, because $A$ is quasi-unmixed, we conclude that $\dim A = \dim A/\p_0 + \dim A_{\p_0} = d + 0 = d$; see \cite[(34.5)]{N}. This completes the proof.
\end{proof}

We are now ready to prove the claim.

\begin{proof}[Proof of Claim]
Firstly, suppose that $A_{\p}$ is quasi-unmixed for every $\p \in \Spec A$. Then
\begin{eqnarray*}
\deg_AB/L  &=& \sum_{\p \in H_1(A)}\ell_{A_{\p}}(B_{\p}/LB_{\p}) \\ 
                  &\ge& \sum_{\p \in H_1(A)}\ell_{B_{\p}}(B_{\p}/LB_{\p}) \\
                  &=& \sum_{\p \in H_1(A)} \sum_{M \in \Max B_{\p}}\ell_{(B_{\p})_M}((B_{\p}/LB_{\p})_M) \\
                  &=& \sum_{\p \in H_1(A)}\left(\sum_{Q\in H_1(B), \ Q\cap A = \p} \ell_{B_Q}(B_Q/LB_Q)\right) \\
                  &=& \sum_{Q \in H_1(B)} \ell_{B_Q}(B_Q/LB_Q)  \\
                  &=& \deg_BB/L
\end{eqnarray*}
as wanted. If $\dim A=1$, then $\deg_AB/L = \ell_A(B/L) \le \ell_B(B/L) = \deg_BB/L$.
\end{proof}
Therefore, we have 
$$
\deg_BB/L \le \deg_AB/L < \deg_AA/I = \ell
$$
and the minimality for $\ell$ guarantees that 
$$
L^2 = \left(\frac{a}{b}\right)L.
$$
Hence $I^2 = a I$, which makes a contradiction. Thus, for every $I \in \Lambda(A)$, there exists $a \in I$ such that $I^2 = aI$. Therefore, $A$ is a weakly Arf ring. This completes the proof.
\end{proof}

In the rest of this section, we discuss the problem of when the chains defined in Definition \ref{9.9} are uniquely determined. In what follows, let $S = k[t]$ be the polynomial ring over a field $k$, and let $R$ be a core of $S$, i.e., $k \subseteq R \subseteq S$ is an intermediate ring and $t^cS \subseteq R$ for some $c>0$. Then $t^cS = \overline{(t^c)} \in \Lambda(R)$ and $\overline{(t^c)} \subsetneq R$.

Let $I \in \Lambda(R)$ such that $I \ne R$. Set $I = \overline{(a)}$ with $a \in W(R)$. 
We now write $IS = \varphi S$ where $\varphi = t^n f$, $n \ge 0$, $f \in S$, and $f(0)=1$. Since $I = IS \cap R$, we then have $(0) \ne IS \subsetneq S$, and 
$$
I = \varphi S \cap R = (t^n S \cap R) \cap (f S \cap R) =   (t^n S \cap R) \cdot (f S \cap R). 
$$
See \cite[Lemma 3.1]{EGMY} for a proof. 

With this notation we have the following.

\begin{prop}\label{9.13}
The following assertions hold true. 
\begin{enumerate}[$(1)$]
\item If $n = 0$, then $I = (a) = (f)$. 
\item If $n > 0$ and $R = k[H]$, then $I \subseteq t^nS\cap R \subseteq R_+ = \overline{(t^e)}$  
\end{enumerate}
where $H$ denotes a numerical semigroup, $e = \min (H\setminus \{0\})$ stands for the multiplicity of $H$, and $R_+ = tS \cap R$. 
\end{prop}

\begin{proof}
$(1)$ Suppose $n = 0$. Then $I = fS \cap R$, $f \not\in k$, and $fS = IS = aS$, so that $f = x a$ for some $x \in k\setminus \{0\}$. Thus $f \in R$ and hence $I = (f) = (a)$ by \cite[Theorem 2.4 (2)]{EGMY}.

$(2)$ Since $n > 0$, we have $I \subset t^n \cap R$. If $R = k[H]$ is a semigroup ring of $H$, then
$$
t^nS \cap R \subseteq R \subseteq tS \cap R = R_+ = t^e S \cap R = \overline{(t^e)}
$$
as desired. In particular, $I \subseteq R_+ \in \Max\Lambda(R)$. 
\end{proof}

Consequently, we have the following. 

\begin{cor}\label{9.14}
Let $R = k[H]$ be the semigroup ring of a numerical semigroup $H$, $e = \min (H\setminus \{0\})$, and $R_+ = tS \cap R$. Then the following assertions hold true. 
\begin{enumerate}[$(1)$]
\item $R_+ = \overline{(t^e)} \in \Max \Lambda(R)$, and $\mu_R(R_+) = 1$ if and only if $e=1$.  
\item For every $I \in \Max\Lambda(R)$, we have $I = R_+$, or $\mu_R(I) = 1$. 
\end{enumerate} 
\end{cor}

Therefore, if $R$ is not integrally closed, i.e., $\mu_R(R_+) \ge 2$, we have
$$
R_1 = R^{R_+} = R\left[\frac{R_+}{t^e}\right] = k\left[t^e, t^{a_1-e}, t^{a_2-e}, \ldots, t^{a_{\ell}-e}\right]
$$
because $(t^e)$ is a reduction of $R_+ =(t^{a_1}, t^{a_2}, \ldots, t^{a_{\ell}})$, where $H = \left< a_1, a_2, \ldots, a_{\ell}\right>$ for $\ell>0$ and $0 < a_1, a_2, \ldots, a_{\ell}\in \Bbb Z$ such that $\gcd (a_1, a_2, \ldots, a_{\ell}) = 1$. Hence, $R_1$ is uniquely determined by $R$ and it forms a numerical semigroup ring.  However, $R_1$ is not necessarily unique in general, which we will show in the following example. 

For an $\m$-primary ideal $I$ in a Noetherian local ring $(A, \m)$, recall that $\rme^0_{I}(A)$ denotes the multiplicity of $A$ with respect to $I$.

\begin{ex}\label{9.15}
Let $\ell \ge 2$ and $R = k[t^{\ell} + t^{\ell + 1}] + t^{\ell + 2} S$ in $S =k[t]$. We set $M = tS \cap R$. Then $R$ is a core of $S$ and the following assertions hold true. 
\begin{enumerate}[$(1)$]
\item $t^q \not\in R$ for every $1 \le q \le \ell +1$. 
\item $R:S = t^{\ell +2}S$ and $M \not\in \Lambda(R)$. 
\item Let $I = t^{\ell+2}S$. Then $I \in \Max\Lambda(R)$, $\mu_R(I)>1$, and $R_1 = R^I=S$. 
\item $R$ is strictly closed in $S=\overline{R}$. 
\item Let $a =t^{\ell} + t^{\ell + 1}$ and $I = \overline{(a)}$. Then $I \in \Max\Lambda(R)$, $\mu_R(I)>1$, and $R_1 = R^I = k[t^2, t^3]$. 
\end{enumerate}
\end{ex}
 
\begin{proof}
Notice that $R \ne k[H]$ for any numerical semigroup $H$ and $t^{\ell + 1} \not\in R$. 

$(1)$ Suppose that $t^q \in R$ for some $1 \le q \le \ell$. Then 
$$
t^q = c_1(t^{\ell} + t^{\ell + 1}) + c_2(t^{\ell} + t^{\ell + 1})^2 + \cdots + c_n(t^{\ell} + t^{\ell + 1})^n + t^{\ell + 2} \xi
$$
where $c_i \in k$, $n \gg 0$, and $\xi \in S$. For each $2 \le i \le n$, because $\ell +2  \le 2 \ell \le i\ell$, we get 
$$
t^q = c_1 t^{\ell} +c_1 t^{\ell + 1} + t^{\ell+2} \eta
$$
for some $\eta \in S$. Hence $q = \ell$ and $c_1 = 1$, which make a contradiction. Therefore, $t^q \not\in R$ for every $1 \le q \le \ell +1$. 

$(2)$ Let us write $R:S = t^qS$, where $0 \le q \le \ell + 2$. By $(1)$, we have $q = 0$, or $q = \ell + 2$. If $q = 0$, then $R = S$.  This implies $t \in R$, which is absurd. Hence $q = \ell + 2$, so that $R:S = t^{\ell +2}S$. Suppose that $M \in \Lambda(R)$. Let us write $M = \overline{(b)}$, where $ 0 \ne b \in R$ and $b$ is not a unit in $R$. Then, since $MS = bS$, we get $t^{\ell} = \alpha b$ for some $\alpha \in k \setminus \{0\}$. Thus $t^{\ell}\in R$. This is impossible. Hence, $M \not\in \Lambda(R)$, as claimed. 

$(3)$ We set $c = \ell + 2$ and $I = t^cS$. Then $I = \overline{t^c}$. Since $I \ne (t^c)$, we have $\mu_R(I) >1$. In particular, $I \subsetneq R$. Choose $J \in \Max\Lambda(R)$ such that $I \subseteq J$. By setting $J = \overline{(b)}$ with $b \in R \setminus k$, we then have 
$$
t^c = t^{\ell + 2} \in J \subseteq bS.
$$
Choose $\xi \in S$ such that $t^{\ell + 2} = b  \xi$. We may assume $b =  t^q$ for some $1 \le q \le \ell +2$. Since $t^q =b \in R$, we get $q = \ell +2$. Hence $I = J \in \Max \Lambda(R)$. Moreover, because $(t^c)$ is a reduction of $I$, we conclude that
$$
R_1 = R^I = R\left[\frac{I}{t^c}\right] = S
$$
as desired.

$(4)$ Let $a = t^{\ell} + t^{\ell +1}$ and $f = 1 + t$. We set $N = tS$ and consider $A = R_M \subseteq V= S_N$. Since $M = (a) + t^{\ell + 2} S$, we have $MS = t^{\ell}((f, t^2)S) = t^{\ell}S$.  Let $\m = MA$ and $\n = NV$. Then $\m V = t^{\ell}V = a V$ which shows that $Q=aA$ is a reduction of $\m$. Hence
$$
\rme^0_{\m}(A) = \rme^0_Q(A) = \ell_A(V/QV) = \ell_A(V/\m V)= \ell_A(V/t^{\ell}V) = \ell.
$$
We then have $V  =\sum_{i = 0}^{\ell-1}At^i$. Let $\fka = t^{\ell+2}V$. Since $\m = Q + \fka$, we have $\m^2 = Q\m + \fka^2$. Now, because $\fka^2 = t^{2\ell + 4}V$, $a = t^{\ell}f$, and $f$ is a unit in $V$, we get 
$$
\frac{\fka^2}{a} = \frac{\fka^2}{t^{\ell}} = t^{\ell + 4}V \subseteq \m
$$
whence $\fka^2 \subseteq Q\m$. Therefore $\m^2 = Q\m$, which yields that
$$
B  = A^\m = \frac{\m}{a} = \frac{Q+\fka}{a} = A + t^2V = k + t^2 V
$$
and the Jacobson ideal $J$ of $B$ is $t^2V$. Hence $B^J = V$. Consequently $A$ is an Arf ring by \cite[Theorem 2.2]{L}. Therefore $ R = R^*$. 
 
$(5)$ Let $a = t^{\ell} + t^{\ell +1}$ and $f = 1 + t$. Set $I = \overline{(a)}$. Then $I \ne R$ and $I \in \Lambda (R)$. Since $IS = aS$, we have that
$$
I = a S \cap R = (t^{\ell} S \cap R){\cdot}(fS \cap R) = MK.
$$
where $K = fS \cap R \in \Max R$. Let $J \in \Max\Lambda(R)$ such that $I \subseteq J$. We write $J = \overline{(b)}$, where $b \in R \setminus k$. Since $a \in J \subseteq bS$, we can write $a = b\xi$ for some $\xi \in S$. Hence 
$$
b= \alpha{\cdot} t^{q_1}{\cdot} (1+t)^{q_2}
$$
where $\alpha \in k\setminus\{0\}$, $0 \le q_1 \le \ell$, and $q_2 \in \{0, 1\}$. If $q_2 =0$, then $t^{q_1} \in R$. Thus $q_1=0$ by assertion $(1)$, so that $b \in k$. This is impossible. Therefore, $q_2 =1$. If $q_1 =0$, then $f \in R$. Hence $t\in R$, a contradiction. Thus $q_1 > 0$ and 
$$
J = (t^{q_1}S \cap R){\cdot}(fS \cap R) = MK = I
$$
which shows $I \in \Max \Lambda(R)$. In particular, $I^2 = a I$, because $R$ is weakly Arf. Moreover, because $at^2 \in a S \cap R = \overline{(a)}=I$ and $t^2 \not\in R$, we obtain that $I = \overline{(a)}\ne (a)$. Hence $\mu_R(I)>1$. 
 Because $R^I = I:I$ and $t^{\ell + 2}S \subseteq R$, it is straightforward to check that $t^2, t^3 \in R^I= I:I$. Hence $k[t^2, t^3] \subseteq R^I$. If $R^I = S$, then $t \in R^I$. Thus $t^{\ell +1} +t^{\ell +2} = at \in I R^I \subseteq I \subseteq R$. As $t^{\ell + 2} \in R$, we have $t^{\ell+1} \in R$. This is impossible. Therefore $R^I = k[t^2, t^3]$, as desired. 
\end{proof}


\section{Arf rings versus weakly Arf rings}\label{sec10}

By definition, when $A$ is a Cohen-Macaulay local ring with $\dim A=1$ possessing an infinite residue class field, the ring $A$ is Arf if and only if it is weakly Arf. However, as we have shown in Example \ref{9.6}, a weakly Arf ring is not necessarily Arf. In this section we delve into this example.

\begin{thm}\label{10.1}
Let $(A, \m)$ be a reduced Noetherian local ring. Suppose that the following conditions hold, where $\Ass A =\{P_1, P_2, \ldots, P_n\}~(n >1)$.
\begin{enumerate}[$(1)$]
\item $P_i + P_j = \m$ for every $1 \le i, j \le n$ such that $i \ne j$. 
\item $A/{P_i}$ is integrally closed for every $1 \le i \le n$. 
\item $A$ contains a field $k$ such that the composite map $k \to A \to A/\m$ is bijective. 
\end{enumerate}
Then $A^* = k + \m \overline{A}$. In particular, $\overline{A}/{A^*}$ is a vector space over $A/\m$. 
\end{thm}

\begin{proof}
Since $A$ is reduced, we have an injection
$$
A \overset{\varphi}{\to} \overline{A} = A/P_1 \times A/P_2 \times \cdots \times A/P_n, \ \ a \mapsto (\overline{a}, \overline{a}, \ldots, \overline{a}).
$$
For each $1 \le i \le n$, we set $e_i = (0, \ldots, 0, 1, 0, \ldots, 0) \in \overline{A}$. We then have 
$$
\overline{A} = \sum_{i=1}^n Ae_i = \bigoplus_{i=1}^n Ae_i 
$$
whence 
$$
\overline{A}\otimes_A \overline{A} = \sum_{i, j} A(e_i \otimes e_j) = \bigoplus_{i, j}  A(e_i \otimes e_j). 
$$ 
Let $\alpha \in \overline{A}$ and write $\alpha = (\overline{a_1}, \overline{a_2}, \ldots, \overline{a_n})$ with $a_i \in A$. We have 
$$
\alpha \otimes 1 = \sum_{i, j}a_i(e_i\otimes e_j), \ \ \  1 \otimes \alpha = \sum_{i, j}a_j(e_i\otimes e_j)
$$
so that $\alpha \otimes 1 = 1 \otimes \alpha$ if and only if $a_i(e_i\otimes e_j) = a_j(e_i\otimes e_j)$ for every $1 \le i, j \le n$ in $Ae_i \otimes_A Ae_j$. The latter condition is equivalent to $a_i \equiv a_j$ mod $\m$ for every $1 \le i, j \le n$ such that $i \ne j$, because $Ae_i \otimes_A Ae_j = A/P_i \otimes_A A/P_j \cong A/\m$. That is, $a_1 \equiv a_2  \equiv \cdots \equiv a_n$ mod $\m$. Therefore, since $k \cong A/\m$, we conclude that $A^* = k + \m \overline{A}$. In particular $\m(\overline{A}/A^*) =(0)$. 
\end{proof}

As a corollary, we get the following.

\begin{cor}\label{10.2}
Under the same hypothesis as in Theorem $\ref{10.1}$, $A$ is strictly closed in $\overline{A}$ if and only if $\m \overline{A} = \m$. When this is the case, $A$ is an Arf ring, provided that $A$ is a Cohen-Macaulay local ring with $\dim A=1$.
\end{cor}

\begin{proof}
Notice that $A^* = k + \m \overline{A}$. Hence $A = A^*$ if and only if $\m \overline{A}\subseteq A$, i.e., $\m \overline{A} = \m$. This condition is equivalent to $\m:\m = \overline{A}$. If $\dim A = 1$ and $A$ is Cohen-Macaulay, then $A$ is an Arf ring, because $\left(\overline{A}\right)_M$ is a DVR for every $M \in \Max \overline{A}$ and $\overline{A}$ is only the blow-up of $A$; see \cite[Theorem 2.2]{L}.
\end{proof}

Let us revisit Example \ref{9.6}.

\begin{ex}\label{10.3}
Let $B= k[[X, Y]]$ be the formal power series ring over a field $k$ and set $A = B/(XY(X+Y))$. Then we have the following.
\begin{enumerate}[$(1)$]
\item $A$ is a Cohen-Macaulay local ring with $\dim A =1$,
\item $A$ is reduced, and
\item $\Ass A =\{(x), (y), (x+y)\}$, where $x, y$ denote the images of $X, Y$ in $A$, respectively. 
\end{enumerate}
Hence, $A^* = k+\m \overline{A}$. 
\end{ex}

Moreover, we have the following, which gives another proof for the fact that a weakly Arf ring $A$ is not necessarily Arf, even though $A$ is a one-dimensional Cohen-Macaulay local ring.

\begin{thm}\label{10.4}
Let $B= k[[X, Y]]$ be the formal power series ring over a field $k$ and set $A = B/(XY(X+Y))$. Then, for every integrally closed $\m$-primary ideal $I$ of $A$, we have $I = \m$, or $I^2 = a I$ for some $a \in I$, where $\m$ denotes the maximal ideal of $A$. 
\end{thm}

\begin{proof}
Let $x, y$ be the images of $X$, $Y$ in $A$, respectively.  Suppose the contrary. We choose an integrally closed $\m$-primary ideal $I \in \calF_A$ such that $I \ne \m$, and $I^2 \ne a I$ for every $a \in I$. We then have $I \subsetneq \m$, $\mu_A(I) \ge 2$, and $I = I \overline{A} \cap A$. First, we prove the following. 

 \begin{claim*}
$A:\overline{A} = \m^2$.
 \end{claim*}
 
 \begin{proof}
 Remember that $\overline{A} = A/(x) \times A/(y) \times A/(x+y)$. For each $\alpha \in A$, we have $\alpha \overline{A} \subseteq A$ if and only if there exist $a, b, c \in A$ such that 
\begin{eqnarray*}
&&\alpha - a \in (x), \ \ a \in (y) \cap (x+y) = (y(x+y)), \\
&&\alpha - b \in (y), \ \ b \in (x) \cap (x+y) = (x(x+y)), \\
&&\alpha- c \in (x+y), \ \ \text{and} \ \ c \in (x) \cap (y) =(xy).
\end{eqnarray*}
The latter condition is equivalent to $\alpha \in (x, y^2) \cap (y,x^2) \cap (xy, x+y)=\m^2$, as claimed. 
\end{proof}
 
 Hence $I \not\subseteq \m^2$. Indeed, if $I \subseteq \m^2$, then $I \overline{A} \subseteq A$. Since $\overline{A}$ is a principal ideal ring, there exists $\alpha \in I$ such that $I = I \overline{A} = \alpha \overline{A}$. Thus 
$$
I^2 = (\alpha \overline{A})^2 = \alpha(\alpha \overline{A}) = \alpha I
$$
which makes a contradiction. Therefore $I \not\subseteq \m^2$. Choose $f \in I \setminus \m^2$ such that $f \in W(A)$. Then $f \in \m \setminus \m^2$, so that $I/(f)$ is a non-zero ideal in an Artinian local ring $A/(f)$ and $\m/(f)$ is principal. Therefore, $\mu_A(I) = 2$. By setting $f = x \alpha + y \beta$ with $\alpha, \beta \in A$, because $f \not\in \m^2$, we may assume $\alpha \not\in \m$. Set $f' = x + \beta\alpha^{-1}y$. Since $fA = f'A$, we may also assume that $f = x + \beta y$ for some $\beta \in A$. Remember that $(x) \in \Ass A$ and $f$ is a non-zerodivisor on $A$. This implies $f \not\in (x)$, and hence $\beta \not\in (x)$. Therefore, because $A/(x)$ is a DVR with maximal ideal $\m/(x) = (\overline{y})$, we have
$$
\overline{\beta} = \overline{\tau} \cdot\overline{y}^{\ell} \quad \text{in} \ \ A/(x)
$$
where $\tau \in A$ is a unit in $A$ and $\ell \ge 0$. Hence 
$
\beta = \tau y^{\ell} + x \eta
$
for some $\eta \in A$. This yields the equalities
$$
f = x + \beta y = x + (\tau y^{\ell} + x \eta)y = x (1 + y \eta) + \tau y^{\ell} = (1 + y \eta)(x + \tau'y^{\ell+1})
$$
where $\tau' = \tau(1 + y \eta)^{-1} \in A$. If we set $f' = x + \tau'y^{\ell+1}$, then $fA = f'A$. Hence, without loss of generality, we may assume 
$$ 
f = x + \varepsilon y^{\ell} \quad \text{for some} \ \ \varepsilon \not\in \m, \ \ell>0. 
$$
As $\mu_A(I)=2$, we can write $I = (f, \xi)$ for some $\xi \in A$. 
Again, because $\overline{\xi}$ is a non-zero element in $A/(f)$, we obtain 
$$
\overline{\xi} = \overline{\tau}\cdot \overline{y}^n \quad \text{in} \ \ A/(f)
$$ 
where $\tau \in A$ is a unit in $A$ and $n > 0$. Thus $I = (f, \xi) = (f, y^n)$.  Hence
\begin{center}
$I  = (x + \varepsilon y^{\ell}, y^n)$, where $\varepsilon \in A$ is a unit in $A$ and $\ell, n >0$. 
\end{center}
We now assume that $n \le \ell$. Then $I = (x, y^n)$, so that $I^2 = (x^2, xy^n, y^{2n})$. Note that $n \ge 2$, because $I \ne \m$. Since $xy^n = xy^2\cdot y^{n-2} = -x^2y \cdot y^{n-2} \in (x^2)$, we get $I^2 = (x^2, y^{2n})$. By \cite[Theorem]{Eakin-Sathaye}, $I^2 = a I$ for some $a \in I$ (actually, $I^2 =(x+y^n) I$). This is impossible. Hence $n > \ell~ (\ge 1)$. Remember that $I = I \overline{A} \cap A$ and $\overline{A} = A/(x) \times A/(y) \times A/(x+y)$. We then have
$$
I \overline{A} = (\overline{y}^{\ell}) \oplus (\overline{x}) \oplus (\overline{y} \cdot (1-\overline{\varepsilon} \cdot \overline{y}^{\ell -1}))
$$
so that, if $\ell >1$, then $(1-\overline{\varepsilon} \cdot \overline{y}^{\ell -1})$ is a unit in $A/(x+y)$ and 
$$
I \overline{A} = (\overline{y}^{\ell}) \oplus (\overline{x}) \oplus (\overline{y})
$$
which forms a principal ideal in $\overline{A}$ generated by $(\overline{y}^{\ell}, \overline{x}, \overline{y}) \in \overline{A}$. Hence, for $\alpha \in A$, we have $\alpha \in I$ if and only if 
$$
(\overline{\alpha}, \overline{\alpha}, \overline{\alpha}) = (\overline{y}^{\ell}, \overline{x}, \overline{y})  \cdot (\overline{a}, \overline{b}, \overline{c}) \quad \text{in} \ \ \overline{A} = A/(x) \times A/(y) \times A/(x+y)
$$
for some $a, b, c \in A$. The latter condition is equivalent to saying that 
$$
\alpha \in (x, y^{\ell}) \cap (x, y) \cap (x+y, y) = (x, y^{\ell})
$$
whence $I =  (x, y^{\ell})$. Hence $I^2 = (x^2, xy^{\ell}, y^{2\ell})$, and because $\ell \ge 2$, we conclude that $I^2 = (x^2, y^{2\ell})$. Therefore, by \cite[Theorem]{Eakin-Sathaye}, we get $I^2 = a I$ for some $a\in I$ (in this case, $I^2 = (x+y^{\ell}) I$). This makes a contradiction. Thus $\ell = 1$, and hence
\begin{center}
$I  = (x + \varepsilon y, y^n)$, where $\varepsilon \in A$ is a unit in $A$, $n > 1$, and $f = x+\varepsilon y \in W(A)$. 
\end{center}
Now, we have 
$$
I \overline{A} = (\overline{y}) \oplus (\overline{x}) \oplus ((\overline{\varepsilon -1}) \cdot \overline{y}, \ \overline{y}^n).
$$
Because $(x+y) \in \Ass A$ and $f \in W(A)$, we have $\overline{f} = (\overline{\varepsilon -1}) \cdot \overline{y}$ is a non-zero element in $A/(x+y)$. In particular, $\varepsilon -1 \not\in (x+y)$. If $\varepsilon -1 \not\in \m$, then 
$$
I \overline{A} = (\overline{y}) \oplus (\overline{x}) \oplus (\overline{y}).
$$
By the same argument as above, we get $I = (x, y)\cap (x, y) \cap (x+y, y) = \m$. This is impossible. Hence $\varepsilon -1 \in \m$. Since $A/(x+y)$ is a DVR with maximal ideal $\m/(x+y)$, we can write
$$
\overline{\varepsilon-1} = \overline{\tau} \cdot\overline{y}^{m} \quad \text{in} \ \ A/(x+y)
$$
where $\tau \in A$ is a unit in $A$ and $m > 0$. Hence, $\varepsilon -1 = \tau y^{m} + (x+y)h$ for some $h \in A$. Then $\varepsilon = 1 + (x+y)h + \tau y^{m}$. Hence, by setting $\tau' = \tau(1 + yh)^{-1}$, we have 
\begin{eqnarray*}
I &=& (x + \varepsilon y, y^n) = \left(x + y + (x+y)yh + \tau y^{m+1}, y^n \right) \\
  &=& \left((x+y)(1 + yh) + \tau y^{m + 1}, y^n \right) \\
  &=& \left( (1 + yh)(x + y + \tau' y^{m + 1}), y^n \right) \\
  &=& (x + y + \tau' y^{m + 1}, \ y^n).
\end{eqnarray*}
We now divide into two cases. If $m + 1 \ge n$, then $I = (x + y + \tau' y^{m + 1}, y^n) = (x+y, y^n)$. Otherwise, we assume that $m + 1 <n$. Then
$$
I \overline{A} = (\overline{y}) \oplus (\overline{x}) \oplus \left(\overline{f}, \overline{y^n} \right) =  (\overline{y}) \oplus (\overline{x}) \oplus \left(\overline{\tau}\cdot \overline{y}^{m +1}, \overline{y^n} \right) =  (\overline{y}) \oplus (\overline{x}) \oplus (\overline{y}^{m+1})
$$
so that $I = (x, y) \cap (x, y) \cap (x+y, y^{m + 1}) = (x+y, y^{m+1})$. Hence, in any case, we have
\begin{center}
$I = (x + y, y^n)$ for some $n > 1$. 
\end{center}
Then, because $xy^n + y^{n+1} = (x+y)^2 y^{n-1}$, we get
\begin{eqnarray*}
I^2 &=& \left( (x+y)^2, xy^n + y^{n+1}, y^{2n}\right) = \left( (x+y)^2, (x+y)^2y^{n-1}, y^{2n}\right) \\
     &=& \left( (x+y)^2, y^{2n}\right)
\end{eqnarray*}
which yields $I^2 = a I$ for some $a \in I$ (we can choose $a = x+y + y^n$, or use \cite{Eakin-Sathaye}). Finally we found a contradiction. Hence, $I = \m$, or $I^2 = a I$ for some $a \in I$. This completes the proof.
\end{proof}

Hence, we have the following.

\begin{cor}\label{10.5}
Under the same notation as in Theorem $\ref{10.4}$, if $|k| =2$, then $A$ is weakly Arf, but not an Arf ring.
\end{cor}
  
\begin{proof}
Let $I \in \Lambda(A)$ and assume that $I \ne A$. Since $|k| =2$, by \cite[Example 8.3.2]{SH} (or \cite[(14.5)]{HIO}), we obtain that, for every $f \in \m$, $(f)$ is not a reduction of $\m$. This implies that $\m \not\in \Lambda(A)$. Hence, $I^2 = a I$ for some $a \in I$ by Theorem \ref{10.4}. That is, $A$ is a weakly Arf ring; see Theorem \ref{2.4}. Since $\rme(A) = 3$ and $\mu_A(\m) =2$, the ring $A$ does not have minimal multiplicity. Therefore, $A$ is not an Arf ring.
\end{proof}

Closing this section we prove the following.

\begin{cor}\label{10.6}
Let $B= k[[X, Y]]$ be the formal power series ring over a finite field $k$ and set $A = B/(Y\prod_{\alpha \in k}(X+\alpha Y))$.  Then $A$ is a weakly Arf ring if and only if $|k| = 2$. 
\end{cor}

\begin{proof}
We only prove the `only if' part. Assume that $A$ is a weakly Arf ring. We set $Z = \prod_{\alpha \in k}(X+\alpha Y)$. We denote by $x, y, z$ the images of $X, Y, Z$ in $A$, respectively. Notice that $\overline{A} = A/(y) \times \prod_{\alpha \in k} A/(x + \alpha y)$. Let $I = (y, z)$. Then, by setting $f = y+z$, we have $I = (y, f)$. Hence, because $yf = y^2 + yz = y^2$, we obtain $I^2 = fI$. Note that $f \in I$ forms a non-zerodivisor on $A$, and 
$$
I \overline{A} =(\overline{x}^q) \oplus (\overline{y}) \oplus \cdots  \oplus (\overline{y})  = (\overline{z}) \oplus (\overline{y}) \oplus \cdots  \oplus (\overline{y})  = f\overline{A}
$$ 
where $q$ denotes the cardinality of the field $k$. Since $\overline{A}$ is a principal ideal ring, we get $I \overline{A} = \overline{I} \cdot\overline{A} = \overline{I\overline{A}}$. Hence 
$$
\overline{I} = I \overline{A} \cap A = f\overline{A} \cap A = 
(y, x^q) = (y, z) = I
$$
which shows $I \in \Lambda(A)$, $I = \overline{(f)}$. Hence, by Corollary \ref{9.2}, $I:I$ is a weakly Arf ring, because $I^2 = fI$. On the other hand, note that $(0):_Ay = (z)$, $(0):_Az = (y)$, and $I = (y) \oplus (z)$. Besides, $\Hom_A((y), (z)) = \Hom_A((z), (y)) = (0)$. Therefore we have an isomorphism
$$
I:I \cong \End_A((y)) \times \End_A((z)) \cong A/(z) \times A/(y)
$$
of $A$-algebras. Hence, by Proposition \ref{3.8}, we get $A/(z)$ is a weakly Arf ring. We set $C = B/(Z) \cong A/(z)$. Then $|\Max \overline{C}| = q = |k|$. Therefore, by \cite[Hilfssatz 2]{Jager}, for each $J \in \calF_C$, there exists $\xi \in J$ such that $\xi \overline{C} = J\overline{C}$. In particular, if we take $J$ to be the maximal ideal of $C$, then $J$ contains a parameter ideal $Q=(\xi)$ as a reduction. Hence, because $J \in \Lambda(C)$ and $C$ is weakly Arf, we conclude that $C$ has minimal multiplicity. Consequently, the multiplicity $q = |k|$ is equal to the embedding dimension $2$. That is, $|k| = q = 2$, as desired. 
\end{proof}

To sum up the arguments in this section, we record the following.

\begin{rem}\label{10.7}
Let $B= k[[X, Y]]$ be the formal power series ring over a field $k=\Bbb Z/(2)$ and set $R = B/(XY(X+Y))$. Let $A$ be one of the following rings:
\begin{enumerate}[$(1)$]
\item $A = R \ltimes M$, where $M$ is a finitely generated $\overline{R}$-module which is torsion-free as an $R$-module. 
\item $A = R \times_{R/\m}R$, where $\m$ denotes the maximal ideal of $R$.
\item $A = \prod_{i=1}^n R$, where $n >0$.
\end{enumerate}
Then $A$ is a weakly Arf ring, but not an Arf ring. 
\end{rem}


\section{Examples arising from invariant subrings}\label{sec11}

Throughout this section, unless otherwise specified, let $k$ be a field of characteristic two and $S=k[X, Y]$ the polynomial ring over $k$. Consider $S$ to be a $\Bbb Z$-graded ring via the grading $S_0 =k$ and $X, Y \in S_1$. For the rest of this paper, we denote by $\operatorname{ch}k$ the characteristic of the field $k$. Since $\ch k =2$, we obtain the general linear group
$$
\GL_2(k) = \left\{ 
\left(\begin{smallmatrix}
1 & 0 \\
0 & 1
\end{smallmatrix}\right),
\left(\begin{smallmatrix}
1 & 0 \\
1 & 1
\end{smallmatrix}\right), 
\left(\begin{smallmatrix}
0 & 1 \\
1 & 0
\end{smallmatrix}\right), 
\left(\begin{smallmatrix}
1 & 1 \\
0 & 1
\end{smallmatrix}\right), 
\left(\begin{smallmatrix}
1 & 1 \\
1 & 0
\end{smallmatrix}\right), 
\left(\begin{smallmatrix}
0 & 1 \\
1 & 1
\end{smallmatrix}\right)
\right\}
$$
which is isomorphic to the symmetric group of degree $3$. For each $\sigma \in \GL_2(k)$, we define 
$$
\left(\begin{smallmatrix}
X_1 \\
Y_1
\end{smallmatrix}\right)
= \sigma^{-1}
\left(\begin{smallmatrix}
X \\
Y
\end{smallmatrix}\right).
$$
Then the element $\sigma \in \GL_2(k)$ induces the $k$-automorphism 
$\widehat{\sigma}: S \to S$ of $S$ so that $\widehat{\sigma}(X)=X_1$ and $\widehat{\sigma}(Y)=Y_1$. This gives a group homomorphism $$\rho: \GL_2(k) \to \Aut_k(S)$$ and it assigns the map $\widehat{\sigma}$ to each $\sigma \in \GL_2(k)$. 

\medskip

In this section, let us explore the invariant subring $R^G$ of $R = k[X, Y]/(XY(X+Y))$, where $G$ is a subgroup of $\GL_2(k)$.  We denote by $x, y$ the images of $X, Y$ in $R$, respectively. For a matrix $M$ with the entries in a ring $R$, let ${\rm I}_2(M)$ be the ideal of $R$ generated by $2\times 2$-minors of $M$.

Let $H = \left< 
(\begin{smallmatrix}
1 & 0 \\
1 & 1
\end{smallmatrix})\right>$ be a subgroup of $\GL_2(k)$. This acts linearly on $S$. Since $XY\cdot(X+Y) \in S^H$, the cyclic group $H$ also acts on $R$. 

\begin{thm}\label{15.1}
The following assertions hold true. 
\begin{enumerate}[$(1)$]
\item $R^H = k[x, xy + y^2, xy^2 + y^3]$.
\item $R^H \cong k[X, A, B]/{\rm I}_2
(\begin{smallmatrix}
X & A^2 & B \\
0 & B & A
\end{smallmatrix})
$ as a graded $k$-algebra, where $k[X, A, B]$ denotes the polynomial ring over the field $k$. 
\item $R^H$ is strictly closed in $\overline{R^H}$. In particular, $R^H$ is a weakly Arf ring.
\end{enumerate}
\end{thm}

\begin{proof}
$(1)$ The homogeneous component $R_n$ of $R$ is spanned by
\begin{center}
$R_0 = k$, $R_1 = \left<x, y\right>$, and $R_n = \left<x^n, x^{n-1}y, y^n \right>$,
\end{center}
because $x^{n-1}y = x^{n-2}y^2 = \cdots = xy^{n-1}$ in $R$ for every $n \ge 2$. 
Denote by $[[R]] = \sum_{n=0}^{\infty}\dim_kR_n \lambda^n \in \Bbb Z[[\lambda]]$ the Hilbert series of $R$. Then 
$$
[[R]] = \frac{1-\lambda^3}{(1-\lambda)^2} = \frac{1+ \lambda + \lambda^2}{1-\lambda} = 1 + 2 \lambda + 3 \lambda + \cdots + 3 \lambda^n + \cdots
$$
which shows that $\{x^n, x^{n-1}y, y^n\}$ is a $k$-basis of $R_n~(n \ge 2)$. We are now computing the homogeneous component of $R^H$. Obviously, $\left(R^H\right)_0 = k$. Suppose that $n =1$. For each $\alpha \in R_1$, we can write $\alpha = ax + by$, where $a, b \in k$. We then have $\alpha \in R^H$ if and only if $ax + b(x+y) = ax + by$, which implies $a+b = a$. Hence $b = 0$, so that $\left(R^H\right)_1 = \left<x\right>$. We assume that $n \ge 2$. Consider $\alpha \in R_n$ and write $\alpha = ax^n + bx^{n-1}y + c y^n$, where $a, b, c \in k$. Then, $\alpha \in R^H$ if and only if 
$$
ax^n + bx^{n-1}(x+y) + c(x+y)^n = ax^n + bx^{n-1}y + c y^n
$$
which is equivalent to 
$$
bx^n + c \left[x^n + \binom{n}{1} x^{n-1}y + \cdots + \binom{n}{n-1} xy^{n-1}\right] = 0.
$$
Hence $(b+ c)x^n + c(2^n-2) x^{n-1}y = 0$, so that $b+c =0$. Consequently, we conclude that 
$$
\left(R^H\right)_n = \left<x^n, x^{n-1}y+y^n \right>.
$$
Therefore, $R^H$ is generated by $x$ and $\{x^{n-1}y + y^n\}_{n\ge2}$ as a $k$-algebra. Besides, for $n \ge 4$, let us write $n = 2 a + 3 b$ with $a, b \ge 0$. If $a = 0$ and $b > 0$, then $y^n + x^{n-1}y = y^{3b} + x^{3b-1}y = (x^2y+y^3)^b$. If $a > 0$ and $b=0$, then  $y^n + x^{n-1}y =  y^{2a} + x^{2a-1}y = (xy + y^2)^a$. Finally if $a, b >0$, then 
$$
y^n + x^{n-1}y = (y^{2a} + x^{2a-1}y)(y^{3b} + x^{3b-1}y) = (xy + y^2)^a \cdot (x^2y+y^3)^b.
$$
Therefore, $R^H = k[x, xy+y^2, x^2y + y^3]$, as desired.

$(2)$ Note that $\dim R^H =1$ and $x^2 + xy + y^2 \in R$ is a non-zerodivisor on $R$. Hence $R^H$ is a Cohen-Macaulay ring. Set $a = xy + y^2$ and $b = x^2 y + y^3$ in $R^H$. 
Let 
$$\psi : R^H/(x^2+a) \to R/(x^2 + a) \cong k[X, Y]/ (X^2 + XY + Y^2, X^2Y + XY^2)
$$ be the $k$-algebra map such that $\psi(\overline{x}) = \overline{X}$ and $\psi(\overline{y}) = \overline{Y}$. Then $\{1, \overline{X}, \overline{X^2}, \overline{X^2Y}\}$ forms a $k$-basis of $\Im \psi$, whence 
$\dim_kR^H/(x^2+a) \ge 4$.
On the other hand, let $U = k[X, A, B]$ be the polynomial ring over $k$ and consider the $k$-algebra map 
$$\varphi: U \to R^H=k[x, a, b]
$$ so that $\varphi(X) = x$, $\varphi(A) = a$, and $\varphi(B) = b$. Since $xa  = 0$, $xb = 0$, and $a^3 = b^2$, the map $\varphi$ induces the surjection $\overline{\varphi} : U/(XA, XB, A^3 - B^2) \to R^H$. Remember that $x^2 + a  = x^2 + xy + y^2$ is a non-zerodivisor on $R$. By setting $K = \Ker \overline{\varphi}$, we obtain the exact sequence
$$
0 \to K/(X^2+A)K \to U/(XA, XB, A^3-B^2, X^2+A) \to R^H/(x^2+a) \to 0
$$
of $U$-modules. Notice that the middle term $U/(XA, XB, A^3-B^2, X^2+A)$ of the sequence is isomorphic to $k[X, B]/(X^3, XB, B^2)$. By computing the dimension as a $k$-vector space, we conclude that $\dim_k R^H/(x^2+a) \le 4$. Hence $\dim_k R^H/(x^2+a) = 4$ and we have an isomorphism 
$$
U/(XA, XB, A^3-B^2, X^2+A) \cong R^H/(x^2+a).
$$
This yields $K/(X^2+A)K=(0)$, so that $K = (0)$. Therefore we obtain the isomorphism
$$
R^H \cong U/(XA, XB, A^3 - B^2) = k[X, A, B]/ {{\rm I}_2(
\begin{smallmatrix}
X & A^2 & B \\
0 & B & A
\end{smallmatrix})}
$$
of graded $k$-algebras. 

$(3)$ This follows from Proposition \ref{15.2} below.
\end{proof}

\begin{prop}\label{15.2}
Let $k$ be an arbitrary field (not necessarily of characteristic two). Let $U = k[X, Y, Z]$ be the polynomial ring over $k$. Set $R = U/{\rm I}_2(
\begin{smallmatrix}
X & Y^2 & Z \\
0 & Z & Y
\end{smallmatrix}
)$. Then $R$ is strictly closed in $\overline{R}$. In particular, $R$ is a weakly Arf ring. 
\end{prop}

\begin{proof}
We consider $U$ as a $\Bbb Z$-graded ring under the grading $U_0 =k$, $X \in U_1$, $Y \in U_2$, and $Z \in U_3$. Set $\fka = {\rm I}_2(
\begin{smallmatrix}
X & Y^2 & Z \\
0 & Z & Y
\end{smallmatrix}
) = (XY, XZ, Y^3 - Z^2)$. Then $\fka$ is a graded ideal of $U$ with $\height_U\fka = 2$. In particular, $R$ is a Cohen-Macaulay ring with $\dim R=1$. Since $\fka = (X, Y^3 - Z^2) \cap (Y, Z)$, we then have an exact sequence
$$
0 \to R \overset{\xi}{\to} U/(Y, Z) \times U/(X, Y^3-Z^2) \to U/(X, Y, Z) \to 0
$$
of $U$-modules. Note that $U/(Y, Z) =k[x]$ is a DVR and $U/(X, Y^3-Z^2) = k[y, z] \cong k[t^2, t^3] \subseteq k[t]$, where $x, y, z$ denote the images in the corresponding rings and $t$ is an indeterminate over $k$. 
Let $M = (x, y, z) \in \Max R$, where $x, y, z$ denotes, again, the images in $R$. Then $M^2 = (x+y)M$ and $\mu_R(M) =3$. Hence
$$
R^M = \frac{M}{x+y} = R + \left<\frac{x}{x+y}, \frac{z}{x+y}\right> \quad \text{in} \ \ \rmQ(R)
$$
because $M = (x+y, x, z)$. Since $\xi(x) = (x, 0)$, $\xi(z) = (0, z)= (0, t^3)$, $\xi(x+y) = (x, y) = (x, t^2)$ via the identification $U/(X, Y^3-Z^2) = k[y, z] = k[t^2, t^3]$, we obtain
$$
\frac{x}{x+y} = \left(\frac{1}{x}, \frac{1}{t^2}\right)\cdot(x,0) = (1, 0), \ \ \frac{z}{x+y} =\left(\frac{1}{x}, \frac{1}{t^2}\right)\cdot(0,t^3) = (0, t) 
$$
which yield that $R^M = \left< 1, (1, 0), (0, t)\right>$. Since $(0, 1) \in R^M$, we have 
$$
R^M \supseteq R(1, 0) + R(0, 1) = k[x] \times k[y, z] = k[x] \times k[t^2, t^3].  
$$
Now, because $k[t] = k[t^2, t^3] + k[t^2, t^3]t$ and $(0, t) \in R^M$, we see that
$$
R^M \subseteq k[x] \times k[t] = \overline{R}
$$
whence $R^M = \overline{R}$. Therefore, $R_M$ is an Arf ring. 
Let $N \in \Max U$ such that $N \supseteq \fka$ and $N \ne (X, Y, Z)$. We then have either $N \supsetneq (X, Y^3-Z^2)$ and $N \not\supseteq (Y, Z)$, or $N \supsetneq (Y, Z)$ and $N \not\supseteq (X, Y^3-Z^2)$. Hence the multiplicity of $R_N$ is at most $2$, that is, $R_N$ is an Arf ring. Therefore, $R$ is strictly closed in $\overline{R}$. 
\end{proof}

In what follows, let $G = \left<(
\begin{smallmatrix}
0 & 1\\
1 & 1
\end{smallmatrix})
\right>$ be the subgroup of $\GL_2(k)$. Then, similarly for the subgroup $H$ of $\GL_2(k)$, it acts linearly on $S$. 


\begin{thm}\label{15.3}
The following assertions hold true. 
\begin{enumerate}[$(1)$]
\item $S^G = k[X^2 + XY + Y^2, XY(X+Y), X^3 + X^2Y + Y^3]$. 
\item $S^G \cong k[A, B, C]/(A^3 - (B^2 + BC + C^2))$ as a graded $k$-algebra, where $k[A, B, C]$ denotes the polynomial ring over the field $k$. 
\end{enumerate}
\end{thm} 

\begin{proof}
Set 
$a = X^2 + XY + Y^2$, $b = XY(X+Y) = X^2Y + XY^2$, and $c = X^3 + X^2Y + Y^3$.
We then have $a, b, c \in S^G$ and $a^3 = b^2 + bc + b^2$. Let $R = k[a, b, c]$. Then $R \subseteq S^G$. Let 
$$
f(t) = (t-X)(t-(X+Y)) (t-Y) = t^3 -at +b \in S[t]
$$ be the polynomial in $S$. Then $f(X) = f(Y) = 0$, so that $S$ is integral over $R$; hence $\dim R =2$. Let $U = k[A, B, C]$ be the polynomial ring over $k$, and consider $U$ as a $\Bbb Z$-graded ring with the grading $U_0 = k$, $A \in U_2$, and $B, C \in U_3$. Let 
$$
\varphi: U \to R
$$ be the $k$-algebra map so that $\varphi(A) = a$, $\varphi(B) =b$, and $\varphi(C) = c$. Then $\varphi$ is a surjective,   graded homomorphism, and $A^3 - (B^2 + BC +C^2) \in \Ker \varphi$. Hence we get a surjection 
$$
\overline{\varphi} : U/(A^3 - (B^2 + BC +C^2)) \to R
$$
induced by the map $\varphi$. Set $L = \Ker \overline{\varphi}$. We get the exact sequence
$$
0 \to L/BL \to U/(A^3 - (B^2 + BC +C^2), B) \to R/(b) \to 0
$$
of $U$-modules. Notice that
$$
U/(A^3 - (B^2 + BC +C^2), B) \cong k[A, C]/(A^3 - C^2) \cong k[t^2, t^3] \ (\subseteq k[t])
$$
is an integral domain of dimension one, where $t$ denotes an indeterminate over $k$. Therefore, because $\dim R/(b) = 1$, we have 
$$
U/(A^3 - (B^2 + BC +C^2), B) \cong R/(b)
$$
which yields $L/BL=(0)$. Hence $L =(0)$. Finally we get the isomorphism 
$$
R \cong U/(A^3 - (B^2 + BC +C^2))
$$
of graded $k$-algebras, and the Hilbert series $[[R]]$ of $R$ is given by 
$$
[[R]] = \frac{1-\lambda^6}{(1- \lambda^2)(1-\lambda^3)(1-\lambda^3)} = \frac{1 + \lambda^3}{(1-\lambda^2)(1-\lambda^3)}.
$$
Notice that $S^G$ is a Cohen-Macaulay ring, because the order of $G$ is invertible in $k$. Since $a, b \in S^G$ is a homogeneous system of parameters in $S^G$, we have an isomorphism
$$
S^G/(a, b)S^G \cong \left[S/(a,b)S\right]^G.
$$ 
The $k$-basis of $S/(a, b)S$ is $\{1, X, Y, X^2, XY, X^2Y\}$, which yields that $\left[S/(a,b)S\right]^G = k + k\cdot x^2y$, where $x, y$ stand for the images of $X, Y$ in $S/(a, b) S$. Hence 
$
[[S^G/(a, b)S^G]] = [[\left[S/(a,b)S\right]^G]] = 1 + \lambda^3
$
so that 
$$
[[S^G]] = \frac{1+\lambda^3}{(1-\lambda^2)(1-\lambda^3)}.
$$
Therefore $R = S^G$ as desired. 
\end{proof}

Consequently we reach the following. 

\begin{cor}\label{15.4}
Let $R = k[X, Y]/(XY(X+Y))$. Then the following assertions hold true. 
\begin{enumerate}[$(1)$]
\item $R^G \cong k[t^2, t^3]$ as a graded $k$-algebra, where $t$ denotes an indeterminate over $k$. 
\item $R^G$ is strictly closed in $\overline{R^G}$. 
\end{enumerate}
\end{cor}

\begin{proof}
$(1)$ This follows from $\left[S/(b)S\right]^G \cong S^G/bS^G$. 

$(2)$ Since $R^G \cong k[t^2, t^3]$ and $k[t^2, t^3]$ is a weakly Arf ring,  $R^G$ is strictly closed by Theorem \ref{8.2}. 
\end{proof}


\section{Examples arising from determinantal rings} \label{sec12}

The aim of this section is to prove Theorem \ref{ch}, which shows that the Arf property depends on the characteristic of the base ring.
In what follows, let $U =k[[X, Y, Z]]$ denote the formal power series ring over a field $k$, and $\fka = {\rm I}_2(
\begin{smallmatrix}
X & Y & Z \\
Y & Z & X
\end{smallmatrix})$. Recall that  ${\rm I}_2(
M)$ is the ideal of $U$ generated by $2\times 2$-minors of a matrix $M$. 
Set 
$$
A = U/\fka=k[[x, y, z]]
$$
where $x, y, z$ stand for the images of $X, Y, Z$ in $A$, respectively. Notice that $A$ is a Cohen-Macaulay local ring with $\dim A =1$. 

\begin{thm}\label{ch}
The following assertions hold true.
\begin{enumerate}[$(1)$]
\item If $\ch k =3$, then $A$ is not an Arf ring. 
\item If $\ch k \ne 3$ and there exists $\alpha \in k$ such that $\alpha \ne 1$, $\alpha^3 =1$, then $A$ is an Arf ring.  
\end{enumerate}
\end{thm}

\begin{proof}
Let $\m = (x, y, z)$ be the maximal ideal of $A$. Then $\m^2 = x\m = y\m = z\m$. Hence $x, y, z\in W(A)$ and $A$ has minimal multiplicity $3$. Set $t = y/x \in \rmQ(A)$. We then have 
$$
t = \frac{y}{x} = \frac{z}{y} = \frac{x}{z}
$$
so that $y = tx$, $z = ty$, and $x = tz$. Hence $t^3 = 1$ in $\rmQ(A)$. Let $B = A^\m$ denote the blow-up of $A$ at $\m$. Then, because $\m^2 = x\m$, we have
$$
B = A^{\m} = \frac{\m}{x} = \left<1, t, t^2\right> = A[t]
$$
where the third equality follows from $\m = (x, tx, t^2x)$.

$(1)$ Suppose that $\ch k =3$. Set $s = t-1 \in B$. Then $s^3 = 0$ and $B = A[s]$. 
Let $M$ be the maximal ideal of $B$. We then have $s \in M$ and $M \cap A = \m$. Hence
$$
M \supseteq \m B + sB = (x, tx, t^2x, s)B = (x, s)B
$$
which implies the surjection $A/\m \to B/(x,s)B \ne (0)$. Actually, it is a bijection, and hence $B/(x,s)B$ is a field. Because of the surjection $B/(x,s)B \to B/M$, we have the isomorphisms
$$
A/\m \cong B/(x,s)B \cong B/M.
$$ 
Thus, $M = \m B+sB = (x,s)B$. Therefore, we obtain that $B$ is a local ring with maximal ideal $M = (x, s)B$ and $A/\m \cong B/M$. 

Let $k[[X, S]]$ be the formal power series ring over $k$ and $\varphi : k[[X, S]] \to B$ be the $k$-algebra map defined by $\varphi(X) = x$ and $\varphi(S) = s$. Then $\varphi$ is a surjective, graded ring homomorphism, and $S^3 \in \Ker \varphi$. Since $B$ is a Cohen-Macaulay local ring with $\dim B =1$, we get 
\begin{center}
$\Ker \varphi = (S^n)$ for some $1 \le n \le 3$.
\end{center}
If $n \ne 3$, then $S^2 \in \Ker \varphi$. That is, $s^2 = 0$ in $B$. Since $0=s^2 = t^2 + t+1$, by multiplying $x^2$, we have $y^2 + xy + x^2 = 0$ in $A$. Hence $X^2 + XY + Y^2 \in \fka$, so that 
$$
X^2 + XY + Y^2 = a (X^2 -YZ) + b (Y^2-XZ) + c(Z^2 -XY)
$$
for some $a, b, c \in k$, which makes a contradiction, because $a = 1 = 0$. Therefore, $n=3$ and we have the isomorphism
$$
B \cong k[[X, S]]/(S^3)
$$
of $k$-algebras. Hence, the blow-up $B=A^{\m}$ does not have minimal multiplicity, so that $A$ is not an Arf ring by \cite[Theorem 2.2]{L}. 

$(2)$ Suppose $\ch k \ne 3$ and there exists $\alpha \in k$ such that $\alpha \ne 1$ and $\alpha^3 = 1$. We consider $\Lambda=\{\alpha \in k \mid \alpha^3=1\}$. We then have $|\Lambda| = 3$, because the polynomial $f(t) = t^3 -1 \in k[t]$ does not have a double root. For each $\alpha \in \Lambda$, we define
$$
P_{\alpha} = (Y-\alpha X, Z - \alpha^2X) \subseteq U. 
$$
Then $P_{\alpha} + (X) = (X, Y, Z)$, which implies $Y-\alpha X, Z - \alpha^2X, X$ is a regular system of parameters in $U$. Since $U/P_{\alpha}$ is a DVR, $P_{\alpha} \in \Spec U$ and $\height_{U}P_{\alpha} = 2$. Let $k[[X]]$ denote the formal power series ring over $k$ and $\varphi : U \to k[[X]]$ be the $k$-algebra map so that $\varphi(X) = X$, $\varphi(Y) = \alpha X$, and $\varphi(Z) = \alpha^2 X$. Then $\Ker \varphi = P_{\alpha}$ and $\fka \subseteq P_{\alpha}$. 


\begin{claim*}
With the above notation, we have
$\bigcap_{\alpha \in \Lambda}P_{\alpha} = \fka$ and, for each $\alpha, \beta \in \Lambda$, $P_{\alpha} \ne P_{\beta}$ if $\alpha \ne \beta$. 
\end{claim*}

\begin{proof}[Proof of Claim]
If $P_{\alpha} = P_{\beta}$ for some $\alpha, \beta \in \Lambda$ such that $\alpha \ne \beta$. Then $Y-\alpha X, Y-\beta X \in P_{\alpha}$. Hence $(\alpha -\beta) X \in P_{\alpha}$, showing $X \in P_{\alpha}$. Thus $\height_U P_{\alpha} = 3$, which makes a contradiction. Hence $\{P_{\alpha}\}_{\alpha \in \Lambda}$ is distinct. Let us make sure of the first assertion. To do this, for each $\alpha \in \Lambda$, let 
$$
\p_{\alpha} = (y-\alpha x, z- \alpha^2 x) \subseteq A = U/\fka
$$
which is an associated prime ideal of $A$, because $\p_{\alpha} \in \Min A$. Since $\{P_{\alpha}\}_{\alpha \in \Lambda}$ is distinct, $|\Lambda| =3$, and $\Ass A \supseteq \{\p_{\alpha} \mid \alpha \in \Lambda\}$, we have $|\Ass A| \ge 3$. On the other hand, we get
$$
3 = \mu_A(\m) = \rme^0_{\m}(A) = \sum_{\p \in \Ass A}\ell_{A_{\p}}(A_{\p})\cdot \rme^0_{\m/\p}(A/\p) \ge |\Ass A| \ge 3
$$
where the second equality follows from the fact that $A$ has minimal multiplicity. This  shows $|\Ass A| = 3$ and hence $\Ass A = \{\p_{\alpha} \mid \alpha \in \Lambda\}$. Besides, for each $\p \in \Ass A$, we have that $A_{\p}$ is a field and $A/\p$ is a DVR. Consequently, $\bigcap_{\alpha \in \Lambda} \p_{\alpha} = (0)$ in $A$, and therefore $\bigcap_{\alpha \in \Lambda}P_{\alpha} = \fka$, as desired. 
\end{proof}

In particular, by the above claim, $A$ is a reduced ring, so that 
$$
\overline{A} = U/P_{\alpha} \times U/P_{\beta} \times U/P_{\gamma}
$$
where $\Lambda = \{\alpha, \beta, \gamma\}$. We now consider $A$ in $\overline{A}$ via the injection
$$
A \to \overline{A} = U/P_{\alpha} \times U/P_{\beta} \times U/P_{\gamma} = k[x] \times k[x] \times k[x]
$$
where $x$ denotes the images in the corresponding rings. Thus, we identify $x \in A$ with $(x, x, x)$, $y \in A$ with $(\alpha x, \beta x, \gamma x)$, and $z \in A$ with $(\alpha^2 x, \beta^2 x, \gamma^2 x)$. Hence 
$$
t = \frac{y}{x} = (\alpha, \beta, \gamma) \in \overline{A}
$$
so that $t - \alpha = (\alpha, \beta, \gamma) - (\alpha, \alpha, \alpha) = (0, \beta - \alpha, \gamma - \alpha)$. Therefore
$$
(t-\alpha)\overline{A} = (0) \times U/P_{\beta} \times U_{P_{\gamma}} \subsetneq \overline{A}.
$$
Choose $N \in \Max \overline{A}$ such that $(t-\alpha) \overline{A} \subseteq N$. Since $N \supseteq \m \overline{A} + (t-\alpha) \overline{A}$, we obtain 
$$
M \supseteq \m B = (t -\alpha )B
$$
where $M = N \cap B$. Now, since $B= A[t] = A[t-\alpha]$, we get $A/\m \cong B/\left[\m B + (t -\alpha )B\right]$. Moreover, because $B/\left[\m B + (t -\alpha )B\right] \to B/M$ is surjective, it is a bijection. Hence $A/\m \cong B/M$ and $M = \m B + (t-\alpha)B$. For each $\alpha \in \Lambda$, we set
$$
M_{\alpha} = \m B + (t-\alpha )B
$$
which is a maximal ideal of $B$. Then $\{M_{\alpha}\}_{\alpha \in \Lambda}$ is distinct. Indeed, if $M_{\alpha} = M_{\beta}$ for some $\alpha, \beta \in \Lambda$ such that $\alpha \ne \beta$, then because $t-\alpha, t-\beta \in M_{\alpha}$, we have $0 \ne \alpha -\beta \in M_{\alpha}$. This is impossible. Hence, the elements of $\{M_{\alpha}\}_{\alpha \in \Lambda}$ are distinct from each other. In particular, $|\Max B| \ge 3$. In the meantime, for each $M \in \Max B$, we can choose $N \in \Max \overline{A}$ such that $N \cap B = M$. Since $|\Max \overline{A}| = 3$, we conclude that $|\Max B| \le 3$. Hence, $|\Max B| = 3$ and $\Max B = \{M_{\alpha}, M_{\beta}, M_{\gamma}\}$ (Here $\Lambda=\{\alpha, \beta, \gamma\}$). Therefore, because  $B$ is complete (since $A = \widehat{A}$), we choose local rings $B_1$, $B_2$, and $B_3$ such that
$$
B \cong B_1 \times B_2 \times B_3.
$$
We then have the equalities
$$
3 = \rme^0_{\m}(A) = \rme^0_{xA}(A) = \rme^0_{xA}(B) = \ell_A(B/XB) = \sum_{i =1}^3\ell_A(B_i/XB_i). 
$$
Hence $\ell_A(B_i/XB_i) = 1$ for every $1 \le i \le 3$, which shows $\ell_{B_i}(B_i/XB_i) = 1$. Therefore, $B_i$ is a DVR with maximal ideal $XB_i$. Thus, $B \cong B_1 \times B_2 \times B_3$ is regular. Therefore $B = \overline{B} = \overline{A}$, because $\rmQ(B) = \rmQ(A)$. This implies that $A$ is an Arf ring, since $B$ has minimal multiplicity and $B = \overline{A}$. This completes the proof.
\end{proof}

Hence, the Arf property depends on the characteristic of the field $k$. Consequently we have the following.

\begin{cor}
Suppose that $k$ is an algebraically closed field. Then $A$ is an Arf ring if and only if $\ch k \ne 3$. 
\end{cor}

In the rest of this section, let us consider $R = k[X, Y, Z]/{\rm I}_2 (
\begin{smallmatrix}
X & Y & Z \\
Y & Z & X
\end{smallmatrix})$. We denote by $x, y, z$ the images of $X, Y, Z$ in $R$, respectively.

\begin{cor}
Suppose that $\ch k \ne 3$ and there exists $\alpha \in k$ such that $\alpha \ne 1$, $\alpha^3 =1$. Then $R$ is strictly closed in $\overline{R}$.
\end{cor}

\begin{proof}
Let $U=k[X, Y, Z]$ be the polynomial ring over the field $k$ and $\fka = {\rm I}_2 (
\begin{smallmatrix}
X & Y & Z \\
Y & Z & X
\end{smallmatrix})$. Notice that $R$ is Cohen-Macaulay and of dimension one. Thanks to Theorem \ref{ch} (2), it is enough to show that, for every $\p \in \Max R$ such that $\p \ne (x, y, z)$, $R_{\p}$ is regular. To do this, let $P \in \Spec U$ such that $P \supseteq \fka$ and $P \ne (X, Y, Z)$. Then $X \not\in P$. Let us consider
$\widetilde{U} = U\left[\frac{1}{x}\right]$ in $\rmQ(U)$.
Then 
$$
\widetilde{U} = U\left[\frac{1}{x}\right] = k\left[X, \frac{1}{X}\right]\left[\frac{Y}{X}, \frac{Z}{X}\right] = \widetilde{k}[Y_1, Z_1]
$$
is a polynomial ring over $\widetilde{k}$, where $\widetilde{k} = k\left[X, \frac{1}{X}\right]$, $Y_1 =\frac{Y}{X}$, and $Z_1 = \frac{Z}{X}$. Besides, $\fka \widetilde{U} = (1- Y_1Z_1, Y_1^2 - Z_1, Z_1^2 - Y_1)$. Hence we get the isomorphism
$$
\widetilde{U}/\fka \widetilde{U} \cong \widetilde{k}[Y_1]/(Y_1^3-1, Y_1^4-Y_1) \cong \widetilde{k}[Y_1]/(Y_1^3-1)
$$
of $\widetilde{k}$-algebras. 
We consider $\Lambda=\{\alpha \in k \mid \alpha^3=1\}$ and set $\Lambda =\{\alpha, \beta, \gamma\}$. Because 
$$
Y_1^3 = (Y_1 -\alpha)(Y_1 -\beta)(Y_1 -\gamma)
$$
in $k[Y_1]$, we have the isomorphisms
\begin{eqnarray*}
\widetilde{U}/\fka \widetilde{U} &\cong& \widetilde{k}[Y_1]/(Y_1^3-1) \cong \widetilde{k} \otimes_k \left(k[Y_1]/(Y_1^3-1)\right) \\
&\cong& \widetilde{k} \otimes_k \left(k[Y_1]/(Y_1-\alpha) \times k[Y_1]/(Y_1-\beta) \times k[Y_1]/(Y_1-\gamma)\right) \\
&\cong& \widetilde{k} \otimes_k \left(k \times k \times k\right) \\
&\cong& \widetilde{k} \times \widetilde{k} \times \widetilde{k} 
\end{eqnarray*}
of $\widetilde{k}$-algebras. As $X \not\in P$, we conclude $R_P$ is regular. Hence $R$ is strictly closed in $\overline{R}$.
\end{proof}


\section{Direct summands of weakly Arf rings} \label{sec13}

Some of the rings of invariants could be strictly closed in their integral closures; see Theorem \ref{15.1} and Corollary \ref{15.4}. In this section, we investigate this phenomenon. 

For commutative rings $A$ and $B$, we denote by $B/A$ a ring extension, i.e., $A$ is a subring of $B$. 
Let $T = W(A)$ be the set of all non-zerodivisors on $A$. 
We begin with the basic definition. 

\begin{defn}\label{11.1}
We say that the extension $B/A$ satisfies {\it the condition $(\sharp)$}, if $W(A) \subseteq W(B)$, and for every $x \in \rmQ(B)$, there exists $t \in W(A)$ such that $tx \in B$. 
\end{defn}

In particular,  if $B/A$ satisfies $(\sharp)$, then $\rmQ(B) = T^{-1}B$. For the converse, we have the following lemmata. 


\begin{lem}\label{11.3}
Suppose that $B$ is an integral domain, and $B$ is integral over $A$. Then the extension $B/A$ satisfies $(\sharp)$.
\end{lem}

\begin{proof}
We have $T \subseteq W(B)$. Hence $B \subseteq T^{-1}B \subseteq \rmQ(B)$. Since $T^{-1}B$ is an integral extension over $\rmQ(A) = T^{-1}A$, we obtain that $T^{-1}B$ is a field. Thus $T^{-1}B = \rmQ(B)$, so that $B/A$ satisfies the condition $(\sharp)$.
\end{proof}

\begin{lem}\label{11.4}
Let $B$ be a Noetherian ring, which is integral over $A$. Suppose that the following conditions hold true. 
\begin{enumerate}[$(1)$]
\item For every $P \in \Ass B$, $P \cap A \in \Ass A$.
\item $\rmQ(A)$ is an Artinian ring. 
\end{enumerate}
Then the extension $B/A$ satisfies $(\sharp)$.
\end{lem}

\begin{proof}
Take a non-zerodivisor $f \in T$ on $A$. If we assume that $f \in P$ for some $P \in \Ass B$, then $f \in P \cap A \in \Ass A$. This makes a contradiction. Hence $T \subseteq W(B)$, so that $B \subseteq T^{-1}B \subseteq \rmQ(B)$. Since $T^{-1}B$ is an integral extension over $\rmQ(A) = T^{-1}A$, $T^{-1}B$ is an Artinian ring. For each $f \in W(B)$, it is a unit in $T^{-1}B$. Therefore $T^{-1}B = \rmQ(B)$. Hence $B/A$ satisfies the condition $(\sharp)$.
\end{proof}

With this notation we have the following. 

\begin{thm}\label{11.5}
Suppose that $B$ is integral over $A$, $A$ is a direct summand of $B$ as an $A$-module, and the extension $B/A$ satisfies $(\sharp)$. If $B$ is strictly closed in $\overline{B}$, then so is $A$ in $\overline{A}$. In particular, $A$ is a weakly Arf ring.  
\end{thm}

\begin{proof}
Since $A$ is a direct summand of $B$, we choose an $A$-linear map $\rho: B \to A$ such that $\rho(a) = a$ for each $a \in A$. The condition $(\sharp)$ of $B/A$ guarantees that $\rmQ(A) \subseteq \rmQ(B)$. Hence, $\overline{A} \subseteq \overline{B}$. We now consider a homomorphism $\varphi: \overline{A}\otimes_A \overline{A} \to \overline{B} \otimes_B \overline{B}$ of additive groups such that $\varphi(x \otimes y) = x \otimes y$ for each $x, y \in \overline{A}$.
Let $x \in A^*$. Then $x \in \overline{A} \subseteq \overline{B}$ and 
\begin{center}
$x\otimes1 = 1 \otimes x$ \  in \ $\overline{B} \otimes_B \overline{B}$. 
\end{center}
This implies $x \in B^* = B$. As $x \in \rmQ(A) = T^{-1}A$, we write $x = a/t$ where $a \in A$ and $t \in T$. Then $a = \rho(a) = \rho(tx) = t\rho(x)$, so that $x = a/t = \rho(x) \in A$. Therefore $A$ is strictly closed in $\overline{A}$, as desired.  
\end{proof}

\begin{cor}
Let $R$ be a commutative ring, and let $G$ be a finite group whose order is invertible. Suppose that the extension $R/R^G$ satisfies $(\sharp)$, where $R^G$ denotes the invariant subring of $R$. If $R$ is strictly closed in $\overline{R}$, then so is $R^G$ in $\overline{R^G}$. In particular, $R^G$ is a weakly Arf ring. 
\end{cor}

\begin{proof}
Let $n$ denote the order of $G$. Then an $R^G$-linear map
$$
\rho : R \to R^G, \ \ a \mapsto \frac{1}{n}\sum_{\sigma \in G} \sigma(a)
$$
gives a split monomorphism, so that $R^G$ is a direct summand of $R$ as an $R^G$-module. Since the order of $G$ is finite, $R$ is integral over $R^G$. Hence, the assertion follows from Theorem \ref{11.5}.
\end{proof}

Consequently, by Lemma \ref{11.3}, we get the following.

\begin{cor}
Let $R$ be an integral domain, and let $G$ be a finite group such that the order of $G$ is invertible in $R$. If $R$ is strictly closed in $\overline{R}$, then so is $R^G$ in $\overline{R^G}$. In particular, $R^G$ is a weakly Arf ring. 
\end{cor}

Let us note one more consequence of Theorem \ref{11.5}.

\begin{cor}\label{11.6}
Suppose that $B$ is integral over $A$, $A$ is a direct summand of $B$ as an $A$-module, and $B$ is a Noetherian ring. Moreover, we assume that for every $P \in \Ass B$, $P \cap A \in \Ass A$ and $A$ satisfies $(S_1)$. If $B$ is strictly closed in $\overline{B}$, then so is $A$ in $\overline{A}$. 
\end{cor}

\begin{proof}
Since $A$ is a direct summand of $B$, $A$ is Noetherian. The total ring of fractions $\rmQ(A)$ of $A$ is Artinian, because $\Min A = \Ass A$. By Lemma \ref{11.4}, the extension $B/A$ satisfies $(\sharp)$. Hence, the assertion follows from Theorem \ref{11.5}.
\end{proof}

We now apply Corollary \ref{11.6} to the idealization $A=R \ltimes M$ of a finitely generated torsion-free module $M$ over a Noetherian ring $R$.

\begin{thm}
Let $R$ be a Noetherian ring, and let $M$ be a finitely generated $R$-module such that $M$ is torsion-free.
Suppose that $R$ satisfies $(S_1)$. If $A=R \ltimes M$ is strictly closed in $\overline{A}$, then so is $R$ in $\overline{R}$.
\end{thm}

\begin{proof}
For each $P \in \Ass A$, we choose $\fkp \in \Spec R$ such that $P = \fkp \times M$. We then have an isomorphism 
$$
A_P \cong R_{\fkp} \ltimes M_{\fkp}
$$ of $R_{\fkp}$-algebras. We show that $\fkp \in \Ass R$. Indeed,  assume that $\fkp \not\in \Ass R$. Then, because $R$ satisfies $(S_1)$, $\fkp$ contains a non-zerodivisor on $M$. Hence $a$ is a non-zerodivisor on $M_{\fkp}$. This makes a contradiction. Hence $\fkp \in \Ass R$, as desired. 
\end{proof}

In the rest of this section, unless otherwise specified, we maintain the following.

\begin{setting}
Let $B$ be a Noetherian ring, and let $A$ be an arbitrary subring of $B$. We assume that $B$ is integral over $A$ and $A$ is a direct summand of $B$ as an $A$-module. 
\end{setting}

\begin{defn}\label{11.7}
We say that $B$ satisfies {\it the condition $(C)$}, if the following conditions are satisfied: 
\begin{enumerate}[$(1)$]
\item $d = \dim B < \infty$, and 
\item For every maximal chain $P_0 \subsetneq P_1 \subsetneq \cdots \subsetneq P_n$ in $\Spec B$, we have $n=d$.
\end{enumerate}
\end{defn}
\noindent
Notice that if $B$ satisfies $(C)$, then $\dim B_M = \dim B$ for every $M\in \Max B$.

\begin{lem}\label{11.8}
If $B$ satisfies $(C)$, then $P\cap A \in \Min A$ for every $P\in \Min B$. 
\end{lem}

\begin{proof}
For each $P \in \Min B$, we choose a maximal ideal $M$ in $B$ such that $P \subseteq M$. Let us consider a maximal chain 
$$
P = P_0 \subsetneq P_1 \subsetneq \cdots \subsetneq P_n = M
$$
of prime ideals in $B$ starting from $P$, and ending with $M$. We then have $n = \dim B$. Hence, by setting $\m = M \cap A$, we see that 
$$
\dim A = \dim B = n  \le \dim A_{\m},
$$
so that $\dim A_{\m} = n$. In particular, $P \cap A \in \Min A$.
\end{proof}

Hence we get the following.

\begin{prop}\label{11.9}
Suppose that $B$ satisfies $(C)$. For each $n \in \Bbb Z$, if $B$ satisfies $(S_n)$, then so does $A$.
\end{prop}

\begin{proof}
We may assume $n >0$. Suppose that $n=1$. We will show that $\Min A = \Ass A$. Indeed, let $\fkp \in \Ass A$. Since $\fkp \in \Ass_A B$, we can take $P \in \Ass B$ such that $P \cap A = \fkp$. Remember that $B$ satisfies $(S_1)$, i.e., $\Min B = \Ass B$. By Lemma \ref{11.8}, we have $\fkp = P \cap A \in \Min A$, so $A$ satisfies $(S_1)$. Suppose that $n \ge 2$ and the assertion holds for $n-1$. We assume that $A$ does not satisfy $(S_n)$. Then there exists $\fkp \in \Spec A$ such that $\depth A_{\fkp} < \min \{n, \dim A_{\fkp}\}$. Hence $\depth A_{\fkp} \le n-1$ and $\dim A_{\fkp} \ge n$. Since $B$ satisfies $(S_{n-1})$, so does $A$. Therefore, we have 
$$
\depth A_{\fkp} \ge \min \{n-1, \dim A_{\fkp}\} = n-1
$$
which yields that $\depth A_{\fkp} = n-1 >0$. Hence $\fkp \not\in \Ass A$. Now, because $A$ satisfies $(S_1)$, there exists $f \in \fkp$ such that $f \in W(A)$. Then $f \in W(B)$. The splitting monomorphism $A \to B$ gives rise to a split morphism
$$
0 \to A/fA \to B/fB
$$
of $A/fA$-modules, which is also injective. Notice that $B/fB$ is integral over $A/fA$ and $B/fB$ satisfies the condition $(C)$. In fact, because $f\in W(B)$, $\dim B/fB = d-1$, where $d = \dim B$. We now take $P_1 \in \Min_B B/fB$. Choose a maximal chain in $\Spec B$ from $P_1$:
$$
fB \subseteq P_1 \subsetneq P_2 \subsetneq \cdots \subsetneq P_n.
$$
Since $f$ is a non-zerodivisor on $B$, we see that $\height_B P_1 = 1$, whence $P_1$ contains a minimal prime $P_0$. Thus $n = d$, which implies that $B/fB$ satisfies the condition $(C)$. By induction arguments, $A/fA$ satisfies $(S_{n-1})$. Therefore
$$
\depth (A/fA)_{\fkp}\ge \min \{n-1, \dim (A/fA)_{\fkp}\}
$$
which is absurd, because $\depth A_{\fkp} = n-1$ and $\dim A_{\fkp} \ge n$. Consequently, $A$ satisfies $(S_n)$ condition. This completes the proof.
\end{proof}

We summarize some consequences.

\begin{cor}\label{11.10}
Suppose that $B$ satisfies $(C)$. If $B$ is Cohen-Macaulay, then so is $A$.
\end{cor}

\begin{cor}\label{11.11}
Suppose that $B$ satisfies $(C)$. Then the following assertions hold true. 
\begin{enumerate}[$(1)$]
\item If $B$ satisfies $(S_1)$ and $B$ is strictly closed in $\overline{B}$, then $A$ is strictly closed in $\overline{A}$.
\item If $B$ is a weakly Arf ring satisfying $(S_2)$ and $B_P$ is an Arf ring for every $P \in \Spec B$ with $\height_B P = 1$, then $A$ is a weakly Arf ring, and $A_\p$ is an Arf ring for every $\p \in \Spec A$ with $\height_A \p = 1$.
\end{enumerate}
\end{cor}

\begin{proof}
$(1)$ By Proposition \ref{11.9}, $A$ satisfies $(S_1)$. Hence the assertion follows from Corollary \ref{11.6} and Lemma \ref{11.8}.

$(2)$ Since $A$ is a direct summand of $B$, we see that $A$ is Noetherian. Besides, $A$ satisfies $(S_2)$ by Proposition \ref{11.9}. By Theorem \ref{Zariski-Lipman}, we conclude that $B$ is strictly closed in $\overline{B}$. Therefore, by assertion $(1)$, $A$ is strictly closed in $\overline{A}$. Again, by Theorem \ref{Zariski-Lipman}, we get the required assertions.
\end{proof}

For a Noetherian semi-local ring $B$ such that $B_M$ is a one-dimensional Cohen-Macaulay local ring for every $M \in \Max B$, the ring $B$ satisfies $(C)$.
Hence we have the following. 

\begin{thm}\label{11.12}
Suppose that $B$ is a Noetherian semi-local ring such that $B_M$ is a one-dimensional Cohen-Macaulay local ring for every $M \in \Max B$. If $B$ is an Arf ring, then so is $A$. 
\end{thm}

\begin{proof}
To show $A$ is an Arf ring, we need to confirm $\height_A \fkp = 1$ for every $\fkp \in \Max A$. Suppose that there exists $\fkp \in \Max A$ such that $\height_A\fkp = 0$. Then $\fkp \in \Ass A$. We can choose $P \in \Ass B$ such that $P \cap A = \fkp$. Since $B$ is integral over $A$, we see that $P \in \Max B$. Hence, by our assumption, $\height_B P = 1$ which makes a contradiction. Thus $\height_A \fkp \ge 1$ for every $\fkp \in \Max A$, whence $\height_A \fkp = 1$. In particular, $A_{\fkp}$ is a one-dimensional Cohen-Macaulay local ring for every $\fkp \in \Max A$. Since $B$ is an Arf ring, by Theorem \ref{4.3}, $B$ is strictly closed in $\overline{B}$. Therefore, $A$ is strictly closed in $\overline{A}$, and hence $A$ is an Arf ring. 
\end{proof}

Finally we reach the goal of this section. 

\begin{cor}\label{Final}
Let $R$ be a Noetherian semi-local ring such that $R_M$ is a one-dimensional Cohen-Macaulay local ring for every $M \in \Max R$. Suppose that $R$ is an Arf ring. Then, for every finite subgroup $G$ of $\Aut R$ such that the order of $G$ is invertible in $R$, $R^G$ is an Arf ring. 
\end{cor}


\vspace{0.5em}

\begin{ac}
The authors would like to thank Joseph Lipman for valuable comments. 
\end{ac}



\end{document}